\documentclass[10pt]{amsart}
\usepackage{amsmath,epsf}
\usepackage{xr}
\usepackage{amssymb,enumitem}

\usepackage{tikz} \usetikzlibrary{matrix,arrows,decorations.pathmorphing}

\title{A twisted theorem of Chebotarev}
\date{\today}

\author{Ivan Toma{\v s}i{\'c}}
\address{Ivan Toma{\v s}i{\'c}\\
         School of Mathematical Sciences\\
  	Queen Mary University of London\\
         London, E1 4NS\\
        United Kingdom}
\email{i.tomasic@qmul.ac.uk}
\theoremstyle{plain}
\newtheorem{theorem}{Theorem}[section]
\newtheorem{corollary}[theorem]{Corollary}
\newtheorem{proposition}[theorem]{Proposition}
\newtheorem{lemma}[theorem]{Lemma}

\theoremstyle{definition}
\newtheorem{definition}[theorem]{Definition}
\newtheorem{example}[theorem]{Example}

\theoremstyle{remark}
\newtheorem{remark}[theorem]{Remark}

\newtheorem{notation}[theorem]{Notation}

\newtheorem{conjecture}[theorem]{Conjecture}

\providecommand{\OO}{\mathcal{O}}

\providecommand{\Z}{\mathbb{Z}}

\providecommand{\N}{\mathbb{N}}
\providecommand{\p}{\mathfrak{p}}

\providecommand{\ii}{\mathfrak{i}}
\providecommand{\cF}{\mathcal{F}}

\providecommand{\cA}{\mathcal{A}}

\providecommand{\cC}{\mathcal{C}}
\providecommand{\cD}{\mathcal{D}}
\providecommand{\C}{\mathbb{C}}
\providecommand{\LL}{\mathbb{L}}

\providecommand{\dl}{\mathop{\rm dl}\nolimits}
\providecommand{\dt}{\mathop{\rm dimtot}\nolimits}
\providecommand{\dte}{\mathop{\rm dimtoteff}\nolimits}
\providecommand{\sd}{\mathop{\sigma\text{\rm-dim}}\nolimits}
\providecommand{\trdeg}{\mathop{\rm tr.deg}\nolimits}
\providecommand{\dd}{{\mathbf{d}}}
\providecommand{\dde}{{\mathbf{d}_{\rm eff}}}
\providecommand{\kk}{\mathbf{k}}
\providecommand{\specd}{{\rm Spec}^\sigma}
\providecommand{\spec}{{\rm Spec}}

\providecommand{\F}{\mathbb{F}}

\providecommand{\Gal}{\text{\rm Gal}}

\providecommand{\Hom}{\text{\rm Hom}}
\providecommand{\diff}{\text{\it Diff}}

\providecommand{\aut}{Aut}

\def\mathrlap{\mathpalette\mathrlapinternal}
\def\mathrlapinternal#1#2{%
           \rlap{$\mathsurround=0pt#1{#2}$}}

\providecommand{\ztild}[1]{\rlap{$\smash{\tilde{\phantom{#1}}}$}\rlap{$\mathring{\phantom{#1}\kern1.1ex}$}#1\kern.1ex}

\providecommand{\lexp}[2]{{\vphantom{#2}}^{#1}{\kern-.1ex#2}}

\providecommand{\acirc}[1]{\phantom{a}\llap{$\scriptstyle#1$}
\kern.01ex\lower.75ex\hbox{$\smash{\mathring{}}$}}

\providecommand{\lzexp}[3]{{\vphantom{#2}}^{\lower0.0ex\hbox{\smash{$\acirc{#1}$}}}\kern-.1ex #2}

\providecommand{\lrexp}[3]{{\vphantom{#2}}^{#1}{\kern-.1ex#2^#3}}
\begin{document}

\begin{abstract} 
We prove a function-field version of Chebotarev's density theorem 
in the framework of difference algebraic geometry by developing
the notion of Galois coverings of generalised difference schemes, and
using Hrushovski's twisted Lang-Weil estimate.
\end{abstract}
\maketitle

\tableofcontents

\section{Introduction}

\noindent{\bf The main result.}
The classical function fields version of  Chebotarev's  theorem
states that the local Frobenius substitutions associated with a Galois covering
of algebraic varieties over a finite field are equidistributed with respect to
a suitably defined Dirichlet density. It can be proved using the Lang-Weil
estimate for the number of points of varieties over finite fields, together with
an `untwisting trick'. We encourage the interested reader to compare 
the original number-theoretic  theorem and the function field version by consulting 
\cite{fried-jarden}, and to find a beautiful unification in \cite{serre-L}.

We prove an analogue of the function fields version of Chebotarev's theorem
in \emph{difference algebraic geometry}. Suppose 
$p:(Z,\Sigma)\to (X,\sigma)$ is an \emph{\'etale Galois covering} of finite-dimensional difference
schemes over a finite field with a power of Frobenius. 
Intuitively, $\Sigma$ is a set of endomorphisms of $Z$ closed under
a binary operation reminiscent of conjugation and a finite group with operators 
$(G,\tilde{\Sigma})$ acts on $(Z,\Sigma)$ in a particular fashion so that
$p$ identifies $X$  with the quotient $Z/G$ and $\Sigma/G$ identifies with 
$\{\sigma\}$. Let $C$ be a 
conjugacy domain in $\Sigma$.
For a point $z\in Z(\bar{\F}_q,\varphi)$
with values in the algebraic closure of a finite field equipped with
a power of Frobenius $\varphi$, the \emph{Frobenius substitution} at $z$
is the element $\varphi_z\in\Sigma$ which matches the action of the Frobenius 
power $\varphi$ on $z$, i.e.,
$$
\varphi_z.z=z\varphi.
$$
For a point $x\in X(\bar{\F}_q,\varphi)$, the \emph{Frobenius substitution} at $x$
is the conjugacy class $\varphi_x\subseteq\Sigma$ of any $\varphi_z$ with
$p(z)=x$. The following is an informal restatement of Theorem~\ref{chebotarev-density}.
\begin{theorem}\label{main-ch-th}
The Dirichlet density of the set 
of $x\in X(\bar{\F}_q,\varphi)$ with varying $\varphi$, with the property
that $\varphi_x\subseteq C$ is equal to $|C|/|\Sigma|$.
\end{theorem}

\noindent{\bf Motivation and historical overview.}
There is a significant body of work related in one way or another to
counting solutions of difference
polynomial equations over  algebraic closures of finite fields equipped
with powers of the Frobenius automorphism. Firstly, since counting the
number of solutions of polynomial equations over finite fields is a special case,
it subsumes the amazing achievements of Grothendieck's circle around the
Weil conjectures and Deligne's proof of the Riemann hypothesis over finite fields.
Given a well-known translation mechanism between the languages of difference
equations and algebraic correspondences (as expounded in 
\cite[\ref{s:diff-pro}]{ive-tgs}), the work of Pink \cite{pink}, Fujiwara \cite{fujiwara} 
and Varshavsky \cite{varshavsky} on
Deligne's conjecture regarding the number of fixed points of correspondences  
twisted by powers of Frobenius is highly relevant. 
However, due to the strong properness assumptions 
these authors require to prove a very precise trace formula,  these
results cannot be applied to a general difference polynomial system.
Without the restrictive assumptions, Hrushovski produced an ingenious yet
very difficult proof  of a difference analogue of
the Lang-Weil estimate for the number of points on a difference scheme over
fields with Frobenii \cite{udi}. 

Inspired by these considerations,  we embarked on a progamme to
develop \emph{difference algebraic geometry} to the level where it reveals the
fine number-theoretic information regarding numbers of points of difference
schemes over fields with Frobenii. In the first instance, we aim to generalise
the techniques  developed by Fried, Jarden et al.\ \cite{fried-sacer},
\cite{fried-jarden}, \cite{FHJ} over finite fields, collectively known
under the name of \emph{Galois stratification}.
 The techniques  around the
theorem of Chebotarev developed in
 this paper are crucial for this work, and the development of \emph{twisted}
  Galois stratification is
 described in the follow-up paper \cite{ive-tgs}.

 One of the main obstacles
was that, apart from the pioneering work in \cite{udi} and \cite{laszlo},
which we quickly review in Section~\ref{strict-dga},
there are no other attempts of a systematic study of difference algebraic geometry,
so all prerequisites would have to be developed from first principles. 
Cohn's monograph \cite{cohn} and a recent book by Levin \cite{levin}
are sources of some of the difference algebra needed. 
We must emphasise that a typical difference scheme that arises when
studying difference polynomial equations  is of \emph{finite transformal type} over a difference field (\ref{finsigmatype}), 
but its ambient scheme is of infinite type 
and thus it falls just beyond the reach of  tools and methodology of the
classical algebraic geometry.

\noindent{\bf A need for generalised difference algebraic geometry.} 
The first key observation we made was that the context of \emph{ordinary} (or \emph{strict}\/)
difference schemes (with a single endomorphism) 
is too rigid and does not allow meaningful Galois actions, coverings or quotients.
Thus we are led into a study of \emph{generalised difference schemes,} 
endowed
with a set of (not necessarily commuting) endomorphisms, closed
under a binary operation of `conjugation'. 

Various authors made attempts to generalise the framework of ordinary difference
algebra and geometry. 
The treatment of
partial difference equations with respect to 
several commuting endomorphisms in \cite{levin} is too restrictive
for our requirements, because we are exactly interested in phenomena
arising in the case of \emph{non-commuting} endomorphisms.
The authors of \cite{acfa2} went in the direction of considering not
just fixed points of a single endomorphism $\sigma$ but also those of
its powers $\sigma^n$. A similar approach is  taken in \cite{wibmer},
where the author proves a Chevalley-type theorem on (near-)constructibility of
images of morphisms of difference schemes (our version is
\cite[\ref{chevalley}]{ive-tgs}). Intuitively, the object that results from a 
consideration of higher powers
$\sigma^n$ is closer to the ambient scheme and is therefore better-behaved. 

Our approach generalises both of these, and has a very interesting interaction with the latter, our framework being slightly more precise
when dealing with Galois actions.
 Section~\ref{gen-dga}
contains the extensive development
of difference
algebraic geometry of generalised difference schemes needed to formulate
the key notion of a \emph{Galois covering} of difference schemes in
Section~\ref{sect:galois}.

Section~\ref{ch:ttCh} contains the main results of the paper.  The
\emph{Dirichlet density} is introduced as an analytic density defined by means of 
\emph{zeta} and 
\emph{$L$-functions} associated
with \emph{constructible functions} on Galois coverings of difference schemes over
fields with Frobenii. This density statement \ref{main-ch-th} in its precise 
formulation \ref{chebotarev-density}  follows
from the trace formula \ref{trace-formula}, which is an approximative difference
avatar of the classical Lefschetz trace formula. It is proved using Hrushovski's 
twisted Lang-Weil estimate.

\noindent{\bf Applications.}
We consider our approach to generalised difference algebra a major advance
in its own right, and it should be of intrinsic interest
in the difference algebra community, 
with possible applications to Galois theory of difference
equations along the lines of \cite{wibmer}. 

Moreover, there are other naturally-occurring contexts where it may be
advantageous to study objects with several endomorphisms. In particular,
as Tom Scanlon pointed out, a key step in most approaches to Manin-Mumford conjecture (\cite{udi-MM}, for example) is to find a difference polynomial
equation that captures all the relevant torsion points, which is very hard to
achieve with a single endomorphism. On the other hand,
it is much easier to capture all of the torsion if one uses \emph{two} endomorphisms.

Apart from number theory and algebraic geometry, our results (especially \ref{trace-formula}) should be of interest
in model-theory and logic since, in conjunction with the description
of definable sets in terms of Galois stratifications from \cite{ive-tgs}, they reduce
the problem of counting points on  definable sets over fields with powers of Frobenius
to a calculation of (twisted) character sums.
Thus, the present paper can be viewed as a  conceptualisation of
the ideas of \cite{CDM} and \cite{ive-mark} for fields with Frobenii.
As already mentioned,  \ref{trace-formula} and its consequence 
\ref{twisted-cebotarev} are used in our subsequent work on Galois stratifications
in \cite{ive-tgs}.

\section{A formulary of difference schemes}\label{strict-dga}

Before embarking on a development of \emph{generalised difference schemes}
in the next section, we include a summary of known results (\cite{udi}, \cite{laszlo}) 
for difference
schemes in the \emph{strict sense} for the benefit of the reader.

A \emph{difference ring} is a ring $R$ together with a distinguished 
monomorphism $\sigma$. Given an element $a\in A$ we may also write $a^\sigma$ for $\sigma a$.
A difference ring homomorphism 
$f:(R,\sigma)\to(S,\tau)$ is a ring map making the following diagram
commutative  

Given a difference ring extension $(R,\sigma)\subseteq (S,\sigma)$, the difference subring of $S$
generated by a set $T\subseteq S$ over $R$ is denoted by $R\{T\}$ or $R[T]_\sigma$. 
Similarly, given a difference
field extension $(K,\sigma)\subseteq (L,\sigma)$, the difference subfield of $L$ generated by a
set $T\subseteq L$ over $K$ is denoted by $K\langle T\rangle$ or $K(T)_\sigma$.
\begin{definition}
Let $I$ be an ideal in a difference ring $(R,\sigma)$. We say that:
\begin{enumerate}
\item $I$ is a $\sigma$-ideal if $\sigma(I)\subseteq I$;
\item $I$ is \emph{well-mixed} if  $ab\in I$ implies $ab^\sigma\in I$;
\item $R$ itself is well-mixed if the zero ideal is;
\item $I$ is \emph{perfect} if $aa^\sigma\in I$ implies $a$ and $a^\sigma$ are both in $I$.
\end{enumerate}
\end{definition}

\begin{definition}
Given a difference ring $(R,\sigma)$, let 
$$
\spec^\sigma(R)=
\{\mathfrak{p}\in\spec(R):\sigma^{-1}(\mathfrak{p})=\mathfrak{p}\},
$$
as a locally ringed space, 
together with the topology induced by the Zariski topology of $\spec(R)$ and the
induced structure sheaf $\OO_{\spec^\sigma(R)}=\OO_{\spec(R)}\restriction\spec^\sigma(R)$.
\end{definition}
The following notation is useful when discussing the induced topology. For $f\in R$ and an ideal $I$ in $R$, we let $V^\sigma(I)=V(I)\cap\spec^\sigma(R)$ and $D^\sigma(f)=D(f)\cap\spec^\sigma(R)$.

Since the endomorphism of $\spec(R)$ induced by $\sigma$ gives a morphism
$\sigma^{-1}\OO_{\spec(R)}\to\spec(R)$ and $\sigma$ is the identity on $\OO_{\spec^\sigma(R)}$,
we obtain a sheaf morphism $\sigma:\OO_{\spec^\sigma(R)}\to\OO_{\spec^\sigma(R)}$.
It defines a morphism of locally ringed spaces since for $\p\in\spec^\sigma(R)$, the
corresponding morphism of stalks is just the morphism $R_\p\to R_\p$ induced by $\sigma$,
which is local. This also makes the residue field $\kk(\p)=R_\p/\p R_{\p}$ into a 
\emph{difference field.}

It is clear that a prime $\sigma$-ideal $\p$ is in $\specd(R)$ if and only if it is perfect. Every 
$\sigma$-ideal $I$ has a \emph{perfect closure} $\{I\}$. If
$I$ is well-mixed, its perfect closure is clearly given by
$$
\{I\}=\{a\in R: \exists\nu\in\N[\sigma], a^\nu\in I\}.
$$
Given a $\sigma$-ideal $I$, the set 
$V^\sigma(I)$
only depends on the
perfect closure $\{I\}$ and henceforth we adopt the notation $V\{I\}$ for it.
The closed sets in the topology on $\spec^\sigma(R)$  are the sets $V\{I\}$ and
we have the following.
\begin{proposition}
Let $(R,\sigma)$ be a difference ring.
\begin{enumerate}
\item $V\{I\}\subseteq V\{J\}$ if and only if $\{J\}\subseteq\{I\}$. If $R$ is well-mixed, an
element $r\in R$ defines a zero section in $\bar{R}$ if and only if $r$ is $\sigma$-nilpotent.
\item If $I$ is perfect, $V\{I\}$ is irreducible if and only if $I\in\spec^\sigma(R)$.
\item $\spec^\sigma(R)$ is a quasi-compact topological space.
\end{enumerate}
\end{proposition}

\begin{proposition}\label{wmaffi}
Suppose $(R,\sigma)$ is well-mixed. 
\begin{enumerate}
\item The canonical morphism
$$
R\to \bar{R}:=H^0(\spec^\sigma(R),\OO_{\spec^\sigma(R)})
$$
is injective. In particular, if $R$ is nontrivial, $\spec^\sigma(R)$ is non-empty.
\item For $f\in R$, the map $R[1/f]_\sigma\to H^0(D^\sigma(r),\OO_{\spec^\sigma(R)})$
is injective.
\item The induced morphism $\spec^\sigma(\bar{R})\to\spec^\sigma(R)$ is
an isomorphism.
\end{enumerate}
\end{proposition}

In every difference ring $(R,\sigma)$ there exists a smallest well-mixed ideal $0_w$
and thus we have the largest well-mixed quotient $R_w=R/0_w$, with the universal
property that every morphism from $(R,\sigma)$ to a well-mixed $(S,\sigma)$ factors through
$R_w$. The closed
immersion $\spec(R_w)\hookrightarrow\spec(R)$ induces a homeomorphism
$\spec^\sigma(R_w)\stackrel{\sim}{\to}\spec^\sigma(R)$.
In view of this discussion and the merits of \ref{wmaffi}, we shall not
hesitate to assume well-mixedness when necessary.

\begin{definition}
\begin{enumerate}
\item An \emph{affine difference scheme} $(X,\OO_X,\sigma)$ consists of a locally ringed space
$(X,\OO_X)$ with a morphism $\sigma:\OO_X\to\OO_X$, which is  isomorphic
to $\spec^\sigma(R,\sigma)$ for some well-mixed difference ring $(R,\sigma)$.

\item A \emph{difference scheme} $(X,\OO_X,\sigma)$ is a locally ringed space which
is locally isomorphic to an affine difference scheme.

\item A \emph{morphism of difference schemes} 
$f:(X,\OO_X,\sigma_X)\to(Y,\OO_Y,\sigma_Y)$ is a morphism of locally ringed spaces
which respects the difference structure, $\sigma_Y\circ f=f\circ\sigma_X$.
\end{enumerate}
\end{definition}

For a point $x$ on a difference scheme $X$, we denote by $\OO_x$ the local
(difference) ring at $x$, and by $\kk(x)$ the residue (difference) field at $x$.

The following is an important consequence of \ref{wmaff}. 
\begin{proposition}\label{strembedcat}
The `global sections' functor $H^0$ is left adjoint to the contravariant functor 
$\spec^\sigma$ from the category of well-mixed difference rings
to the category of difference schemes. For any difference scheme $(X,\sigma)$ and
any well-mixed difference ring $(R,\sigma)$,
$$
\Hom(X,\spec^\sigma(R))\stackrel{\sim}{\to}\Hom(R,H^0(X)).
$$
\end{proposition}
For an overly enthusiastic reader, it is worth remarking that, unlike in the algebraic case, $\spec^\sigma$ and $H^0$ do not
determine an equivalence of categories of difference rings and affine difference schemes, see \ref{rem-embedcat}.
 Moreover,  the global sections functor $H^0$ on the category of
 quasi-coherent sheaves of $(\OO_X,\sigma)$-modules on an affine difference scheme $X$ is not necessarily exact.

\section{Generalised difference schemes}\label{gen-dga}

In this section we wish to broaden the class of difference schemes in order to allow certain
Galois actions. We shall henceforth refer to difference schemes with a single endomorphism
as discussed above as \emph{strict difference schemes}, or difference schemes in the strict
sense, and we shall expand the term `difference scheme' to include objects with multiple
endomorphisms.

\subsection{Difference structure}

\begin{definition}
Let us consider the category $\diff$ as follows. An object  of $\diff$ is a
 (finite) set  $\Sigma$, %
 equipped with a map $\Sigma\times\Sigma\to\Sigma$, 
 $(\sigma,\tau)\mapsto \sigma^{\tau}$ such that $\sigma^\sigma=\sigma$.

A morphism %
$()^\varphi:\Sigma \to T$ is a function such that
for all $\sigma,\tau\in\Sigma$,
$$
(\sigma^{\tau})^\varphi=(\sigma^\varphi)^{({\tau}^\varphi)}.
$$
\end{definition}

\begin{definition}\label{diffcat}
Let $\cC$ be a category. The \emph{difference category} over $\cC$,
denoted $\diff(\cC)$, is defined as follows. Its objects are of form $(X,\Sigma)$,
where $X$ is an object of $\cC$, and $\Sigma$ is a set of $\cC$-endomorphisms of $X$
such that there exists a function $\Sigma\times\Sigma\to\Sigma$, 
$(\sigma,\tau)\mapsto \sigma^{\tau}$ such that:
\begin{enumerate}
\item for every $\sigma,\tau\in\Sigma$, $$\sigma^\tau\circ\tau=\tau\circ\sigma;$$
\item\label{dvaa} $\left(\Sigma,(\cdot)^{(\cdot)}\right)$ is an object of $\diff$;
\item\label{trii} for every $\sigma\in\Sigma$, $()^\sigma:\Sigma\to\Sigma$ is a $\diff$-morphism.
\end{enumerate}

A morphism $(\varphi,()^\varphi): (X,\Sigma)\to(Y,T)$ consists of a $\diff$-morphism
$()^\varphi:\Sigma\to T$ and a $\cC$-morphism $\varphi:X\to Y$ such that for
every $\sigma\in\Sigma$,
$$
\sigma^\varphi\circ\varphi=\varphi\circ\sigma.
$$

For the \emph{strong difference category} over $\cC$, we require that, additionally, 
for every object $(X,\Sigma)$, all endomorphisms in $\Sigma$ must be $\cC$-epimorphisms
of $X$.
\end{definition}

\begin{definition}\label{opdifcat}
Let $\cA$ be a category. 
The (strong) \emph{dual difference category over} $\cA$, is the opposite category of the 
(strong) difference category of the opposite category of $\cA$, 
$\diff(\cA^{op})^{op}$. To avoid misunderstandings, let us specify the details in the strong case.
 
 The objects are of form $(A,\Sigma)$,
where $A$ is an object of $\cA$, and $\Sigma$ is a 
set of $\cA$-monomorphisms $A\to A$
such that for every $\sigma, \tau\in\Sigma$,
there exists a (necessarily unique) $\sigma^{\tau}\in\Sigma$ such that
$$
\tau\circ\sigma^{\tau}=\sigma\circ\tau.
$$
It also follows that $\sigma^\sigma=\sigma$ for every $\sigma\in\Sigma$.

A morphism $\varphi:(B,T)\to (A,\Sigma)$ consists of an $\cA$-morphism 
$\varphi:B\to A$ and a map $()^\varphi:\Sigma\to T$ such that 
$$\varphi\circ\sigma^\varphi=\sigma\circ\varphi.$$
Moreover, we require that
$$
(\tau^{\sigma})^\varphi=(\tau^\varphi)^{({\sigma}^\varphi)}.
$$
\end{definition}

\begin{remark}
\begin{enumerate}
\item In the definition of a \emph{strong} difference category, the conditions
\ref{dvaa} and \ref{trii} from \ref{diffcat} are superfluous by arguments analogous to 
the discussion for the dual case below.
\item 
If $(A,\Sigma)$ is an object of the strong dual difference category over $\cA$, 
then every $\sigma\in\Sigma$ automatically
defines an endomorphism of $(A,\Sigma)$. Indeed, the equality $(\tau^\sigma)^{(\rho^\sigma)}=(\tau^\rho)^\sigma$ for arbitrary $\rho,\tau\in\Sigma$ follows from the fact that all
the relevant arrows are monomorphisms so the following diagram can be completed 
in only one way.
$$
 \begin{tikzpicture}
[cross line/.style={preaction={draw=white, -,
line width=4pt}}]
\matrix(m)[matrix of math nodes, row sep=.05cm, column sep=.6cm, text height=1.2ex, text depth=0.25ex]
{
|(a)|{A}	&		& |(b)|{A}&			\\[.6cm]
&			|(A)|{A} &			& |(B)|{A}\\[.3cm]          
|(c)|{A}	&		& |(d)|{A} &			\\[.6cm]
&			|(C)|{A} &			& |(D)|{A}   \\};
\path[->,font=\scriptsize,>=to, thin,inner sep=1pt]
(a)edge node[pos=0.5,above]{$\rho^\sigma$}(b)
(b)edge node[pos=0.3, right%
]{$\tau^\sigma$}(d)
(a) edge[dashed] node[pos=0.5,left]{$(\tau^\sigma)^{(\rho^\sigma)}=(\tau^\rho)^\sigma$}(c)
(c)edge node[pos=0.25,above%
]{$\rho^\sigma$}(d)
(A)edge[cross line]  node[pos=0.25,above]{$\rho$}(B)
(B)edge[cross line]  node[pos=0.5,right]{$\tau$}(D)
(A)edge[cross line]  node[pos=0.62,right%
]{$\tau^\rho$}(C)
(C)edge[cross line]  node[pos=0.5,below]{$\rho$}(D)
(a)edge node[pos=0.5,above right]{$\sigma$}(A)
(c)edge node[pos=0.5,above right]{$\sigma$}(C)
(b)edge node[pos=0.5,above right]{$\sigma$}(B)
(d)edge node[pos=0.5,above right]{$\sigma$}(D);
\end{tikzpicture}
$$ 

\item The last requirement from \ref{opdifcat} shows that the composite of an arbitrary morphism $\varphi$ and the structure morphism $\sigma\in\Sigma$ in  $\diff(\cA^{op})^{op}$
is well defined, 
$\tau^{\sigma\circ\varphi}=(\tau^{\sigma})^\varphi=(\tau^\varphi)^{({\sigma}^\varphi)}$.
In case $\varphi$ is a monomorphism itself, the condition is superfluous by a diagram
similar to the one above.
\end{enumerate}
\end{remark}

\begin{remark}\label{fibrediff}
It is quite illuminating to view the construction of $\diff(\cC)$ in the language of
fibred categories. Let us consider the functor $H:\diff^{op}\to\text{\bf Cat}$, such that
$H(\Sigma)=\cC(\Sigma)$, the category of $\cC$-objects with distinguished
endomorphisms $\Sigma$, with $\Sigma$-equivariant $\cC$-morphisms.
To a $\diff$-morphism $\phi:\Sigma\to T$ we assign $H(\phi)=\phi^*:\cC(T)\to\cC(\Sigma)$,
which maps a $T$-object $Y$ to the $\Sigma$-object $\phi^*(Y)$ which is just $Y$
considered with the $\Sigma$-action where $\sigma$ acts as $\phi(\sigma)$.

The split fibration $\int H\to\diff$ associated with $H$ is exactly the functor
$\diff(\cC)\to\diff$ assigning to each $(X,\Sigma)$ its `structure' $\Sigma$. The canonical
arrows $\phi^*Y\to Y$ are cartesian and a $\diff(\cC)$-morphism $(X,\Sigma)\to (Y,T)$
is a pair consisting of a $\diff$-morphism $\phi:\Sigma\to T$ and a 
 $\cC(\Sigma)$-morphism $f:X\to\phi^*Y$.
 
 The dual fibration of $\int H\to\diff$ is the fibration $\int ()^{op}\circ H$ which in our
 language corresponds to $\diff(\cC^{op})$.
\end{remark}

\begin{definition}
\begin{enumerate}
\item
A difference object $(X,\Sigma)$ is called \emph{inversive}, if every $\sigma\in\Sigma$ is an automorphism of $X$.
\item A difference object $(X,\Sigma)$ is \emph{almost-strict}, 
  if there exists a finite subgroup $G$ of $\aut(X,\Sigma)$
 such that for all $\sigma,\sigma'\in\Sigma$, there exists a $g\in G$ such that $\sigma'=g\sigma$,
 i.e., $\Sigma\subseteq G\sigma$ for some $\sigma\in\Sigma$.
\end{enumerate}
\end{definition}

\begin{definition}\label{def-regular}
\begin{enumerate}
\item A $\diff$-structure $\Sigma$ is \emph{regular}, if for every $\sigma\in \Sigma$,
the map $()^\sigma:\Sigma\to\Sigma$ is bijective.
\item A regular $\diff$-structure $\Sigma$ is \emph{full}, if it is  
equipped with %
 \emph{generalised conjugation} in the sense that for any
$\tau,\tau'\in \Sigma$, there is an bijective assignment
 $\sigma\to\lrexp{\tau}{\sigma}{{\tau'}}$ on 
$\Sigma$ which has the property that $()^\sigma\circ()^{\tau'}=()^{{\tau}}\circ()^{\lrexp{\tau}{\sigma}{{\tau'}}}$. In a given $\Sigma$-object $X$, this element gets interpreted as
a morphism satisfying
$$\sigma\circ\tau'=\tau\circ\lrexp{\tau}{\sigma}{{\tau'}}.$$
Intuitively, in an inversive structure, $\lrexp{\tau}{\sigma}{{\tau'}}$ should be thought
of as ${\tau}^{-1}\sigma\tau'$.
\end{enumerate}
\end{definition}
As Shahn Majid pointed out, a regular $\diff$-object is usually called a \emph{quandle} in the literature.
\begin{remark}
\begin{enumerate}
\item If $(X,\Sigma)$ is inversive, then $\Sigma$ is regular.
\item If $(X,\Sigma)$ is almost-strict, then $\Sigma$ is full.
\end{enumerate}
\end{remark}

 \begin{remark}\label{FC}
 If $(X,\Sigma)$ is inversive and almost-strict, then the group of automorphisms
 $\langle\Sigma\rangle$ 
 of $X$ generated by $\Sigma$ is  finite-by-cyclic,
 i.~e., for any $\sigma\in\Sigma$, we have an exact sequence
$$
 \begin{tikzpicture} 
\matrix(m)[matrix of math nodes, row sep=0em, column sep=2em, text height=1.5ex, text depth=0.25ex]
 {  |(1)|{1} & |(2)|{G}& |(3)|{\langle\Sigma\rangle} & |(4)|{\langle\sigma\rangle}
 & |(5)|{1}\\}; 
\path[->,font=\scriptsize,>=to, thin]
(1) edge  (2) (2) edge (3)
(3) edge  (4)
(4) edge  (5);
\end{tikzpicture}
$$
Clearly, $G$ can be thought of as a group with an operator $()^\sigma$ and
a morphism $\varphi:(X,\Sigma)\to(Y,T)$ of inversive almost-strict objects
 gives rise to a homomorphism of groups with operators, or, 
equivalently, a commutative diagram
$$
 \begin{tikzpicture} 
\matrix(m)[matrix of math nodes, row sep=2em, column sep=2em, text height=1.5ex, text depth=0.25ex]
 {
 |(1)|{1} & |(2)|{G}& |(3)|{\langle\Sigma\rangle} & |(4)|{\langle\sigma\rangle}& |(5)|{1}\\
 |(l1)|{1} & |(l2)|{H}& |(l3)|{\langle T\rangle} & |(l4)|{\langle\sigma^\varphi\rangle}& |(l5)|{1}\\
 }; 
\path[->,font=\scriptsize,>=to, thin]
(1) edge  (2) (2) edge (3) (3) edge  (4) (4) edge  (5)
(l1) edge  (l2) (l2) edge (l3) (l3) edge  (l4) (l4) edge  (l5)
(2) edge  (l2) (3) edge (l3) (4) edge  (l4);
\end{tikzpicture}
$$
 \end{remark}

\begin{proposition}\label{definvring}
Every (strong) difference ring $(R,\Sigma)$ with $\Sigma$ finite and regular has an \emph{inversive closure} with the following
universal property. There is an embedding
$(R,\Sigma)\hookrightarrow(R^{inv},\Sigma^{inv})$ such that every morphism 
$(R,\Sigma)\to(S,T)$ to an inversive difference ring factors uniquely through 
$R^{inv}$.
\end{proposition} 
\begin{proof}
The proof of \cite[2.1.7]{levin} can be lifted to our framework in spite
of having to deal with non-commuting endomorphisms.
Let $\Sigma=\{\sigma_1,\ldots,\sigma_n\}$, and let $\bar{\sigma}:\Sigma\to\Sigma$
denote the bijection $()^\sigma$, for $\sigma\in\Sigma$.
We know how to take inversive closures of strict difference rings, so let
$R_1=(R,\sigma_1)^{\rm inv}$. We must show that $R_1$ can be endowed with
a $\Sigma$-structure. Indeed, if $\tau\in\Sigma\setminus\{\sigma_1\}$, and
$a\in R_1$, then there is a positive $r$ such that $(\sigma_1^{\rm inv})^r a\in R$. 
We define 
$$
\tau(a)=(\sigma_1^{\rm inv})^{-r}\left(\bar{\sigma}_1^{-r}(\tau)\right)(\sigma_1^{\rm inv})^r a,
$$
and it can be checked that this is independent of the choice of $r$.
We continue inductively, by letting $R_{i+1}=(R_i,\sigma_{i+1})^{\rm inv}$ and endowing
it with $\Sigma$-structure for $i\geq 1$. It is clear that $(R_n,\Sigma^{\rm inv})$ is the inversive closure of $(R,\Sigma)$.
\end{proof}
Clearly, for every $a\in R^{inv}$, there exists a $\nu$ in the free semigroup generated by $\Sigma$  such that $\nu(a)\in R$ (see \ref{diff-ops}). 

As a consequence, every difference scheme $(X,\Sigma)$ can be dominated by an inversive difference scheme
$(X^{inv},\Sigma^{inv})\to (X,\Sigma)$ with the dual universal property, where the morphism is bijective at the level of points.

 \subsection{Difference spectra}

\begin{definition}\label{defspecsigma}
Let $(R,\Sigma)$ be an object of a difference category over the category of
commutative rings with identity.
We shall consider each of the following subsets of $\spec(R)$ as locally ringed spaces
with the Zariski topology and the structure sheaves induced from $\spec(R)$:
\begin{enumerate}
\item $\spec^\sigma(R)=\{\p\in\spec(R):\sigma^{-1}(\p)=\p\}$, for any $\sigma\in\Sigma$;
\item $\spec^\Sigma(R)=\cup_{\sigma\in\Sigma}\spec^\sigma(R)$;
\item $\spec^{(\Sigma)}(R)=\cap_{\sigma\in\Sigma}\spec^\sigma(R)$.
\end{enumerate}
 In discussions of induced topology, we shall use the notation
 $V^\sigma(I)$, $D^\sigma(I)$, $V^\Sigma(I)$, $D^\Sigma(I)$, $V^{(\Sigma)}(I)$, $D^{(\Sigma)}(I)$
for the traces of $V(I)$ and $D(I)$ on $\spec^\sigma(R)$, $\spec^\Sigma(R)$, $\spec^{(\Sigma)}(R)$, respectively.
 \end{definition}
 To be more precise, let us denote $X=\spec(R)$, $X^\Sigma=\spec^\Sigma(R)$, 
 $i:X^\Sigma\hookrightarrow X$ and we define the sheaf $\OO_{X^\Sigma}=i^{-1}\OO_X$. 
 Using the adjunction $1\to i_*i^{-1}$, we get a morphism $\OO_X\to i_*\OO_{X^\Sigma}$, i.e., 
 for every $U$ open in $X$, we get a map 
 $$
 \OO_X(U)\to \OO_{X^\Sigma}(U\cap X^{\Sigma}).
 $$
 Moreover, for each $x\in X^\Sigma$, the stalks $\OO_{X,x}$ and $\OO_{X^\Sigma,x}$
 are both isomorphic to $R_{\ii_x}$.
 
 For each 
 $\sigma\in\Sigma$ we have the morphism of locally ringed spaces
  $({\lexp{a}{\sigma}},\tilde{\sigma}):(X,\OO_X)\to(X,\OO_X)$  induced by $\sigma$, where
  $\tilde{\sigma}:\OO_X\to{\lexp{a}{\sigma}}_*\OO_X$. 
  For an arbitrary $\tau\in\Sigma$, ${\lexp{a}{\sigma}}(X^\tau)\subseteq X^{\tau^\sigma}$,
  so we have ${\lzexp{a}{\sigma}{0}}={\lexp{a}{\sigma}}\restriction X^\Sigma:X^\Sigma\to X^\Sigma$ and
 we get a diagram
 $$
 \begin{tikzpicture} 
\matrix(m)[matrix of math nodes, row sep=2.6em, column sep=2.8em, text height=1.5ex, text depth=0.25ex]
 {X^\Sigma& X\\
   X^\Sigma & X\\}; 
\path[->,font=\scriptsize,>=to, thin]
(m-2-1) edge node[auto] {$i$} (m-2-2) edge node[left] {${\lzexp{a}{\sigma}{0}}$} (m-1-1)
(m-2-2) edge node[right] {${\lexp{a}{\sigma}}$} (m-1-2)
(m-1-1) edge node[auto] {$i$} (m-1-2);
\end{tikzpicture}
$$
 The `restriction' of the sheaf morphism $\tilde{\sigma}$ to $X^\Sigma$ is the composite 
$$
\ztild{\sigma}:\OO_{X^\Sigma}=i^{-1}\OO_X\stackrel{i^{-1}\tilde{\sigma}}{\longrightarrow}i^{-1}\,{\lexp{a}{\sigma}}_*\OO_X
\stackrel{\text{\rm BC}}{\longrightarrow}{\lzexp{a}{\sigma}{0}}_*i^{-1}\OO_X=
{\lzexp{a}{\sigma}{0}}_*\OO_{X^\Sigma},
$$
 where the base change morphism (BC) is derived from the adjunction maps
 $1\to i_*i^{-1}$ and $i^{-1}i_*\to 1$ as the composite
 $$
 i^{-1}\,{\lexp{a}{\sigma}}_*\stackrel{\text{\rm adj}}{\longrightarrow} i^{-1}\,{\lexp{a}{\sigma}}_*i_*i^{-1}\stackrel{\sim}{\longrightarrow}
 i^{-1}i_*\,{\lzexp{a}{\sigma}{0}}_*i^{-1}\stackrel{\text{\rm adj}}{\longrightarrow}{\lzexp{a}{\sigma}{0}}_*i^{-1}.
 $$
 In order to avoid cumbersome notation, we may write $(\lexp{a}{\sigma}, \tilde{\sigma})$
 for $(\lzexp{a}{\sigma}{0}, \ztild{\sigma})$, when considered on 
 $(X^\Sigma,\OO_{X^\Sigma})$. 
 
 Thus, the locally ringed space $(X^\Sigma,\OO_{X^\Sigma})$ is equipped with
 a set of endomorphisms of locally ringed spaces 
 $$\lexp{a}{\Sigma}=\{({\lexp{a}{\sigma}},{\tilde{\sigma}}):\sigma\in\Sigma\},$$
 which, by functoriality (as in the case of general morphisms below), happens to be closed under `conjugation', i.e., has a $\diff$-structure.
 In other words, $(X^\Sigma,{\lexp{a}{\Sigma}})$ is an object of the difference category over the
 opposite category of locally ringed spaces.
 
 Suppose we are given a morphism $\varphi:(S,T)\to(R,\Sigma)$ in the difference category over
 the category of commutative rings with identity. Let us write $X=\spec(R)$, 
 $X^\Sigma=\spec^\Sigma(R)$, $i:X^\Sigma\hookrightarrow X$, $Y=\spec(S)$, $Y^T=\spec^T(S)$,
 $j:Y^T\hookrightarrow Y$ and we let $(\lexp{a}{\varphi},\tilde{\varphi})$ be the
 morphism of locally ringed spaces $(X,\OO_X)\to(Y,\OO_Y)$ induced by $\varphi$,
 where $\tilde{\varphi}:\OO_Y\to\lexp{a}{\varphi}_*\OO_X$. 
 One easily verifies that for any $\sigma\in\Sigma$, 
 $\lexp{a}{\varphi}(X^\sigma)\subseteq Y^{\sigma^\varphi}$, so 
 $\lzexp{a}{\varphi}{0}=\lexp{a}{\varphi}\restriction X^\Sigma:X^\Sigma\to Y^T$,
 and we get a diagram
 $$
 \begin{tikzpicture} 
\matrix(m)[matrix of math nodes, row sep=2.6em, column sep=2.8em, text height=1.5ex, text depth=0.25ex]
 {X & Y\\
   X^\Sigma & Y^T\\}; 
\path[->,font=\scriptsize,>=to, thin]
(m-2-1) edge node[auto] {$i$} (m-1-1) edge node[below] {${\lzexp{a}{\varphi}{0}}$} (m-2-2)
(m-2-2) edge node[right] {$j$} (m-1-2)
(m-1-1) edge node[auto] {$\lexp{a}{\varphi}$} (m-1-2);
\end{tikzpicture}
$$
The `restriction' of the sheaf morphism $\tilde{\varphi}$ to $Y^T$ is the composite 
$$
{\ztild{\varphi}}:\OO_{Y^T}=j^{-1}\OO_Y\stackrel{j^{-1}\tilde{\varphi}}{\longrightarrow}j^{-1}\,{\lexp{a}{\varphi}}_*\OO_X
\stackrel{\text{\rm BC}}{\longrightarrow}{\lzexp{a}{\varphi}{0}}_*i^{-1}\OO_X={\lzexp{a}{\varphi}{0}}_*\OO_{X^\Sigma},
$$
 where the base change morphism (BC) is derived from the adjunction maps
 $1\to i_*i^{-1}$ and $j^{-1}j_*\to 1$ as the composite
 $$
 j^{-1}\,{\lexp{a}{\varphi}}_*\stackrel{\text{\rm adj}}{\longrightarrow} j^{-1}\,{\lexp{a}{\varphi}}_*i_*i^{-1}\stackrel{\sim}{\longrightarrow}
 j^{-1}j_*\,{\lzexp{a}{\varphi}{0}}_*i^{-1}\stackrel{\text{\rm adj}}{\longrightarrow}{\lzexp{a}{\varphi}{0}}_*i^{-1}.
 $$

In order to show that $(\lzexp{a}{\varphi}{0},\ztild{\varphi})$ defines a morphism
$(X^\Sigma,\lexp{a}\Sigma)\to(Y^T,\lexp{a}{T})$ in the difference category over
the opposite category of the locally ringed spaces, for each $\sigma\in\Sigma$ one 
must chase through the following diagram, where we write $\tau=\sigma^\varphi$ for
ease of notation.
$$
 \begin{tikzpicture}
[cross line/.style={preaction={draw=white, -,
line width=4pt}}]
\matrix(m)[matrix of math nodes, row sep=.6cm, column sep=.6cm, text height=1.5ex, text depth=0.25ex]
{		& |(u4)| {X}	&			& |(u3)| {Y}	\\   %
|(u1)|	{X^\Sigma} &		& |(u2)|{Y^T} 	&			\\
		& |(d4)| {X}	&			& |(d3)| {Y}	\\   %
|(d1)|	{X^\Sigma} &		& |(d2)|{Y^T} 	&			\\};
\path[->,font=\scriptsize,>=to, thin]
(d1) edge node[left]{$\lzexp{a}{\sigma}{0}$} (u1) edge node[below]{$\lzexp{a}{\varphi}{0}$} (d2) 
	edge node[auto]{$i$} (d4) 
(d4) edge node[pos=0.25,left]{$\lexp{a}{\sigma}$} (u4)  edge node[pos=0.25, below]{$\lexp{a}{\varphi}$} (d3)

(u1) edge [cross line] node[pos=.8, above]{$\lzexp{a}{\varphi}{0}$} (u2)  edge node[auto]{$i$} (u4)
	
(d2) edge [cross line] node[pos=0.75, right]{$\lzexp{a}{\tau}{0}$} (u2)  edge node[below]{$j$} (d3)
(u4) edge node[above]{$\lexp{a}{\varphi}$} (u3)	
(u2) edge node[pos=0.4,above]{$j$} (u3)	
(d3) edge node[right]{$\lexp{a}{\tau}$} (u3);
\end{tikzpicture}
$$ 
The crucial property that needs to be verified is that
$$
\lzexp{a}{\tau}{0}_*\ztild{\varphi}\circ{\ztild{\tau}}=
\lzexp{a}{\varphi}{0}_*\ztild{\sigma}\circ\ztild{\varphi}.
$$
This is unpleasant to verify directly due to a number of base change morphisms, but it is
straightforward to check the equivalent condition for the relevant morphisms obtained by
adjunction. Indeed, let us write 
$\ztild{\varphi}^{\sharp}:\lzexp{a}{\varphi}{0}^{-1}\OO_{Y^T}\to\OO_{X^\Sigma}$,
$\ztild{\sigma}^{\sharp}:\lzexp{a}{\sigma}{0}^{-1}\OO_{X^\Sigma}\to\OO_{X^\Sigma}$,
$\ztild{\tau}^{\sharp}:\lzexp{a}{\tau}{0}^{-1}\OO_{Y^T}\to\OO_{Y^T}$
for the morphisms corresponding to $\ztild{\varphi}$, $\ztild{\sigma}$, $\ztild{\tau}$ by adjointness. Up to some canonical identifications, 
$\ztild{\varphi}^{\sharp}=i^{-1}\tilde{\varphi}^\sharp$, 
$\ztild{\sigma}^{\sharp}=i^{-1}\tilde{\sigma}^\sharp$ and
 $\ztild{\tau}^{\sharp}=j^{-1}\tilde{\tau}^\sharp$.
The condition 
$$\ztild{\sigma}^{\sharp}\circ \lzexp{a}{\sigma}{0}^{-1}\ztild{\varphi}^{\sharp}=
\ztild{\varphi}^{\sharp}\circ \lzexp{a}{\varphi}{0}^{-1}\ztild{\tau}^{\sharp}$$
is now readily verified, using the fact that 
$$\tilde{\sigma}^{\sharp}\circ \lexp{a}{\sigma}^{-1}\tilde{\varphi}^{\sharp}=
\tilde{\varphi}^{\sharp}\circ \lexp{a}{\varphi}^{-1}\tilde{\tau}^{\sharp}.$$

To avoid the above discombobulation with the notation, we may choose to write
$(\lexp{a}{\varphi},\tilde{\varphi})$ in place of $(\lzexp{a}{\varphi}{0},\ztild{\varphi})$
when it is clear from the context that we are referring to a morphism 
$(X^\Sigma,\OO_{X^\Sigma},\lexp{a}{\Sigma})\to (Y^T,\OO_{Y^T},\lexp{a}{T})$
induced by $\varphi:(S,T)\to (R,\Sigma)$.
 
 \begin{remark}\label{specsigma}
 We would like to state now that the functor $\spec$ mapping an object $(R,\Sigma)$ to the object  $(\spec^\Sigma(R),\OO_{\spec^\Sigma(R)},\lexp{a}{\Sigma})$, and a morphism $(\varphi,()^\varphi):(S,T)\to(R,\Sigma)$ to the morphism
 $$(\lexp{a}{\varphi},\tilde{\varphi},()^\varphi):(\spec^\Sigma(R),\OO_{\spec^\Sigma(R)},\lexp{a}{\Sigma})\to(\spec^T(S),\OO_{\spec^T(S)},\lexp{a}{T})$$ is a `contravariant' functor 
 from the strong dual difference category of
 commutative rings with identity to the difference category of locally ringed spaces which
 respects the difference structure. This statement can be made precise in the language
 of fibered categories. Using the terminology from \ref{fibrediff}, let $\cD$ denote the
 strong dual difference category over commutative rings with identity, and let $\mathcal{G}$
 denote the difference category of locally ringed spaces. We have fibrations
 $F:\cD^{op}\to\diff$ and $E:\mathcal{G}\to\diff$. Then $\spec$ defines a Cartesian functor
$F^{op}\to E$ from the opposite fibration of $F$ to $E$.

When all $\spec^\Sigma(R)$ are quasi-compact, (e.g.\ when $\Sigma$
is finite), the target difference category is also strong.
 \end{remark}

 \subsection{Structure sheaf in the well-mixed case}
 
\begin{definition}\label{diff-ops}
Given a difference structure $\Sigma$, we shall write $\langle\Sigma\rangle$
for the semigroup generated by $\Sigma$, which is the free semigroup 
generated by $\Sigma$ modulo the relations $\tau\sigma^\tau=\sigma\tau$
for $\sigma,\tau\in\Sigma$. We
 let the \emph{difference operators} rig 
$\N[\Sigma]$ be the set of $\N$-linear combinations of the elements of 
$\langle\Sigma\rangle$.
If $(R,\Sigma)$ is a difference ring, every $\nu\in\N[\Sigma]$ can be
thought of as a difference operator on $R$ via its natural 
action $a\mapsto a^\nu$.
We have:
\begin{enumerate}
\item $a^\sigma=\sigma(a)$ for $\sigma\in\Sigma$;
\item $a^{\nu_1+\nu_2}=a^{\nu_1}a^{\nu_2}$;
\item $(ab)^\nu=a^\nu b^\nu$;
\item if $\nu_1,\nu_2\in\langle\Sigma\rangle$, $\left(a^{\nu_1}\right)^{\nu_2}=a^{\nu_2\nu_1}$.
\end{enumerate}
For an element $\nu\in\N[\Sigma]$ given by $\nu=\sum_i n_i\tau_i$ with $\tau_i\in\langle\Sigma\rangle$, it will be convenient to write $|\nu|=\sum_i n_i$.
\end{definition}

Given a difference ring extension $(R,\Sigma)\subseteq (S,\Sigma)$, the difference subring of $S$
generated by a set $T\subseteq S$ over $R$ is denoted by $R[T]_\Sigma$, and it equals
$R[\N[\Sigma]T]$.
Similarly, given a difference
field extension $(K,\Sigma)\subseteq (L,\Sigma)$, the difference subfield of $L$ generated by a
set $T\subseteq L$ over $K$ is denoted by $K(T)_\Sigma$ and it equals $K(\N[\Sigma]T)$.
\begin{definition}
Let $I$ be an ideal in a difference ring $(R,\Sigma)$. We say that:
\begin{enumerate}
\item $I$ is a \emph{$\Sigma$-ideal} if $\sigma(I)\subseteq I$ for every $\sigma\in\Sigma$;
\item $I$ is \emph{$\Sigma$-reflexive} if $\sigma^{-1}(I)=I$ for every $\sigma\in\Sigma$;
\item $I$ is \emph{$\Sigma$-well-mixed} if  $ab\in I$ implies $ab^\sigma\in I$ for any $\sigma\in\Sigma$;
\item $R$ itself is well-mixed if the zero ideal is;
\item $I$ is \emph{$\Sigma$-perfect} if for every $\sigma\in\Sigma$, $aa^\sigma\in I$ implies $a$ and $a^\sigma$ are both in $I$.
\end{enumerate}
\end{definition}

For a set $T$, we denote by $\{T\}_\Sigma$  the least $\Sigma$-perfect ideal
containing $T$, for which we have the following construction. For a subset $S$ of $(R,\Sigma)$, let 
$S'=\{f\in R: f^\nu\in S\mbox{ for some }\nu\in\N[\Sigma]\}$. Starting with $T_0=T$, let
$T_{n+1}=[T_n]'_\Sigma$. It is clear that $\{T\}_\Sigma=\cup_nT_n$.

\begin{remark}\label{perfwm}
\begin{enumerate}
\item
Clearly, $I$ is $\Sigma$-perfect if and only if $a\in I$ whenever $a^\nu\in I$ for $\nu\in\N[\Sigma]\setminus\{0\}$,
and the notion of a perfect ideal is a natural generalisation of the notion of a radical ideal in the
difference context. 
\item
If $I$ is $\Sigma$-perfect, then $I$ is $\Sigma$-well-mixed. Indeed, if $ab\in I$ and 
$\sigma\in\Sigma$, then $(ba)^\sigma\in I$ so $a(ba)^\sigma b^{\sigma^2}=(ab^\sigma)(ab^\sigma)^\sigma\in I$ and thus $ab^\sigma\in I$.
\item 
If $I$ is $\Sigma$-well-mixed, clearly $\{I\}_\Sigma=\{a\in R: a^\nu\in I\mbox{ for some }\nu\in\N[\Sigma]\}$.
\end{enumerate}
\end{remark}

\begin{lemma}\label{perf1}
For $n\geq 0$, $S_nT_n\subseteq (ST)_{n+1}$. Consequently, $\{S\}_\Sigma\{T\}_\Sigma\subseteq\{ST\}_\Sigma$.
\end{lemma}
\begin{proof}
Let $P$ and $Q$ be $\Sigma$-invariant sets (such as $S_1$, $T_1$).
For $p\in P_1$, $q\in Q_1$, there exist $\nu_1,\nu_2\in\N[\Sigma]$ such that
$p^{\nu_1}\in [P]_\Sigma$ and $q^{\nu_2}\in[Q]_\Sigma$. By invariance of $P$ and $Q$,
we conclude that $p^{\nu_1}q^{\nu_2}\in[PQ]_\Sigma$. If we take some $\nu\in\N[\Sigma]$
such that $p^{\nu_1}q^{\nu_2}$ divides $(pq)^\nu$ ($\nu=\nu_1+\nu_2$, for example), we will have that $(pq)^\nu\in[PQ]_\Sigma$ and thus $pq\in [PQ]'$. 

Let $S^\Sigma=\N[\Sigma]S$ and $T^\Sigma=\N[\Sigma]T$ be the $\Sigma$-invariant
closures of $S$ and $T$. We can write arbitrary $\tilde{s}\in S^\Sigma$, $\tilde{t}\in T^\Sigma$
as $\tilde{s}=s^{\nu_1}$, $\tilde{t}=t^{\nu_2}$ for some $s\in S$, $t\in T$, 
$\nu_1,\nu_2\in\N[\Sigma]$. In order to show that there exist $\nu,\mu\in\N[\Sigma]$ such that 
$(st)^\mu$ divides $(\tilde{s}\tilde{t})^\nu=(s^{\nu_1}t^{\nu_2})^\nu$, we may reduce to the
case where $\nu_1,\nu_2\in\langle\Sigma\rangle$. Then $\nu=\nu_1+\nu_2^{\nu_1}$ and $\mu=\nu_2\nu_1$ suffice:
$$
s^{\nu_1(\nu_1+\nu_2^{\nu_1})}t^{\nu_2(\nu_1+\nu_2^{\nu_1})}=
s^{\nu_1^2+\nu_2\nu_1}t^{\nu_2\nu_1+\nu_2\nu_2^{\nu_1}}=
(st)^{\nu_2\nu_1}s^{\nu_1^2}t^{\nu_2\nu_2^{\nu_1}}.
$$
Therefore, $S^\Sigma T^\Sigma\subseteq[ST]'_\Sigma=(ST)_1$ and from the property
established above for invariant sets, 
$$
S_nT_n\subseteq(S^\Sigma)_n(T^\Sigma)_n\subseteq(S^\Sigma T^\Sigma)_n\subseteq
(ST)_{n+1}.
$$ 
\end{proof}

\begin{lemma}\label{perf2}
Let $S$ and $T$ be subsets of a difference ring $(R,\Sigma)$.
\begin{enumerate}
\item\label{jen} $(ST)_n\subseteq S_n\cap T_n$ for $n\geq 1$;
\item\label{dva} $S_n\cap T_n\subseteq (ST)_{n+1}$.
\end{enumerate}
As a consequence, $\{S\}_\Sigma\cap\{T\}_\Sigma=\{ST\}_\Sigma$.
\end{lemma}
\begin{proof}
The inclusion (\ref{jen}) immediately follows from the fact that $[ST]_\sigma$ is contained in both $[S]_\Sigma$ and $[T]_\Sigma$.

For (\ref{dva}), if $a\in S_n\cap T_n$, then, using \ref{perf1} $a^2\in S_nT_n\subseteq(ST)_{n+1}$
so $a\in (ST)_{n+1}$.
\end{proof}

\begin{proposition}\label{perfintrs}
Let $(R,\Sigma)$ be a difference ring and let $I$ be a $\Sigma$-perfect ideal. Then
$$
I=\bigcap_{\p\in V^{(\Sigma)}(I)}\p.
$$
\end{proposition}

\begin{proof}
Suppose $I$ is a $\Sigma$-perfect difference ideal and $x\notin I$. It suffices to find
a $\p\in V^{(\Sigma)}(I)$ such that $x\notin\p$. By Zorn's lemma, there exists a maximal
$\Sigma$-perfect ideal $J\supseteq I$, $x\notin J$.

Suppose that $J$ is not prime, i.~e.\ there exist $a,b\in R$, $ab\in J$, $a\notin J$, $b\notin J$.
From maximality of $J$, $x\in\{J\cup\{a\}\}_\Sigma$ and  $x\in\{J\cup{b}\}_\Sigma$. But then
$$
x\in\{J\cup\{a\}\}_\Sigma\cap\{J\cup\{b\}\}_\Sigma=\{J\cup\{a\}\cdot J\cup\{b\}\}_\Sigma=\{J\}_\Sigma
=J,$$
which is a contradiction.

Thus $J$ is prime and for every $\sigma\in\Sigma$, $\sigma^{-1}J=J$
since $J$ is $\Sigma$-perfect.
\end{proof}

\begin{corollary}\label{wknulsatz}
$V^{(\Sigma)}(I)\subseteq V^{(\Sigma)}(J)$ if and only if $\{I\}_\Sigma\supseteq\{J\}_\Sigma$.
\end{corollary}

\begin{remark}
With some extra care, a similar string of results to the above can be established
even in the case where $\Sigma$ is just a set of monomorphisms on $R$, without
$(R,\Sigma)$ being a difference ring, but we decided to abandon this generality since
it did not contribute to the geometrical picture we are describing. The results are
as follows. If we denote by $S^\Sigma$ and $T^\Sigma$ the $\Sigma$-invariant closures of 
$S$ and $T$, from the proof of \ref{perf1} we obtain 
$\{S\}_\Sigma\{T\}_\Sigma\subseteq\{S^\Sigma T^\Sigma\}_\Sigma$.
The analogue of \ref{perf2} is $\{S\}_\Sigma\cap\{T\}_\Sigma=\{S^\Sigma T^\Sigma\}_\Sigma$,
and the statement of \ref{perfintrs} is the same. The proof of \ref{perfintrs} is modified
as follows. If $ab\in J$ and $J$ is perfect, then \ref{perfwm} implies that $J$ is well-mixed
so $a^{\nu_1}b^{\nu_2}\in J$ for any $\nu_1,\nu_2\in J$, and we have:
$$
x\in\{J\cup\{a\}\}_\Sigma\cap\{J\cup\{b\}\}_\Sigma=\{(J\cup\{a\})^\Sigma\cdot (J\cup\{b\})^\Sigma\}_\Sigma=\{J\}_\Sigma
=J,$$
and we obtain the same conclusion.
\end{remark}

\begin{corollary}\label{qcompact}
Let $(R,\Sigma)$ be a difference ring, and let $f\in R$. Then $D^{(\Sigma)}(f)$ is quasi-compact.
\end{corollary}
\begin{proof}
It is clear from \ref{wknulsatz} that $D^{(\Sigma)}(g_i)$, $i\in I$, cover $D^{(\Sigma)}(f)$
if and only if $f\in\{\{g_i:i\in I\}\}_\Sigma$, which holds
 if and only if $f\in\{\{g_i:i\in I_0\}\}_\Sigma$ for some
finite $I_0\subseteq I$.
\end{proof}
Clearly, if $(A,\Sigma)$ is a difference ring with $\Sigma$ \emph{finite}, and $f\in A$, \ref{qcompact}
allows us to conclude that $D^\Sigma(f)$ is quasi-compact.

\begin{definition}
Let $(A,\Sigma)$ be a difference ring, and let $f\in A$. The \emph{multiplicative system of generalised
powers} of $f$ is the set $S_{f_\Sigma}=\{f^\nu:\nu\in\N[\Sigma]\}$.
The 
\emph{saturated multiplicative system} associated with $f$ is the set 
$$
S_{\{f\}_\Sigma}=\{g\in A:\{g\}_\Sigma\supseteq\{f\}_\Sigma\}\supseteq S_{f_\Sigma}.
$$
In view of \ref{wknulsatz}, $g\in S_{\{f\}_\Sigma}$ if and only if 
$D^{(\Sigma)}(g)\supseteq D^{(\Sigma)}(f)$. 

Lemma~\ref{perf2} shows that
$S_{\{f\}_\Sigma}$ is a multiplicative set and we define the difference ring localisations
$A_{f_\Sigma}=S_{f_\Sigma}^{-1}A$ and $A_{\{f\}_\Sigma}=S_{\{f\}_\Sigma}^{-1}A$.
\end{definition}

\begin{proposition}\label{wmaff}
Let $(A,\Sigma)$ be a well-mixed difference ring (or even ring with a set of monomorphisms),
$f\in A$. 
\begin{enumerate}
\item\label{jennn}
Both canonical morphisms
$$
A_{f_\Sigma}\to A_{\{f\}_\Sigma}\stackrel{\theta}{\to}\OO_{\spec^{(\Sigma)}A}(D^{(\Sigma)}(f)),
$$
are injective.
\newcounter{enumTemp}
    \setcounter{enumTemp}{\theenumi}
\end{enumerate}
If moreover $D^\Sigma(f)$ is quasi-compact, we have the following.
\begin{enumerate}\setcounter{enumi}{\theenumTemp}
\item \label{dvaaa}
For each $\bar{s}\in\OO_{X^\Sigma}(D^\Sigma(f))$, there exist $g_1,\ldots,g_r\in A$ such that
$D^\Sigma(f)=\cup_i D^\Sigma(g_i)$ and there is a section $s\in\OO_X(\cup_i D(g_i))$
such that $\bar{s}(x)=s(x)$ for $x\in D^\Sigma(f)$.
\item\label{dvaipol}
Let $\bar{s}\in\OO_{X^\Sigma}(D^\Sigma(f))$. The ideal
 $\mathop{\rm Ann}(\bar{s})=\{g\in A:g\bar{s}=0\}$ is well-mixed.
\item\label{triii}
Suppose $\bar{s}\in\OO_{X^\Sigma}(D^\Sigma(f))$ and $\p\in D^\Sigma(f)$ such that
$\bar{s}(\p)=0$. Then there is a $g\notin\p$ such that 
$g\bar{s}=0$ (on $D^\Sigma(f)$).
\item\label{cetriii}
Let $\bar{s}\in\OO_{X^\Sigma}(D^\Sigma(f))$ and $\p\in D^\Sigma(f)$. There exist
$g\notin\p$ and $a\in A$ such that $g\bar{s}-a=0$.
\item\label{pettt}
Let $\bar{s}\in\OO_{X^\Sigma}(X^\Sigma)$ such that $\bar{s}\restriction D^{(\Sigma)}(f)=0$.
Then there exists a $\nu\in\N[\sigma]$ such that $f^\nu\bar{s}=0$ (on $X^\Sigma$).
\item\label{sesttt}
There exist canonical injections $A\stackrel{i}{\hookrightarrow}\bar{A}=
\OO_{X^\Sigma}(X^\Sigma)\hookrightarrow\OO_{X^{(\Sigma)}}(X^{(\Sigma)})$ inducing an isomorphism
of difference schemes $(\lexp{a}{i},\tilde{\imath}):\spec^\Sigma(\bar{A})\to\spec^\Sigma(A)$.
\end{enumerate}
\end{proposition}
\begin{proof}

\noindent(\ref{jennn}) Recall that, for an element $a/g\in A_{\{f\}_\Sigma}$,
and $\p\in D^{(\Sigma)}(f)\subseteq D^{(\Sigma)}(g)$, $\theta(a/g)(\p)=a/g\in A_\p$. 

If $\theta(a/g)=\theta(b/h)$, then for every $\p\in D^{(\Sigma)}(f)$, $a/g=b/h$ in
$A_\p$, i.~e., there exists an  $l_\p\notin\p$, $l_\p(h a-g b)=0$ in $A$. 
Let $\mathfrak{a}=\mathop{\rm Ann}(h a-g b)$. 
 Thus, for every $\p\in D^{(\Sigma)}(f)$,
$l_\p$ witnesses 
$\mathfrak{a}\not\subseteq\p$.

 In other words, $V^{(\Sigma)}(\mathfrak{a})\cap D^{(\Sigma)}(f)=\emptyset$ so
  $V^{(\Sigma)}(\mathfrak{a})\subseteq V^{(\Sigma)}(f)$. Note that the assumption that
  $A$ is $\Sigma$-well-mixed implies that
$\mathfrak{a}$ is a $\Sigma$-well-mixed ideal. By \ref{wknulsatz}, we conclude that
$f\in\{\mathfrak{a}\}_\Sigma$ and by \ref{perfwm}, there exists a $\nu\in\N[\Sigma]$ such
that $f^\nu\in\mathfrak{a}$. Thus, $f^\nu(h a-g b)=0$ so $a/g=b/h$ in
$A_{\{f\}_\Sigma}$. Had we started with $g=f^{\nu_1}$ and $h=f^{\nu_2}$, we would have finished
with an equality in $A_{f_\Sigma}$, so this not only shows that $\theta$ is injective but also that the composite
map is injective.

\noindent(\ref{dvaaa}) By definition, $\OO_{X^\Sigma}=j^{-1}\OO_X$, where
$j:X^\Sigma\hookrightarrow X$ is the inclusion. Unravelling the definition of the
inverse image sheaf in the case of an inclusion (\cite{EGAI}, \S0, 3.7.1) we get that
if $\bar{s}\in\OO_{X^\Sigma}(D^\Sigma(f))$, then for every $x\in D^\Sigma(f)$
there exists a $g_x\in A$ with $x\in D^\Sigma(g_x)\subseteq D^\Sigma(f)$ and a
section $s\in\OO_X(D(g_x))$ such that $\bar{s}(z)=s(z)$ for every $z\in D^\Sigma(g_x)$.
Using the assumption that $D^\Sigma(f)$ is quasi-compact, 
we find a finite subcovering $D^\Sigma(f)=\cup_i D^\Sigma(g_i)$
of $\{D^\Sigma(g_x):x\in D^\Sigma(f)\}$ and $\bar{s}\restriction D^\Sigma(g_i)$ is
represented by a section $s_i\in\OO_X(D(g_i))$ which glue together to 
form a section $s\in\OO_X(\cup_i D(g_i))$, as required.

\noindent(\ref{dvaipol}) 
Suppose $ab\in\mathop{\rm Ann}(\bar{s})$. With the notation of (\ref{dvaaa}), let $s_i=s\restriction D(g_i)=a_i/{g_i^{r_i}}$ for some $r_i$ and we have that for every $i$,
$\theta_{{g_i}_\sigma}(aba_i/{g_i^{r_i}})=0$ on $D^{(\Sigma)}(g_i)$. By (\ref{jennn}), there
exists a $\nu_i$ such that $g_i^{\nu_i}aba_i=0$. By well-mixedness, we get that 
$g_i^{|\nu_i|}a\sigma(b)a_i=0$ and since $g_i^{|\nu_i|}$ is invertible on $D(g_i)$,
we conclude that  $a\sigma(b)s_i=0$ on $D(g_i)$ (and therefore on $D^\Sigma(g_i)$)
for every $i$, so $a\sigma(b)\in\mathop{\rm Ann}(\bar{s})$.

\noindent(\ref{triii}) Assume that $\bar{s}$ is represented by an $s$ as in (\ref{dvaaa})
and that $\p$ falls in some $D^\Sigma(g_i)$. Then $s_i=s\restriction D(g_i)$ is
represented by $a_i/{g_i^{r_i}}$ for some $r_i$, and $a_i/{g_i^{r_i}}=0$ in $A_\p$,
so there exists a $h\notin\p$ such that $ha_i=0$ in $A$.
Alternatively, since $g_i$ is invertible on 
$D(g_i)$, we get that  that $hs_i=0$ on $D(g_i)$.

By replacing $\bar{s}$ by $h\bar{s}$, we may assume
$s_i=0$. For any $l$, on $D(g_i)\cap D(g_l)=D(g_ig_l)$, $s_l$ (represented by $a_l/{g_l^{r_l}}$)
coincides with $s_i$ (represented by $0/1$) so there exists an $n_l$ such that
$(g_ig_l)^{n_l}a_l=0$ in $A$. %
Since $g_l$ is invertible on $D(g_l)$, we get that $g_i^{n_l}s_l=0$ on $D(g_l)$.
Letting $n=\max\{n_l\}$, $g_i^ns=0$ on $\cup_l D(g_l)$ so
$g_i^n\bar{s}=0$ on $D^\Sigma(f)$.

\noindent(\ref{cetriii}) If $\bar{s}(\p)=a/g\in A_\p$, apply (\ref{triii}) to $g\bar{s}-a$.

\noindent(\ref{pettt}) Claim (\ref{triii}) provides, for each $\p\in D^{(\Sigma)}(f)$ a witness
to $\mathop{\rm Ann}(\bar{s})\not\subseteq\p$. Thus, 
$V^{(\Sigma)}(\mathop{\rm Ann}(\bar{s}))\cap D^{(\Sigma)}(f)=\emptyset$ so we conclude that 
$f\in\{\mathop{\rm Ann}(\bar{s})\}_\Sigma$, and, since (\ref{dvaipol}) guarantees that
 $\mathop{\rm Ann}(\bar{s})$ is well-mixed,
according to \ref{perfwm} there exists a $\nu\in\N[\sigma]$ with $f^\nu\in\mathop{\rm Ann}(\bar{s})$.

\noindent(\ref{sesttt}) To see that $i:A\rightarrow\bar{A}$ is injective, note that the composite
of the given canonical maps is injective by (\ref{jennn}) while the injectivity of the right map 
is guaranteed by (\ref{pettt}). Consider 
$\lexp{a}{i}:\spec^\Sigma(\bar{A})\to\spec^\Sigma(A)$, given by 
$\lexp{a}{i}(\bar{\p})=i^{-1}(\bar{\p})=\bar{\p}\cap A$. This map is surjective since for every 
$\sigma\in\Sigma$ and $\p\in\spec^\sigma(A)$, the ideal 
$\{\bar{s}\in\bar{A}:\bar{s}(\p)\in\p A_\p\}\in\spec^\sigma(\bar{A})$
 maps to $\p$. To see that
it is injective, fix a $\p\in\spec^\Sigma(A)$ and take some $\bar{\p}\in\spec^\Sigma(\bar{A})$
such that $\bar{\p}\cap A=\p$. Given $\bar{s}\in\bar{A}$, by (\ref{cetriii}), there exist
$g\notin\p$ and $a\in A$ such that $g\bar{s}-a=0$ on $\spec^\Sigma(A)$.
Thus $\bar{s}\in\bar{\p}$ if and only if $a\in\p$ so $\bar{\p}$ is uniquely determined by $\p$. 

To see that $i$ moreover defines an isomorphism of locally ringed spaces, it suffices to note
that if $\lexp{a}{i}(\bar{\p})=\p$, $\tilde{\imath}^\sharp_{\bar{\p}}:A_\p\to\bar{A}_{\bar{\p}}$
is an isomorphism.
\end{proof}

\begin{remark}
Using \ref{qcompact}, the conditional statements (\ref{dvaaa})--(\ref{sesttt}) become unconditional
if we replace every occurrence of $X^\Sigma$ and $D^\Sigma(f)$ by $X^{(\Sigma)}$ and 
$D^{(\Sigma)}(f)$, respectively.
\end{remark}

\begin{remark}\label{univgs}
A commutative diagram of well-mixed difference rings
$$
 \begin{tikzpicture} 
\matrix(m)[matrix of math nodes, row sep=2.6em, column sep=2.8em, text height=1.5ex, text depth=0.25ex]
 {\bar{A} & \bar{B}\\
   A &  B\\}; 
\path[->,font=\scriptsize,>=to, thin]
(m-2-1) edge node[auto] {$i$} (m-1-1) edge node[below] {$\alpha$} (m-2-2)
(m-2-2) edge node[right] {$j$} (m-1-2)
(m-1-1) edge node[above] {$\beta$} (m-2-2);
\end{tikzpicture}
$$
is uniquely completed to a square by $\bar{\alpha}:\bar{A}\to\bar{B}$, i.e.,
$\bar{\alpha}=j\circ\beta$. Indeed, this follows from
$$
 \begin{tikzpicture} 
\matrix(m)[matrix of math nodes, row sep=2.6em, column sep=2.8em, text height=1.5ex, text depth=0.25ex]
 {\bar{\bar{A}} & \\
 \bar{A} & \bar{B}\\
   A &  B\\}; 
\path[->,font=\scriptsize,>=to, thin]
(m-1-1) edge node[above]{$\bar{\beta}$} (m-2-2)
(m-3-1) edge node[auto] {$i$} (m-2-1) edge node[below] {$\alpha$} (m-3-2)
(m-3-2) edge node[right] {$j$} (m-2-2)
(m-2-1) edge node[above] {$\beta$} (m-3-2) edge node[above]{$\bar{\alpha}$} (m-2-2) edge node[left]{$\bar{\imath}$} (m-1-1);
\end{tikzpicture}
$$
and the fact that $\bar{\alpha}=\bar{\beta}\circ\bar{\imath}$ and $\bar{\beta}\circ\bar{\imath}=j\circ\beta$. Of course, $\bar{\imath}$ is an isomorphism.
\end{remark}

\begin{remark}\label{intrmed}
The proof of \ref{wmaff}(\ref{sesttt}) in fact shows that for any difference ring $(A_1,\Sigma)$ such that
$(A,\sigma)\hookrightarrow(A_1,\sigma)\hookrightarrow(\bar{A},\sigma)$, we obtain an isomorphism of difference
schemes $\spec^\Sigma(A_1)\to \spec^\Sigma(A)$. This observation will prove invaluable 
for proving certain finiteness properties later on.
\end{remark}

 \begin{proposition}\label{embedcataff}
 Let $(R,\Sigma)$ and $(S,T)$ be well-mixed difference rings whose spectra are quasi-compact and let 
 $\bar{R}=\OO_{\spec^\Sigma(R)}(\spec^\Sigma(R))=H^0(\spec^\Sigma(R))$ be the ring of global sections.
 There is a natural isomorphism
 $$
 \Hom(\spec^\Sigma(R),\spec^T(S))\stackrel{\sim}{\longrightarrow}
 \Hom((S,T),(\bar{R},\Sigma)),
 $$
 where the first $\Hom$ is in the difference category over the category of locally ringed spaces and the second is in the category of difference rings.
In other words, the functor $H^0$ is left adjoint to $\spec$.
\end{proposition}
\begin{proof}
Let $X=\spec^\Sigma(R)$, $Y=\spec^T(S)$,
$\bar{X}=\spec^\Sigma(\bar{R})$, $\bar{Y}=\spec^T(\bar{S})$, 
and write $i:R\hookrightarrow\bar{R}$ and $j:S\hookrightarrow\bar{S}$ for
the inclusions from \ref{wmaff}(\ref{sesttt}). We know by \ref{wmaff}(\ref{sesttt}) that
$(\lexp{a}{i},\tilde{\imath})$ is an isomorphism $\bar{X}\to X$ and that
$(\lexp{a}{j},\tilde{\jmath})$ is an isomorphism $\bar{Y}\to Y$.

We have already shown in \ref{specsigma} that for every $\varphi:(S,T)\to(\bar{R},\Sigma)$,
the pair $(\lexp{a}{\varphi},\tilde{\varphi})$ is a morphism 
$(\bar{X},\OO_{\bar{X}},\lexp{a}{\Sigma})\to(Y,\OO_Y,\lexp{a}{T})$, and we can produce 
a morphism $(X,\OO_{X},\lexp{a}{\Sigma})\to(Y,\OO_Y,\lexp{a}{T})$ as the composite
$(\lexp{a}{\varphi},\tilde{\varphi})\circ(\lexp{a}{i},\tilde{\imath})^{-1}$.

Suppose that we are given $(\psi,\theta):(X,\OO_X)\to(Y,\OO_Y)$ a morphism of locally
ringed spaces.
By definition, $\theta:\OO_X\to\psi_*\OO_Y$, so taking global sections 
we get $\bar{\varphi}:\bar{S}\to\bar{R}$ and
precomposing with the inclusion $i:S\to\bar{S}$ from \ref{wmaff}(\ref{jennn}) 
we obtain a morphism $\varphi:S\to\bar{R}$, $\varphi=\bar{\varphi}i$.

We claim that $(\lexp{a}{\varphi},\tilde{\varphi})=(\psi,\theta)\circ(\lexp{a}{i},\tilde{\imath})$. 
By assumption, for $x\in X$, $\theta^\sharp_x:\OO_{Y,\psi(x)}\to \OO_{X,x}$ is a local
homomorphism of local rings. Writing $x=\lexp{a}{i}(\bar{x})$, $y=\psi(x)=\lexp{a}{j}(\bar{y})$, 
the  diagram 
$$
\begin{tikzpicture} 
\matrix(m)[matrix of math nodes, row sep=2.0em, column sep=1.3em, text height=1.5ex, text depth=0.25ex]
 {S		& \bar{S}		&		& \bar{R} \\
 S_{y}& 			&R_{x} &\\
 		&\bar{S}_{\bar{y}}& 	&\bar{R}_{\bar{x}}\\}; 
\path[->,font=\scriptsize,>=to, thin]
(m-1-1) edge node[below] {$i$} (m-1-2) edge (m-2-1) edge [bend left=20] node[pos=0.6,above]{$\varphi$}(m-1-4)
(m-1-2) edge node[below] {$\bar{\varphi}$} (m-1-4) edge (m-3-2)
(m-2-1) edge node[pos=0.7,above] {$\theta^\sharp_x$} (m-2-3) edge node[left] {$\tilde{\jmath}^\sharp_{\bar{y}}$} (m-3-2) 
(m-1-2) edge  (m-2-1)
(m-1-4) edge  (m-2-3) edge (m-3-4)
(m-2-3) edge node[pos=0.35,right] {$\tilde{\imath}^\sharp_{\bar{x}}$} (m-3-4) 
(m-3-2) edge node[auto] {$\bar{\varphi}_{\bar{x}}$} (m-3-4)
(m-2-1) edge [dashed] node[above] {$\varphi_{\bar{x}}$} (m-3-4);
\end{tikzpicture}
$$
is commutative, where the top parallelogram follows from the definition of stalks and the discussion
after \ref{defspecsigma} which allows the identifications  $\OO_{Y,y}=S_y$ and $\OO_{X,x}=R_x$, and the vertical arrows are the localisation maps. It follows that 
$\varphi^{-1}(\ii_{\bar{x}})=\ii_{\psi(x)}$, 
i.e., $(\psi\circ\lexp{a}{i})(x)=\lexp{a}{\varphi}(x)$ and
that, writing $\varphi_{\bar{x}}$ for the morphism obtained by localisation of $\varphi$ at $\bar{x}$, 
$\tilde{\varphi}^\sharp_{\bar{x}}=\varphi_{\bar{x}}=\tilde{\imath}^\sharp_{\bar{x}}\circ\theta^\sharp_x$. Since a morphism is characterised by its action on stalks, we conclude that
$(\lexp{a}{\varphi},\tilde{\varphi})=(\psi,\theta)\circ(\lexp{a}{i},\tilde{\imath})$, as required.
\end{proof}

\begin{remark}\label{rem-embedcat}
It is worth remarking that, unlike in the algebraic case, $\spec$ and $H^0$ do not
determine an equivalence of categories, only a weaker notion which we might dub temporarily
`an embedding of categories' for the lack of a reference:
the unit of the adjunction $1\to \spec\circ H^0$ is a natural isomorphism, while
the counit $1\to H^0\circ\spec$ is only a natural injection by \ref{wmaff}(\ref{sesttt}).
\end{remark}

\subsection{Difference schemes}

\begin{definition}
\begin{enumerate}
\item  An \emph{affine difference scheme} 
is an object $(X,\OO_X,\Sigma)$ of the difference category over the category of locally ringed spaces, which is
isomorphic to
 some $\spec^\Sigma(R)$ for some well-mixed $(R,\Sigma)$.
 \item
A \emph{difference scheme} 
is an object $(X,\OO_X,\Sigma)$ of the difference category over the category of locally ringed spaces, which is locally an affine difference scheme.
\item 
A \emph{morphism of difference schemes} 
$(X,\OO_X,\Sigma)\to(Y,\OO_Y,T)$ is just a morphism in the difference category over
the category of locally ringed spaces.
\end{enumerate}
\end{definition}

\begin{remark}
Given a difference scheme $(X,\OO_X,\Sigma)$ and $\sigma\in\Sigma$, we 
define a locally ringed space
$X^\sigma=\{x\in X:\sigma(x)=x\}$, together with the topology and structure sheaf 
induced from $(X,\OO_X)$. Since $\OO_{X^\sigma}:=\OO_X\restriction X^\sigma$, clearly 
$\sigma^\sharp\OO_{X^\sigma}\subseteq\OO_{X^\sigma}$ and  $(X^\sigma,\OO_{X^\sigma},\sigma)$ is a strict difference
scheme. We have the following properties:
\begin{enumerate}
\item $X=\bigcup_{\sigma\in\Sigma} X^{\sigma}$.
\item For every $\sigma, \tau\in \Sigma$,
there is a unique element $\sigma^{\tau}\in \Sigma$ such that
$$
{\tau}:(X^\sigma,\OO_{X^\sigma},\sigma)\to (X^{\sigma^{\tau}}, \OO_{X^{\sigma^{\tau}}},\sigma^{\tau})
$$
is a morphism of difference schemes in the strict sense.
\item
If $\varphi:(X,\OO_X,\Sigma)\to (Y,\OO_Y,T)$ is a morphism, then for every $\sigma\in \Sigma$ there exists a $\tau:=\sigma^\varphi\in T$
such that
$$
\varphi\restriction X^{\sigma}:(X^\sigma,\OO_{X^\sigma},\sigma)\to (Y^\tau,\OO_{Y^\tau},\tau)
$$
is a morphism of difference schemes in the strict sense. 
\end{enumerate}
\end{remark}

\begin{remark}\label{fulfinorb}
If $(X,\Sigma)$ is a full difference scheme with $\Sigma$ finite, 
then the $\langle\Sigma\rangle$-orbit
of any $x\in X$ is finite of bounded length.  As a direct consequence, the
morphisms $\tau:X^\sigma\to X^{\sigma^\tau}$ discussed above are bijective
at the level of points.

Indeed, let $x\in X$, and
consider the set $O(x)=\{\sigma x:\sigma\in\Sigma\}$. Suppose $x\in X^{\tau_1}$
for some $\tau_1\in\Sigma$. For an arbitrary $\tau,\sigma\in\Sigma$, take a
 $\rho\in\Sigma$ with $\lrexp{\tau}{\rho}{{\tau_1}}=\sigma$.
 Then 
 $$
 \tau\sigma x=\rho\tau_1 x=\rho x,
 $$
 so we conclude that $\tau(O(x))\subseteq O(x)$.

This is consistent with our intuition that a full (or almost-strict) 
difference scheme is a finitary perturbation of a strict difference scheme.  
\end{remark}

\begin{proposition}\label{embedcat}
The `global sections' functor $H^0$ is left adjoint to the contravariant %
functor 
$\spec$ from the category of well-mixed difference rings
to the category of difference schemes. For any difference scheme $(X,\Sigma)$ and
any well-mixed difference ring $(S,T)$ (with $T$ finite),  %
$$
\Hom\left((X,\Sigma),(\spec^T(S),T)\right)\stackrel{\sim}{\to}\Hom\left((S,T),(H^0(X),\Sigma)\right).
$$
\end{proposition}

\begin{definition}
A \emph{quasi-coherent sheaf} $(\cF,\Sigma)$ on a difference scheme 
$(X,\OO_X,\Sigma)$ 
is a quasi-coherent sheaf $\cF$ on the locally ringed space $(X,\OO_X)$ so
that $(\cF,\Sigma)$ is an $(\OO_X,\Sigma)$-module in a natural sense.
A \emph{difference subscheme} of $(X,\OO_X,\Sigma)$ is the locally ringed space
equipped with $\Sigma$-structure associated with a quasi-coherent
sheaf $(\mathcal{I},\Sigma)$ of $(\OO_X,\Sigma)$-ideals.
\end{definition}

\begin{definition}
Let us give definitions of various versions of reduced schemes associated with
a difference scheme $(X,\OO_X,\Sigma)$. In each case, the underlying
topological space is that of $X$, and we specify the
defining sheaf of ideals by giving its stalks:
\begin{enumerate}
\item $X_{\text{\rm red}}$, the \emph{reduced subscheme} of $X$, defined by
$\mathcal{N}$, where $\mathcal{N}_x$ is the nilradical of $\OO_x$;
\item $X_{w}$, the \emph{well-mixed subscheme} of $X$, defined by
$\mathcal{N}_w$, where $\mathcal{N}_{w,x}$ is the least well-mixed ideal of $\OO_x$;
\item $X_{\Sigma\text{\rm-red}}$, the \emph{reflexively reduced subscheme} of $X$, defined by
$\mathcal{N}_{\Sigma\text{\rm-red}}$, where $\mathcal{N}_{\Sigma\text{\rm-red},x}$ is the 
reflexive closure of $0$ in $\OO_x$;
\item $X_{\{\Sigma\}\text{\rm-red}}$, the \emph{perfectly reduced subscheme} of $X$, defined by
$\mathcal{N}_{\{\Sigma\}\text{\rm-red}}$, where $\mathcal{N}_{\{\Sigma\}\text{\rm-red},x}$ is the 
perfect closure of $0$ in $\OO_x$.
\end{enumerate}
We shall say that $X$ is \emph{perfectly reduced} if $X=X_{\{\Sigma\}\text{\rm-red}}$,
and similarly for other notions above.
\end{definition}

\begin{definition}
Let $(X,\Sigma)$ be  a difference scheme.
\begin{enumerate}
\item
We say that $X$ is irreducible (resp.\ connected) if its underlying topological space is.
\item We say that $X$ is integral (resp.\ transformally integral) if it is irreducible and reduced (resp.\ perfectly reduced).
\end{enumerate}
\end{definition}

\subsection{Products and fibres}

\begin{definition}
\begin{enumerate}
\item Let $(S,\Sigma_0)$ be a difference scheme. The category of \emph{$(S,\Sigma_0)$-difference schemes}
has difference scheme morphisms $(X,\Sigma)\to (S,\Sigma_0)$ as objects (considered as 
structure maps), while a morphism between $(X,\Sigma)\to (S,\Sigma_0)$ and 
$(Y,T)\to (S,\Sigma_0)$ is a commutative diagram of difference scheme maps
$$
 \begin{tikzpicture} 
\matrix(m)[matrix of math nodes, row sep=0pt, column sep=1mm, text height=1.2ex, text depth=0.25ex]
 {
 |(11)|{(X,\Sigma)} & 		& |(12)|{(Y,T)}\\[.7cm]
			 & |(22)|{(S,\Sigma_0)} &	\\};
\path[->,font=\scriptsize,>=to, thin]
(11) edge (12) edge(22)
(12) edge (22);
\end{tikzpicture}
$$
\item Let $(R,\Sigma_0)$ be a difference ring. The category of 
\emph{$(R,\Sigma_0)$-difference schemes} consists of difference schemes which are locally of the form $\spec^\Sigma(A)$,
for a difference $(R,\Sigma_0)$-algebra $(A,\Sigma)$. Morphisms are required to locally
preserve the $(R,\Sigma_0)$-algebra structure.
\end{enumerate}
\end{definition}

\begin{definition}
\begin{enumerate}
\item Let $(X,\Sigma)$ be a difference scheme and $(K,\varphi)$ a difference field.
A \emph{$(K,\varphi)$-rational point} of $(X,\Sigma)$ is a morphism 
$x:\spec^\varphi(K)\to(X,\Sigma)$. When $(X,\Sigma)=\spec^\Sigma(R)$, this
means we have a point $\p\in\spec^{\Sigma}(R)$ and a local map 
$(R_\p,\varphi^x)\to(K,\varphi)$, where $\varphi^x$ is the image of $\varphi$ in
$\Sigma$ by the difference structure map $()^x:\{\varphi\}\to \Sigma$.
Alternatively, we have an inclusion $(\kk(\p),\varphi^x)\hookrightarrow (K,\varphi)$.

\item Let $(X,\Sigma)$ be a difference scheme over a difference field $(k,\sigma)$
and let $(k,\sigma)\subseteq(K,\varphi)$. The set of \emph{$(K,\varphi)$-rational points} of $(X,\Sigma)$, henceforth denoted by
$(X,\Sigma)(K,\varphi)$, is the set of all $(k,\sigma)$-morphisms $\spec^\varphi(K)\to(X,\Sigma)$.
\end{enumerate}
\end{definition}

Modulo a small technical condition, fibre products exist in the category of difference schemes. 
\begin{proposition}\label{fibreprod}
Let $(X_i,\Sigma_i)\to (S,\Sigma)$, $i=1,2$, be morphisms of difference schemes. 
If the difference structure map $\Sigma_1\to\Sigma$ is surjective, then the fibre product
$(X_1,\Sigma_1)\times_{(S,\Sigma)}(X_2,\Sigma_2)$ exists, i.e., the functor
$$
(Z,T)\mapsto \Hom_{(S,\Sigma)}((Z,T),(X_1,\Sigma_1))\times\Hom_{(S,\Sigma)}((Z,T),(X_2,\Sigma_2))$$
from the opposite category of $(S,\Sigma)$-schemes to sets is representable.
\end{proposition}
\begin{proof}
Indeed, suppose we have morphisms $(X_1,\Sigma_1)\to (S,\Sigma)$ and 
$(X_2,\Sigma_2)\to (S,\Sigma)$, with $\Sigma_1\to\Sigma$ surjective.
Our task is to show the existence of a Cartesian square:
$$
 \begin{tikzpicture} 
\matrix(m)[matrix of math nodes, row sep=1.5em, column sep=.2em]%
 {                       & |(P)|{X\times_S Y} &           \\
 |(1)|{X} &                         & |(2)| {Y}\\
                & |(h)|{S}            & \\}; 
\path[->,font=\scriptsize,>=to, thin]
(P) edge  (1) edge (2)
(1) edge  (h)
(2) edge  (h);
\end{tikzpicture}
$$
Standard reductions allow us to assume that $X$, $Y$, $S$ are affine, difference spectra
of  $(A_1,\Sigma_1)$, $(A_2,\Sigma_2)$, $(C,\Sigma)$, respectively. Assume we
have a diagram 
$$
 \begin{tikzpicture} 
\matrix(m)[matrix of math nodes, row sep=1.5em, column sep=-.3em]%
 {                       & |(P)|{(Z,T)}&           \\
 |(1)|{(X_1,\Sigma_1)} &                         & |(2)| {(X_2,\Sigma_2)}\\
                & |(h)|{(S,\Sigma)}           & \\}; 
\path[->,font=\scriptsize,>=to, thin]
(P) edge  (1) edge (2)
(1) edge  (h)
(2) edge  (h);
\end{tikzpicture}
$$
Using \ref{embedcat}, we obtain
$$
 \begin{tikzpicture} 
\matrix(m)[matrix of math nodes, row sep=1.5em, column sep=-1.3em]%
 {                       & |(P)|{(H^0(Z),T)} &           \\
 |(1)|{(\bar{A}_1,\Sigma_1)}  &                         & |(2)|{(\bar{A}_2,\Sigma_2)}\\
                & |(h)|{(C,\Sigma)}           & \\}; 
\path[->,font=\scriptsize,>=to, thin]
(h) edge  (1) edge (2)
(1) edge  (P)
(2) edge  (P);
\end{tikzpicture}
$$

Writing $\Sigma_{\times}=\{\sigma_1\otimes\sigma_2:(\sigma_1,\sigma_2)\in\Sigma_1\times_\Sigma\Sigma_2\}$, by the universal property of tensor products as well as the induced difference structure map $T\to\Sigma_1\times_\Sigma\Sigma_2$,
we get a unique morphism 
$
(\bar{A}_1\otimes_C\bar{A}_2,\Sigma_{\times})\to (H^0(Z),T)$.
Since $(H^0(Z),T)$ is well-mixed, we obtain a unique morphism
$(\bar{A}_1\otimes_C\bar{A}_2,\Sigma_{\times})_w\to (H^0(Z),T)$.
 Thus 
$\spec^{\Sigma_{\times}}((\bar{A}_1\otimes_C\bar{A}_2)_w)$ can play the
role of $(X_1,\Sigma_1)\times_{(S,\Sigma)}(X_2,\Sigma_2)$.
\end{proof}

\begin{remark} 
Extra care will be needed when dealing with products in our subsequent work
in order to avoid the following
undesirable situations.
\begin{enumerate}
\item Even in the category of ordinary difference schemes, the (fibre) product
can be $\emptyset$. That is the case when $(X_1,\sigma_1)$ and $(X_2,\sigma_2)$
are \emph{incompatible}, so we cannot find a $(Z,\sigma)$ which admits
a morphism to $(X_1,\sigma_1)$ and $(X_2,\sigma_2)$ at the same time
(see \cite[p.~60]{cohn} for examples). 
\item
 If we omit the surjectivity assumption from \ref{fibreprod}, the product 
need not exist. For example, if $(X,\Sigma)$ is a difference scheme with $|\Sigma|\geq2$,
pick $\sigma_1,\sigma_2\in\Sigma$ and consider the morphisms 
$(X^{\sigma_i},\sigma_i)\to(X,\Sigma)$ for $i=1,2$. The product 
$(X^{\sigma_1},\sigma_1)\times_{(X,\Sigma)}(X^{\sigma_2},\sigma_2)$ will not exist in general.
\end{enumerate}
\end{remark}

\begin{lemma}\label{glsprod}
Suppose we have morphisms of well-mixed difference rings $(C,T)\to(A,\Sigma_1)$ and  $(C,T)\to(B,\Sigma_2)$,
let $\bar{A}=H^0(\spec^{\Sigma_1}(A))$, $\bar{B}=H^0(\spec^{\Sigma_2}(B))$
and let $\Sigma_{\times}$ be as above.
Then we have an isomorphism of difference schemes $$\spec^{\Sigma_{\times}}((A\otimes_C B)_w)\simeq\spec^{\Sigma_{\times}}((\bar{A}\otimes_{C}\bar{B})_w).$$
\end{lemma}
\begin{proof}
In the diagram 

$$
 
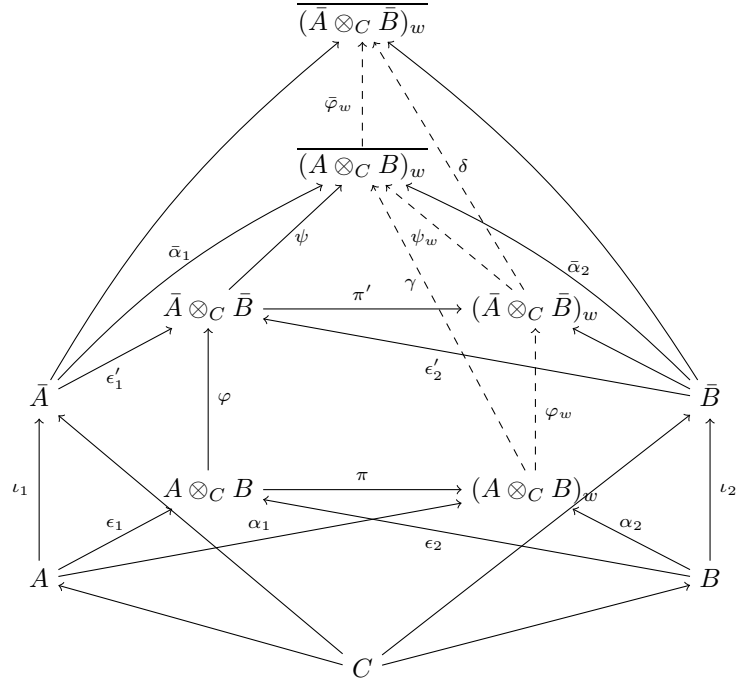
\begin{tikzpicture} 
\matrix(m)[matrix of math nodes, row sep=2em, column sep=1em, text height=1.5ex, text depth=0.25ex]
 {    \hspace{6em}  & &   |(a)|{\overline{(\bar{A}\otimes_C\bar{B})_w}}&       &  \hspace{5em}  \\[2em]
        &&   |(b)|{\overline{(A\otimes_CB)_w}}  &        &              \\[2em]
        & |(d)|{\bar{A}\otimes_C\bar{B}}  &  & |(c)|{(\bar{A}\otimes_C\bar{B})_w}      &    \\    %
                |(f)|{\bar{A}}& &&& |(g)|{\bar{B}}\\           
       &   |(h)|{A\otimes_CB} &  & |(e)|{(A\otimes_CB)_w} &   \\    
 |(i)|{A} &&&& |(j)|{B}\\     
  &&|(k)|{C}&&       \\}; 
\path[->,font=\scriptsize,>=to, thin]
(k) edge  (i) edge (f) edge (g) edge (j)
(i) edge node[left]{$\iota_1$} (f) edge node[above]{$\epsilon_1$} (h) edge[bend left=0] node[above]{$\alpha_1$} (e)
(j) edge node[pos=0.6,below]{$\epsilon_2$} (h) edge node[pos=0.5,above]{$\alpha_2$} (e) edge node[right]{$\iota_2$} (g)
(f) edge[bend left=10] node[above]{$\bar{\alpha}_1$} (b) edge node[below]{$\epsilon_1'$}(d) edge[bend left=10]  (a)
(h) edge node[pos=0.5,right]{$\varphi$} (d) edge node[above]{$\pi$} (e)
(g) edge node[pos=0.6,below]{$\epsilon_2'$} (d) edge[bend right=10] node[right]{$\bar{\alpha}_2$} (b) edge[bend right=10] (a) edge (c)
(e) edge[dashed] node[pos=0.4,right]{$\varphi_w$} (c) edge[dashed] node[pos=0.65,left]{$\gamma$} (b)
(d) edge[bend left=0] node[pos=0.5,right]{$\psi$} (b) edge node[above]{$\pi'$} (c)
(c) edge[dashed] node[right]{$\delta$} (a) edge[dashed] node[left]{$\psi_w$} (b)
(b) edge[dashed] node[pos=0.4,left]{$\bar{\varphi}_w$} (a);
\end{tikzpicture}
$$
the morphism $\varphi$ is obtained from the initial setup by using \ref{wmaff}(\ref{sesttt}) and the universal property of tensor products, and 
$\varphi_w$ is obtained from it by the universal property of passing to the well-mixed quotient.
The morphisms $\alpha_1$ and $\alpha_2$ are obtained by precomposing with the
quotient $A\otimes_C B\to (A\otimes_CB)_w$, and $\bar{\alpha}_1$ and $\bar{\alpha}_2$ result
from them by functoriality of passing to global sections. We define $\psi$ via the universal property of
tensoring, and we use the universal property of well-mixed quotients to  get $\psi_w$ from it.
Finally, by passing to global sections we derive $\bar{\varphi}_w$ from $\varphi_w$. 

We wish to show that the diagram commutes, with particular emphasis on the dashed
part. The only commutativity relations which 
are not a priori clear are $\psi_w\circ\varphi_w=\gamma$ and
 $\bar{\varphi}_w\circ\psi_w=\delta$. In order to show the first relation, since $\pi$ is onto,
 it suffices to verify $\psi_w\circ\varphi_w\circ\pi=\gamma\circ\pi$, i.e., that
 $\psi\circ\varphi=\gamma\circ\pi$. Since $\psi\circ\varphi$ is the unique morphism
 $\chi$ satisfying $\bar{\alpha}_i\circ\iota_i=\chi\circ\epsilon_i$, $i=1,2$,
 the above follows from the construction since $\bar{\alpha}_i\circ\iota_i=\gamma\circ\alpha_i=\gamma\circ\pi\circ\epsilon_i$.

For the second relation, since $\pi'$ is onto, it suffices to show that
$\bar{\varphi}_w\circ\psi=\delta\circ\pi'$. Since $\bar{\varphi}_w\circ\psi$ is
the unique morphism $\chi$ satisfying 
$\chi\circ\epsilon_i'=\bar{\varphi}_w\circ\bar{\alpha}_i$, $i=1,2$,
it is enough to verify that 
$\delta\circ\pi'\circ\epsilon_i'=\bar{\varphi}_w\circ\bar{\alpha}_i=\overline{\varphi_w\circ\alpha_i}$, but say the first one these follows by applying \ref{univgs} to $A$, 
$(\bar{A}\otimes_C\bar{B})_w$, $\varphi_w\circ\alpha_1$ and $\pi\circ\epsilon_1'$.

From the dashed part of the above diagram, we deduce:
$$
 \begin{tikzpicture} 
\matrix(m)[matrix of math nodes, row sep=3em, column sep=2em, text height=1.5ex, text depth=0.25ex]
 { \spec^{\Sigma_{\times}}({\overline{(\bar{A}\otimes_C\bar{B})_w}})&  \spec^{\Sigma_{\times}}({\overline{(A\otimes_CB)_w}})\\
    { \spec^{\Sigma_{\times}}((\bar{A}\otimes_C\bar{B})_w)} &  { \spec^{\Sigma_{\times}}((A\otimes_CB)_w})    \\}; 
\path[->,font=\scriptsize,>=to, thin]
(m-1-1) edge  (m-1-2) edge (m-2-1)
(m-2-1) edge (m-2-2) 
(m-1-2) edge  (m-2-2) edge (m-2-1);
\end{tikzpicture}
$$
 The vertical arrows are isomorphisms by \ref{wmaff}(\ref{sesttt}) and thus  the diagonal is also an isomorphism.
\end{proof}

When working in a relative setting over a base $(S,\Sigma_0)$, it is natural to
think of a morphism $(X,\Sigma)\to (S,\Sigma_0)$ as a \emph{family} of difference
schemes parametrised by parameters from $S$. We make the notion of
a \emph{fibre} of the morphism precise.

\begin{definition}\label{d:fibre}
Let $(X,\Sigma)\to (S,\Sigma_0)$  be a morphism of difference schemes, let
$(K,\varphi)$ be a difference field and let
$s\in (S,\Sigma_0)(K,\varphi)$. The \emph{fibre} $X_s$ is the 
$(K,\varphi)$-difference scheme obtained by base change via the 
morphism $s:\spec^\varphi(K)\to(S,\Sigma_0)$,
$$
(X_s,\Sigma_s)=(X,\Sigma)\times_{(S,\Sigma_0)}\spec^\varphi(K).
$$
\end{definition}

\begin{definition}
Let $P$ be a property of difference schemes. If $(X,\Sigma)\to (S,\Sigma_0)$
is a difference scheme over a given base, we shall say that $X$ is 
\emph{geometrically} $P$, if every base change of $X$ has the property $P$.
\end{definition}

\subsection{Finiteness properties}

If $(R,\Sigma_0)$ is a difference ring and $\pi:\Sigma\to\Sigma_0$ is a $\diff$-morphism, the \emph{difference polynomial ring}
 $R[x_1,\ldots,x_n]_\Sigma$ in $n$ variables over $(R,\Sigma_0)$
is defined as the polynomial ring
$$
R[x_{1,\nu},\ldots,x_{n,\nu}:\nu\in \langle\Sigma\rangle],
$$
where $\langle\Sigma\rangle$ is the semigroup generated by $\Sigma$ defined in
\ref{diff-ops}, 
together with the unique $\Sigma$-structure whereby a $\sigma\in\Sigma$ acts as 
$\pi(\sigma)$  on $R$ and
maps $x_{j,\nu}$ to $x_{j,\sigma\nu}$.

\begin{definition}\label{finsigmatype}
Let $(R,\Sigma_0)$ be a difference ring.
\begin{enumerate}
\item An $(R,\Sigma_0)$-algebra $(S,\Sigma)$ is of \emph{finite $\Sigma$-type} if it is an equivariant quotient of some difference polynomial ring 
$R[x_1,\ldots,x_n]_\Sigma$. Equivalently, 
there exist elements $a_1,\ldots,a_n\in S$ such that $S=R[a_1,\ldots,a_n]_\Sigma$.
\item An $(R,\Sigma_0)$-difference scheme $(X,\Sigma)$ is of \emph{finite $\Sigma$-type}, or of \emph{finite transformal type} if
it is a finite union of affine difference schemes of the form $\spec^\Sigma(S)$, where $(S,\Sigma)$ is of finite $\Sigma$-type over $(R,\Sigma_0)$.
\item\label{finstypmor} A morphism $f:(X,\Sigma)\to (Y,\Sigma_0)$ is
\emph{of finite $\Sigma$-type} if $Y$ is a finite union
of open affine subsets $V_i=\spec^{\Sigma_0}(R_i)$ such that for each $i$, $f^{-1}(V_i)$ is
of finite $\Sigma$-type over $(R_i,\Sigma_0)$. 
\item A morphism $f:(X,\Sigma)\to(Y,\Sigma_0)$ is \emph{integral} (resp.~\emph{finite})
if $Y$ is a finite union of open affine subsets 
$V_i=\spec^{\Sigma_0}(R_i)$ such that for each $i$, 
$f^{-1}(V_i)$ is $\spec^\Sigma(S_i)$, where $S_i$ is integral (resp.~finite)
over $R_i$.
\item A morphism is \emph{transformally finite} if it is integral and of transformally finite type.
\item A morphism $f:(X,\Sigma)\to(Y,\Sigma_0)$ is \emph{quasi-finite} if
if is of finite transformal type and for every $y\in Y$, the fibre $X_y$ is finite
over $\kk(y)$.
\end{enumerate}
\end{definition}

\begin{remark}\label{fullfintyp}
Let $(S,\Sigma)$ be an $(R,\Sigma_0)$ algebra, with $\Sigma$ full, and let
$\pi:\Sigma\to\Sigma_0$ be the corresponding difference structure morphism.
The following statements are equivalent:
\begin{enumerate}
\item $(S,\Sigma)$ is of finite $\Sigma$-type over $(R,\Sigma_0)$;
\item $(S,\sigma)$ is of finite $\sigma$-type over $(R,\pi(\sigma))$ for all $\sigma\in\Sigma$;
\item $(S,\sigma)$ is of finite $\sigma$-type over $(R,\pi(\sigma))$ for some $\sigma\in\Sigma$.
\end{enumerate}
\end{remark}
\begin{proof}
Writing $\rho_{\tau,\sigma'}:\Sigma\to\Sigma$ for the bijection
$\rho_{\tau,\sigma'}(\sigma)=\lrexp{\tau}{\sigma}{{\sigma'}}$,
we have that 
$$
\sigma\tau=\sigma'\rho_{\tau,\sigma'}^{-1}(\sigma).
$$
Suppose $S=R[a]_\Sigma$ for some tuple $a$. Let $\bar{a}=\{\tau a:\tau\in\Sigma\}$.
 By the above generalised conjugation relation, we conclude
that for any $\sigma,\sigma'\in\Sigma$, $\sigma\bar{a}=\sigma'\bar{a}$
and that indeed $S=R[\bar{a}]_\sigma$
for any $\sigma\in\Sigma$.
\end{proof}

\begin{remark}\label{fintyp}
Suppose $X=\spec^\Sigma(R)$ and $Y=\spec^T(S)$ are two affine difference schemes
of finite transformal type over some difference field $(k,\sigma)$ 
and let $f:X\to Y$ be a morphism.
We cannot conclude that $f$ is the spectrum of some morphism 
$(S,T)\to (R,\Sigma)$.
On the other hand, by \ref{embedcat}, we have a map $f^\sharp:S\to \bar{R}$, and even though 
$\bar{R}$ may not be of finite $\Sigma$-type over $k$, 
by \ref{intrmed}, we can get $f$ as the spectrum of the
induced map $S\to R[f^\sharp(S)]$, which is a map of algebras of finite transformal type over $k$.

Therefore, any given commutative diagram of affine difference schemes of finite 
transformal type
can be assumed to arise from a (dual) commutative diagram of difference rings of finite
transformal type (over a given base).

Moreover, this shows that a morphism between affine difference schemes of finite 
transformal type
is of finite transformal type, justifying \ref{finsigmatype}(\ref{finstypmor}).
\end{remark}

\begin{proposition}
Any base change of a morphism of finite transformal type is again of finite
transformal type.
\end{proposition}
Suppose that we have a diagram
$$
 \begin{tikzpicture} 
\matrix(m)[matrix of math nodes, row sep=1.5em, column sep=.2em]%
 {                       & |(P)|{X\times_S Y} &           \\
 |(1)|{X} &                         & |(2)| {Y}\\
                & |(h)|{S}            &\\}; 
\path[->,font=\scriptsize,>=to, thin]
(P) edge  (1) edge (2)
(1) edge  (h)
(2) edge  (h);
\end{tikzpicture}
$$
where $(X,\Sigma)\to (S,T)$ is of finite $\Sigma$-type and reduce to the affine case
with notation from the proof of \ref{fibreprod}. In view of \ref{fintyp} and \ref{glsprod}, we can find a difference
ring $A'$ of finite $\Sigma$-type over $C$
such that $X\times_SY$ can be realised as $\spec^{\Sigma_{\times}}((A'\otimes_C\bar{B})_w)$, where $\Sigma_{\times}=\Sigma\times_T\Sigma_Y$
and this is clearly of finite $\sigma$-type over $(Y,\Sigma_Y)$.

An easy corollary of the proposition is that if both $(X,\Sigma)\to (S,T)$ and
$(Y,\Sigma_Y)\to (S,T)$ are of finite transformal type, then $X\times_SY\to S$ is
again of finite transformal type. Thus the \emph{category of difference scheme of finite transformal type over a given base
has fibre products.}

The following result shows that the various finiteness properties of morphisms
 are very close to being independent of the choice of an open affine
covering in \ref{finsigmatype}. On the other hand, the notions of 
having \emph{almost} (with terminology from 
\cite[Section~\ref{local-properties}]{ive-tgs}) the corresponding finiteness property we can derive from it 
are intrinsic.
\begin{proposition}
Let $(A,\sigma)\to(B,\sigma)$ be a homomorphism of well-mixed difference rings,
let
$(M,\sigma)$ be an $(A,\sigma)$-module and let $f_1,\ldots,f_n\in A$ be such that
$\spec^\sigma(A)=\cup_i D^\sigma(f_i)$.
\begin{enumerate}
\item\label{jj} If each $(M_{f_i},\sigma)$ is (algebraically) finite over $(A_{f_i},\sigma)$, then $(M,\sigma)$ 
is almost finite over $(A,\sigma)$.
\item\label{dd} If each $B_{f_i}$ is integral over $A_{f_i}$, then $B$ is almost integral over $A$.
\item\label{tt} If each $(B_{f_i},\sigma)$ is of finite $\sigma$-type over $(A_{f_i},\sigma)$, then $(B,\sigma)$
is almost of finite $\sigma$-type over $(A,\sigma)$.
\item\label{cc} If each $(B_{f_i},\sigma)$ is $\sigma$-finite 
over $(A_{f_i},\sigma)$, then $(B,\sigma)$ 
is almost $\sigma$-finite over $(A,\sigma)$.
\end{enumerate}
\end{proposition}
\begin{proof}
(\ref{jj}) By assumption, for every $i$ there exists a finite free $A$-module $M_i'$
and a morphism $\varphi_i:M_i'\to M$ whose localisation
${M_i'}_{(f_i)_\sigma}\to M_{(f_i)_\sigma}$ is surjective. We can lift the 
operator $\sigma$ to $M_i'$ by freeness, and thus achieve that $\varphi_i$ be
a morphism of $(A,\sigma)$-modules.
Let $M'=\oplus_iM_i'$ and let $\varphi=\oplus_i\varphi_i:M'\to M$.
By construction, for every $i$,  ${M'}_{(f_i)_\sigma}\to M_{(f_i)_\sigma}$ is surjective
so, in particular, for every $\p\in\spec^\sigma(A)$, $M'_\p\to M_p$ is surjective.
Thus, by \cite[\ref{locwksurj}]{ive-tgs}, $M'\to M$ is almost surjective and $M$ is almost finite over $A$.  

We can give a more direct proof of the above as follows. By assumption, there is
an $N$ such that for every $i\leq n$, $\{m_{ij}/f_i^{\nu_j}:j\leq N\}$ generate $M_{(f_i)_\sigma}$, for some $m_{ij}\in M$ and $\nu_j\in\N[\sigma]$. Let $M''$ be generated by $\{m_{ij}\}$. %
Then, for every $m\in M$ there exists a $\mu\in\N[\sigma]$ such that for every $i$, 
$f_i^\mu.m\in M''$.  Let $J=\{a\in A:a.m\in [M'']_w\}$, which is clearly a well-mixed difference ideal.
Then, since $f_i^\mu\in J$, we have that $\{J\}_\sigma=\{f_1,\ldots,f_n\}_\sigma=A$,
so $1\in J$ and we conclude that $m\in [M'']_w$.

(\ref{tt}) We use the template from the first proof of (1). By assumption, for every $i$
there exists a $\sigma$-polynomial $(A,\sigma)$ algebra $B'_i$ and a morphism
$B'_i\to B$ such that ${B'_i}_{(f_i)_\sigma}\to B_{(f_i)_\sigma}$ is surjective.
Let $B'=\bigotimes_iB'_i$. Then the associated morphism (which is in particular
a morphism of $(A,\sigma)$-modules) $B'\to B$ has
the property that for every $\p\in\spec^\sigma(A)$, the localisation $B'_\p\to B_\p$ is surjective
and \cite[\ref{locwksurj}]{ive-tgs} implies that $B$ is almost of finite $\sigma$-type over $A$.
\end{proof}

\begin{definition}
A difference ring $(R,\Sigma)$ is \emph{Ritt} if it has the ascending chain condition on $\Sigma$-perfect ideals, or equivalently, if $\spec^{(\Sigma)}(R)$ is topologically Noetherian.
\end{definition}

The following results are known (\cite{cohn}) in the strict difference case.
\begin{proposition}\label{strictritt}
\begin{enumerate}
\item Every difference ring of finite $\sigma$-type over a Ritt difference ring is Ritt.
\item Difference schemes of finite $\sigma$-type over a Ritt difference ring are topologically 
Noetherian.
\item Every open subset of a difference scheme of finite $\sigma$-type over a Ritt difference ring
if of finite $\sigma$-type.
\end{enumerate}
\end{proposition}

\begin{corollary}\label{strictcomp}
Let $(R,\sigma)$ be a Ritt difference ring. If $I$ is $\sigma$-perfect, we have a unique irredundant decomposition
$V\{I\}=V\{\p_1\}\cup\cdots\cup V\{\p_n\}$ with $\p_i\in\spec^\sigma(R)$. In other words,
$I=\p_1\cap\ldots\cap\p_n$.
\end{corollary}

By generalising techniques of \cite{levin}, it can be shown that if $(S,\Sigma)$
is of finite $\Sigma$-type over a Ritt difference ring $(R,\Sigma_0)$, then
$(S,\Sigma)$ is again Ritt, i.e., $\spec^{(\Sigma)}(S)$ is topologically Noetherian.
We are interested, however, in the study of $\spec^\Sigma(S)$, which might be
unfathomable using these techniques, and that is why we restrict our attention
to rings with full difference structures (which notably includes the almost-strict case).

\begin{lemma}
Let $(S,\Sigma)$ be a full algebra of $\Sigma$-finite type over a Ritt
difference ring $(R,\sigma_0)$, and $\Sigma$ finite.
Then $\spec^\Sigma(S)$ is topologically Noetherian and $S$ has an
ascending chain condition on ideals which are perfect with respect to any
$\sigma\in\Sigma$. 
\end{lemma}   
\begin{proof}
By \ref{fullfintyp}, $(S,\sigma)$ is of finite $\sigma$-type over $R$ for
every $\sigma\in\Sigma$. Thus 
$\spec^\Sigma(S)=\cup_{\sigma\in\Sigma}\spec^\sigma(S)$ is clearly Noetherian 
by \ref{strictritt}.
\end{proof}   
   
Working with generalised difference schemes, we obtain an analogue of
\ref{strictcomp}, as well as identify a new phenomenon where difference 
schemes can decompose topologically but not structurally. We shall return to the
study of
this behaviour in \cite[Section~\ref{sect:bi-fib}]{ive-tgs}. 
\begin{lemma}
Suppose $(R,\Sigma)$ is a full difference ring such that each $(R,\sigma)$ is
Ritt. Let $\iota:\Sigma_0\hookrightarrow\Sigma$ and suppose $I$ is a
$\Sigma$-reflexive ideal which is $\Sigma_0$-well mixed.
Then
we have an irredundant decomposition 
$$
V^{\Sigma}(I)=\bigcup_{i=1}^{n}\iota_*V^{\Sigma_0}(\p_i),$$
where
$\p_i\in\spec^{(\Sigma_0)}(R)$, and 
$$
\iota_*V^{\Sigma_0}(\p_i)=\bigcup_{\tau\in\Sigma}V^{\Sigma_0^\tau}(\tau^{-1}\p_i).
$$
\end{lemma}
The $\Sigma$-difference schemes $\iota_*V^{\Sigma_0}(\p_i)$ should be
thought of as `structural components' of $V^\Sigma(I)$, even though they
decompose further into topological components which are shuffled by $\Sigma$.
\begin{proof}
Since $I$ is $\Sigma_0$-well-mixed, using the ascending chain condition on
$\Sigma_0$-perfect ideals and techniques of \ref{perfintrs}, we can 
find a decomposition 
$V^{\Sigma_0}(I)=\cup_{i\leq n}V^{\Sigma_0}(\p_i)$, with 
$\p_i\in\spec^{(\Sigma_0)}(R)$. Applying $\tau\in\Sigma$ to the equality,
using \ref{fulfinorb}, yields
$V^{\Sigma_0^\tau}(I)=\cup_{i\leq n}V^{\Sigma_0^\tau}(\tau^{-1}\p_i)$,
and it suffices to take the union over all $\tau\in\Sigma$. 
\end{proof}

 When necessary, we shall assume the following {\bf finiteness condition}.
A \emph{difference scheme} will  mean an almost strict difference scheme $(X,\Sigma)$
which can be covered by a finite number of affine difference schemes which are
themselves almost strict of finite $\Sigma$-type over a difference field.

\subsection{Dimensions and degrees}
For a point $x$ on a difference scheme $(X,\Sigma)$, we denote by $(\OO_x,\Sigma_x)$ the local
(difference) ring at $x$, and by $(\kk(x),\Sigma^x)$ the residue (difference) field at
 $x$.

\begin{definition}\label{def-limdeg}
Let $(K,\sigma)\subseteq (L,\sigma)$ be an extension of difference fields.
\begin{enumerate}
\item An element $\alpha\in L$ is \emph{$\sigma$-algebraic} over $K$
if the set $\{\alpha,\sigma(\alpha),\sigma^2(\alpha),\ldots\}$ is
algebraically dependent over $K$.
\item The $\sigma$-algebraic closure over $K$ defines a \emph{pregeometry} on $L$
and the dimension with respect to this pregeometry is called the $\sigma$-transcendence
degree. Alternatively, $\sigma\text{\rm-tr.deg}(L/K)$ is the supremum of numbers
$n$ such that the difference polynomial ring $K\{x_1,\ldots,x_n\}$ in $n$ variables
embeds in $L$.
\item $L$ is \emph{$\sigma$-separable} over $K$ if $L$ is linearly disjoint
from $K^\text{inv}$ over $K$, where the \emph{inversive closure} 
$(K^\text{inv},\sigma)$, as in \ref{definvring}, is the unique
(up to $K$-isomorphism) difference field extension of $(K,\sigma)$ where 
$\sigma$ is an automorphism of $K^\text{inv}$ and 
$$
K^\text{inv}=\bigcup_mK^{\sigma^{-m}}.
$$
\item Suppose $L$ is $\sigma$-algebraic of finite $\sigma$-type over $K$,
$\sigma$-generated by a finite set $A$. 
Let $A_k:=\bigcup_{i\leq k}\sigma^i(A)$ and let $d_k:=[K(A_k):K(A_{k-1})]$.
It is shown in \cite{cohn} that for every $k$, $d_k\geq d_{k+1}$ and
we may define the \emph{limit degree} as
$$
\dl((L,\sigma)/(K,\sigma)):=\min_k d_k.
$$
This definition is independent of the choice of the generators. 
When $L/K$ is $\sigma$-algebraic but not necessarily finitely $\sigma$-generated,
 one defines 
$\dl(L/K)$ as the maximum of $\dl(L'/K)$ where $L'$ runs over the 
extensions of finite $\sigma$-type contained in $L$.
\end{enumerate}
\end{definition}

\begin{lemma}\label{fulllimdeg}
Let $(K,\Sigma_0)\subseteq(L,\Sigma)$ be an extension of difference fields,
with $\Sigma$ full and finite 
and let $\pi:\Sigma\to\Sigma_0$ be the associated $\diff$-morphism. Then, for any $\sigma,\sigma'\in\Sigma$,
$$
\dl((L,\sigma)/(K,\pi(\sigma)))=\dl((L,\sigma')/(K,\pi(\sigma')).
$$
\end{lemma}
\begin{proof}
We may suppose $L=K(a)_\Sigma$ for a finite tuple $a\in L$.
Let $\bar{a}=\{\tau a:\tau\in\Sigma\}$. 
As in the proof of \ref{fullfintyp} we conclude
that $\sigma\bar{a}=\sigma'\bar{a}$ for any $\sigma,\sigma'\in\Sigma$. Let 
$$L_n=K(\bar{a},\sigma\bar{a},\ldots,\sigma^n\bar{a})
=K(\bar{a},\sigma'\bar{a},\ldots,{\sigma'}^n\bar{a}).
$$
Since the limit degree does not depend on the choice of generators,
$$
\dl((L,\sigma)/(K,\pi(\sigma)))=\lim_n[L_{n+1}:L_n]=\dl((L,\sigma')/(K,\pi(\sigma')).
$$
\end{proof}
 The following definition makes sense by \ref{fulllimdeg}.
\begin{definition}
Let $(K,\Sigma_0)\subseteq(L,\Sigma)$ be an extension of difference fields
with $\Sigma$ full and finite. Let $\dl((L,\Sigma)/(K,\Sigma_0))=\dl((L,\sigma)/K)$,
for any $\sigma\in\Sigma$.
\end{definition}

Before introducing the various dimensions and degree invariants of difference
schemes, it is useful to define an auxiliary structure where some of those
invariants will take values.
 
The \emph{rig} (ring without negatives) $\N\cup\{\infty\}[\LL]$ admits a
natural lexicographic polynomial ordering $\leq$, and an equivalence relation
 $\approx$, where
$u\approx v$ if $u,v\in\N[\LL]$ have the same degree in $\LL$ and
 and their leading coefficients are equal. We will consider the
 rig $\N\cup\{\infty\}[\LL]/{\approx}$. 

\begin{definition}\label{dimtot-ring}
Let $(k,\sigma_0)$ be a difference field, $(K,\Sigma)$ a full 
difference field extension and let $(R,\Sigma)$ be a full $(k,\sigma)$-algebra.
\begin{enumerate}
\item
Let the \emph{transformal degree} of $(K,\Sigma)$ be 
$$\dd(K/k)=\dl(K/k)\LL^{\mathop{\rm tr.deg}(K/k)}$$ in $\N\cup\{\infty\}[\LL]/{\approx}$.
\item
Let the \emph{effective transformal degree} of $X$ be
$\dde(K/k)=\dd(K^{\rm inv}/k^{\rm inv})$.
\item
Let $$\dd(R/k)=\sum\limits_{\substack{\p\in\spec(R) \\  \sigma(\p)\subseteq\p}}\dd(\kk(\p)/k),$$ and analogously for $\dde(R/k)$.
\item The  \emph{limit degree} $\dl(R/k)$ and \emph{total dimension}  $\dt(R/k)$
are defined through 
$$
\dl(R/k)\LL^{\dt(R/k)}\approx\dd(R/k),
$$ 
and analogously for the \emph{effective total dimension}.
\end{enumerate}
\end{definition}

\begin{definition}\label{trdim-sch}
Let $(k,\sigma)$ be a difference field, and consider a morphism 
$\varphi:(X,\sigma)\to (Y,\sigma)$ of $(k,\sigma)$-difference schemes.
\begin{enumerate}
\item The \emph{$\sigma$-dimension} of $X$ is 
$
\sd(X)=\sup\limits_{x\in X}\sigma\text{\rm-tr.deg}(\kk(x)/k).
$
\item The \emph{relative $\sigma$-dimension} 
$$
\sd(\varphi)=\sup_{y\in Y}\sd(X_y),
$$
where $X_y=X\times_Y\specd(\kk(y))$ is the fibre over $y$.
\end{enumerate}
\end{definition}

\begin{definition}\label{dimtot-sch}
Let $(k,\sigma)$ be a difference field, and consider a morphism $\varphi:(X,\Sigma)\to (Y,T)$ of full $(k,\sigma)$-difference schemes.
\begin{enumerate}
\item 
Let the \emph{transformal degree} of $X$ be
$$\dd(X)=\sum_{x\in X}\dd(\OO_x/k),$$
 and analogously for $\dde(X)$.

\item The  \emph{limit degree} $\dl(X)$ and \emph{total dimension}  $\dt(X)$
are defined through 
$$
\dl(X)\LL^{\dt(X)}\approx\dd(X),
$$ 
and analogously for the \emph{effective total dimension}.

\item The \emph{relative transformal degree} 
$$
\dd(\varphi)=\sup_{y\in Y}\dd(X_y),
$$
and analogously for $\dde(\varphi)$. From these we derive the notions
of \emph{relative limit degree} and \emph{relative total dimension}.
\end{enumerate}
\end{definition} 

\begin{remark}
\begin{enumerate}
\item Clearly (cf.~\cite{udi}, \cite{laszlo}), $\sd(X)=0$ if and only $\dt(X)$ and $\dte(X)$ are
finite, and analogously for the relative dimensions. In this case, if in addition $\varphi$
is of finite $\sigma$-type, 
$\dd(\varphi)\in\N[\LL]$, i.e., the limit degree is finite.
\item
When $X=\specd(R)$, $\dt(X)=\dt(R)$, so the above definition is consistent. Indeed,
as remarked in \cite{laszlo}, the inequality
\[
\dt(R)\geq\sup_{\p\in\specd(R)}\dt(R_{\p})
\]
is obvious. In the other direction, let $\p\in\spec(R)$ such that $\sigma(\p)\subseteq\p$.
Then $\p$ induces a $\sigma$-ideal in $\spec(R_{\bar{\p}})$, where 
$\bar{\p}=\cup_{m>0}\sigma^{-m}(\p)$ is the
perfect closure of $\p$, and the opposite inequality follows.

\item When $(L,\sigma)$ is a $\sigma$-separable $\sigma$-algebraic extension of $(K,\sigma)$,
$L^{\rm inv}$ is an algebraic extension of $LK^{\rm inv}$ and 
$$\trdeg(L^{\rm inv}/K^{\rm inv})=\trdeg(LK^{\rm inv}/K^{\rm inv})=\trdeg(L/K).$$ Thus, when
$\varphi:X\to Y$ is \emph{$\sigma$-separable} in the sense that for every $x\in X$, the extension 
$\kk(x)/\kk(\varphi(x))$ is
$\sigma$-separable, we get that $$\dd(\varphi)\approx\dde(\varphi).$$

\item\label{tower} Thanks to the corresponding property of the limit degree and
the additivity of total dimension, the transformal degree is multiplicative in towers. 

\item\label{comp} Let $X=\specd(R)$. 
By the Ritt ascending chain condition
for perfect ideals in $R$ (\cite{cohn}), $X$ is a Noetherian topological space and therefore
we get a decomposition of $X$ into irreducible components,
$$
X=X_1\cup\cdots\cup X_n,
$$
where $X_i=\specd(R/\p_i)$ for some $\p_i\in \specd(R)$.
Equivalently, the zero ideal in $R$ can
be represented as
$$
0= \p_1\cap\cdots\cap\p_n.
$$
 Since $X$ is of transformal dimension $0$ (equivalently, of finite total dimension),
 for $i\neq j$, $\dt(X_i\cap X_j)<\dt(X)$ and the results of \cite{laszlo} entail
 $$
 \dd(X)\approx\sum_{i} \dd(X_i)
 \approx\sum_i\dl(\mathop{\rm Fract}(R/\p_i)/k)\LL^{\trdeg(\mathop{\rm Fract}(R/\p_i)/k)}.
 $$
An analogous statement holds for $\dde$.

\end{enumerate}

\end{remark}

\subsection{Normalisation}

\begin{definition}
Let $A$ be an integral domain with fraction field $K$ and let $V$ be a
finite dimensional vector space over $K$. An $A$-submodule $M$ of $V$
is called an \emph{$A$-lattice}
 in $V$ if $M$ contains a basis for $V$ as a $K$-module
(i.e.\ $KM=V$) and $M$ is a submodule of a finitely generated $A$-submodule
of $V$.
\end{definition}
 
\begin{proposition}[\cite{eisenbud-comm-alg}, 13.14]\label{norm-lattice}
Let $A$ be a normal domain with fraction field $K$ and let $L$ be a finite separable extension of $K$. Let $B$ be the integral closure of $A$ in $L$. Choose
an integral basis $b_1,\ldots,b_n\in B$ for $K$ over $L$ and let
$\Delta=\Delta(b_1,\ldots,b_n)$ be its discriminant. Then 
$$
B\subseteq\frac{1}{\Delta}\left( Ab_1+\cdots+Ab_n\right).
$$
In particular, $B$ is an $A$-lattice in $L$. 
\end{proposition}
 
Suppose $(A,\Sigma_0)$ is a normal domain with fraction field $(K,\Sigma_0)$ and
let $(L,\Sigma)$ be an extension of $(K,\Sigma_0)$ with $[L:K]$ finite.
Let $B$ be the integral closure of $A$ in $L$. Then $\sigma(B)\subseteq B$ for
every $\sigma\in\Sigma$
so $(B,\Sigma)$ is a difference ring extension of $(A,\Sigma_0)$.
\begin{corollary}\label{cor:norm-lattice}
With above notation, assume $L$ is separable over $K$. Then there is 
a $\Sigma_0$-localisation $A'=A[1/f]_{\Sigma_0}$ of $A$, for an element 
$f\neq 0$, such that the corresponding normalisation 
$B[1/f]_{\Sigma}$ is a finite $A'$-module.
\end{corollary}
\begin{proof}
Perform a transformal localisation at the discriminant.
\end{proof}
\begin{remark}
If $(A,\sigma)$ is of finite $\sigma$-type over $(B,\sigma)$ with both $A$, $B$
domains such that the fraction field of $A$ is finite separable over the fraction
field of $B$, there is a $\sigma$-localisation $B'$ of $B$ such that 
$A\otimes_BB'$ is a finite $B'$-module.
\end{remark}

\section{Galois covers}\label{sect:galois}

\subsection{Finite group actions and quotient difference schemes}

Suppose a finite group $G$ acts on  a difference scheme $(X,\Sigma)$ by automorphisms from the right, through a homomorphism $G^{op}\to\aut(X,\Sigma)$. We will not require a notational
device for distinguishing an element $g$ of $G$ from the corresponding automorphism, since
it only makes sense to compose $g$ with other morphisms of difference varieties when
$g$ is considered as an automorphism. 
 For any difference scheme $(Z,T)$, $G$ acts on the set $\Hom((X,\Sigma),(Z,T))$ on the
left and we can consider the set $\Hom((X,\Sigma),(Z,T))^G$ of $G$-invariant morphisms.
A natural question to ask is whether the functor $(Z,T)\mapsto\Hom((X,\Sigma),(Z,T))^G$
is representable, i.e., isomorphic to a functor $(Z,T)\mapsto\Hom((Y,\bar{\Sigma}),(Z,T))$.
In other words, is there a difference scheme $(Y,\bar{\Sigma})$ and a $G$-invariant morphism
$p:(X,\Sigma)\to(Y,\bar{\Sigma})$ such that for every $(Z,T)$, the function
$$\Hom((Y,\bar{\Sigma}),(Z,T))\to \Hom((X,\Sigma),(Z,T))^G,$$
defined by $\varphi\mapsto\varphi p$ is  bijective. If this is the case, we say that $(Y,\bar{\Sigma})$
is a \emph{quotient} of $(X,\Sigma)$ by $G$, and it is determined up to a unique isomorphism. 

It would be difficult to consider the existence of a completely general categorical quotient (in the above sense) in the difference context, but if we assume some additional reasonable 
(universal categorical or geometric) properties from the quotient,
such as that the fibres of $p$ are in fact $G$-orbits, we uncover the existence of a richer
structure by the following heuristic. Suppose, in the best possible case,  that for any (structural) morphisms $f,f':X\to X$, $pf=pf'$
implies the existence of a (unique) $h\in G$ such that $f'=hf$. 
Then, since $p\sigma g=\sigma^ppg=\sigma^pp=p\sigma$, there must be an $h\in G$
such that $\sigma g=h\sigma$. By the assumption that $\sigma$ is an epimorphism, it
follows that $h$ is unique and we denote it $g^\sigma$. Thus we obtain a
homomorphism $()^\sigma:G\to G$ for every $\sigma\in\Sigma$, with the property that
for every $g\in G$,
$$
\sigma g=g^\sigma\sigma.
$$
If $\sigma$ is invertible, %
it follows that $()^\sigma$ is in fact a group automorphism.

In this context, we will say that $(G,\tilde{\Sigma})$ acts on $(X,\Sigma)$, where
$\tilde{\Sigma}=\{()^\sigma:\sigma\in\Sigma\}$.

Let us prove the existence of quotients for affine difference schemes, where $(G,\tilde{\Sigma})$ acts (on the left) by automorphisms  on the difference ring $(A,\Sigma)$, i.e., for every 
$\sigma\in\Sigma$ and $g\in G$, $$g\sigma=\sigma g^\sigma.$$

\begin{proposition}\label{quotaff}
Suppose a finite group $(G,\tilde{\Sigma})$ acts on a difference ring $(A,\Sigma)$ 
so that
$\Sigma G=\Sigma$ and
let $(B,\bar{\Sigma})=A^G$ be the subring of invariants of $A$, $X=\spec^\Sigma(A)$, 
$Y=\spec^{\bar{\Sigma}}(B)$ and $p:(X,\Sigma)\to (Y,{\bar{\Sigma}})$ the
canonical ($G$-invariant) morphism. Then the following holds.
\begin{enumerate}
\item\label{ajn} $A$ is integral over $B$.

\item\label{cvaj} The morphism $p$ is surjective, its fibres are $G$-orbits and the topology of 
$Y$ it the quotient of the topology of $X$.

\item\label{draj} Let $x\in X$, $y=p(x)$, let $G_x$ be the stabiliser of $x$ and let 
$\Sigma_x=\{\sigma\in\Sigma: \sigma(x)=x\}=\{\sigma\in\Sigma: x\in X^\sigma\}$. 
Let $\tilde{\Sigma}_x=\{()^\sigma\in\tilde{\Sigma}: \sigma\in\Sigma_x\}$ and
$\tilde{\Sigma}^x=\{()^{\sigma^x}:\sigma\in\Sigma_x\}$, where $\sigma^x:\kk(x)\to\kk(x)$
is induced by $\sigma_x^\sharp:\OO_x\to\OO_x$ for every $\sigma\in\Sigma_x$.
Then $\kk(x)$ is
a quasi-Galois algebraic extension of $\kk(y)$ and the canonical map
$$(G_x,\tilde{\Sigma}_x)\to\left(\Gal(\kk(x)/\kk(y)),\tilde{\Sigma}^x\right)$$ is surjective.

\item\label{funf} The natural homomorphism $\OO_Y\to\left(p_*\OO_X\right)^G$ is an isomorphism.

\item\label{fir} $(Y,{\bar{\Sigma}})$ is a quotient difference scheme of $(X,\Sigma)$ by $G$.
\end{enumerate}
\end{proposition}
\begin{proof}
Let $\bar{\Sigma}=\{\sigma\restriction B:\sigma\in\Sigma\}$ and let 
$()^p:\Sigma\to\bar{\Sigma}$ be the restriction map. We need to check that
the elements of $\bar{\Sigma}$ are endomorphisms of $B$. For $\sigma\in\Sigma$
and $b\in B=A^G$,
$$
g\sigma b=\sigma g^\sigma b=\sigma b,
$$
for every $g\in G$, which shows that $\sigma(B)\subseteq B$, as required. Thus
$p$ is the morphism associated to the inclusion $(B,\bar{\Sigma})\hookrightarrow(A,\Sigma)$.

\noindent(\ref{ajn}) It is well-known that $A$ is integral over $B$ since every $a\in A$ is a root of the monic polynomial 
$\prod_{g\in G}(t-ga)\in B[t]$.

\noindent(\ref{cvaj}) Let us denote $\tilde{p}:\tilde{X}\to\tilde{Y}$ the morphism of
ambient affine schemes $\tilde{X}=\spec(A)$ and $\tilde{Y}=\spec(B)$ 
induced by $B\hookrightarrow A$, so that $p=\tilde{p}\restriction X$.
The statement of (\ref{cvaj}) is known for $\tilde{p}$ (\cite{sga1}, V.1.1), so it suffices
to prove that $\tilde{p}^{-1}(Y)=X$. Pick an $y\in Y$, say $y\in Y^\tau$, for some $\tau\in\bar{\Sigma}$. Then, for any $\sigma\in\Sigma$ restricting to $\tau$  ($\sigma^p=\tau$),
we have $\tilde{p}\circ\lexp{a}{\sigma}=\lexp{a}{\tau}\circ\tilde{p}$ on $\tilde{X}$, 
so given any $x\in\tilde{X}$ such that $\tilde{p}(x)=y$,
$$
\tilde{p}(\lexp{a}{\sigma}x)=\lexp{a}{\tau}(\tilde{p}x)=\lexp{a}{\tau}(y)=y=\tilde{p}(x),
$$
so there exists a $g\in G$ with $\lexp{a}{\sigma}x=gx$. By assumption, $G\Sigma\subseteq\Sigma$, so  $x\in X^{g^{-1}\sigma}\subseteq X$.

\noindent(\ref{draj}) The fact that the natural homomorphism $G_x\to\Gal(\kk(x)/\kk(y))$ is
surjective is known (loc.~cit), and the difference superstructure is a bookkeeping exercise.

\noindent(\ref{funf}) It is known (loc.~cit.) that $\OO_{\tilde{Y}}\to\left(\tilde{p}_*\OO_{\tilde{X}}\right)^G$ is an isomorphism. Let us write $i:X\hookrightarrow\tilde{X}$ and $j:Y\hookrightarrow\tilde{Y}$.
By applying the exact functor $j^{-1}$, we obtain an isomorphism
$$
\OO_Y=j^{-1}\OO_{\tilde{Y}}\to j^{-1}\left(\tilde{p}_*\OO_{\tilde{X}}\right)^G.
$$
On the other hand, for $V$ open in $Y$,
$$
\left(p_*\OO_X\right)^G(V)=\left(p_*i^{-1}\OO_{\tilde{X}}\right)^G(V)=
\left(i^{-1}\OO_{\tilde{X}}(p^{-1}(V))\right)^G=\left(\varinjlim_{\smash{\tilde{U}\supseteq p^{-1}(V)\ \,}}
\OO_{\tilde{X}}(\tilde{U})\right)^G,
$$
where the limit is taken over open subsets $\tilde{U}$ of $\tilde{X}$.  From the proof of (\ref{cvaj}), we know that $p^{-1}(V)=\tilde{p}^{-1}(V)$
is a $G$-invariant set and we can replace the above limit by the limit over the cofinal system of invariant neighbourhoods $p^{-1}(\tilde{V})$ of  $\tilde{p}^{-1}(V)$ with $\tilde{V}$ an open
neighbourhood of $V$ in $\tilde{Y}$, allowing the identifications
$$
 \left(\varinjlim_{\smash{\tilde{V}\supseteq V}}\OO_{\tilde{X}}(\tilde{p}^{-1}(\tilde{V}))\right)^G=
 \varinjlim_{\tilde{V}\supseteq V}\left(\OO_{\tilde{X}}(\tilde{p}^{-1}(\tilde{V}))\right)^G=
 j^{-1}\left(\tilde{p}_*\OO_{\tilde{X}}\right)^G(V),
 $$
 thus providing the desired isomorphism.
 
\noindent(\ref{fir}) This statement is immediate from (\ref{cvaj}) and (\ref{funf}).
\end{proof}

\begin{proposition}\label{propadmissibl}
Let $(X,\Sigma)$ be a difference scheme with a finite group of automorphisms
$(G,\tilde{\Sigma})$, let $p:(X,\Sigma)\to (Y,T)$ an affine invariant morphism
such that $\OO_Y\stackrel{\sim}{\to}(p_*\OO_X)^G$. Then the conclusions
(\ref{ajn}), (\ref{cvaj}), (\ref{draj}), (\ref{fir}) of \ref{quotaff} still hold.
\end{proposition}

\begin{proof}
For (\ref{ajn}), (\ref{cvaj}), (\ref{draj}) we may assume that $Y$ and thus $X$
are affine, and if $(B,T)$ and $(A,\Sigma)$ are their rings, the hypothesis
implies through \ref{embedcat} that $\bar{B}=\bar{A}^G$, so we can apply
\ref{quotaff} again. The statement (\ref{fir}) follows from (\ref{cvaj}) and
$\OO_Y=(p_*\OO_X)^G$.
\end{proof}

\begin{corollary}
With assumptions of \ref{propadmissibl}, for every open $U$ in $Y$, $U$ is
the quotient of $p^{-1}(U)$ by $G$.
\end{corollary}
\begin{proof}
Indeed, the morphism $p^{-1}(U)\to U$ induced by $p$ satisfies the same
assumptions as $p$.
\end{proof}

\begin{corollary}
In addition to the conditions of \ref{propadmissibl}, let $X$ be a difference scheme over $Z$ and suppose the action of $G$
is by $Z$-automorphisms. Then $Y$ is again a difference scheme over $Z$.
Moreover,
$X$ is affine over $Z$ if and only if $Y$ is. If $X$ is of finite $\Sigma$-type over $Z$,
then $X$ is $\Sigma$-finite over $Y$. 
\end{corollary}
\begin{proof}
In order to prove the finiteness statements, we reduce to the case
$X=\spec^\Sigma(A)$, $Y=\spec^{\Sigma_0}(B)$, $Z=\spec^T(C)$,
where both $(A,\Sigma)$ and $(B,\Sigma_0)=A^G$ are $(C,T)$-algebras, with
$A$ of finite $\Sigma$-type over $C$. Let $a_1,\ldots,a_n$ be the 
$\Sigma$-generators of $A$ over $C$, i.e., $A=C[a_1,\ldots,a_n]_\Sigma$.
As already noted in the proof of \ref{quotaff}(\ref{ajn}), each $a_i$ is a root
of the monic polynomial $\prod_{g\in G}(t-ga_i)\in B[t]$, and let $b_i\in B$ denote the
tuple of its coefficients.

Firstly, $A$ is clearly $\Sigma$-finite over $B$, being integral and of finite
$\sigma$-type over $B$ (as it was of finite $\Sigma$-type already over $C$).

Moreover, if we let $B_0=C[b_1,\ldots,b_n]_{\Sigma_0}$, $A$ is again 
$\sigma$-finite over $B_0$, being both integral and of finite $\sigma$-type over it,
but now $B_0$ is of finite $\Sigma$-type (equivalently, $\Sigma_0$-type) over $C$.

Note that $B$ is a well-mixed $(B,\sigma)$-submodule of $A$.
Indeed, if $ba\in B$ for some $b\in B$ and $a\in A$, then
$bg(a)=g(b)g(a)=g(ba)=ba$ for every $g\in G$, so $b(g(a)-a)=0$. By well-mixedness
of $A$, we get that $\sigma b(g(a)-a)=0$. Since $\sigma b\in B$ again, 
$\sigma b\cdot a=\sigma b g(a)=g(\sigma b)g(a)=g(\sigma b\cdot a)$ for every $g\in G$
so $\sigma b\cdot a\in B$.
\end{proof}

\begin{definition}
Let $(X,\Sigma)$ be a difference scheme with a finite group of automorphisms
$(G,\tilde{\Sigma})$. We say that $G$ acts in an \emph{admissible} way, if there
exists a morphism $p:(X,\Sigma)\to(Y,T)$ with properties listed in \ref{propadmissibl}.
This implies that $(Y,T)$ is isomorphic to the quotient difference scheme
of $(X,\Sigma)/(G,\tilde{\Sigma})$.
\end{definition}

\begin{proposition}
If $(G,\tilde{\Sigma})$ acts  admissibly on $(X,\Sigma)$, so does every 
subgroup-with-operators $(H,\tilde{\Sigma})$.
\end{proposition} 

\begin{proposition}
Suppose $(G,\tilde{\Sigma})$ acts admissibly on $(X,\Sigma)$ and
suppose that $(Y,\Sigma_0)=(X,\Sigma)/(G,\tilde{\Sigma})$ is a $(Z,T)$-difference scheme. Taking a base change morphism $(Z',T')\to (Z,T)$,
denote $(X',\Sigma')=(X,\Sigma)\times_{(Z,T)}(Z',T')$, 
$(Y',\Sigma'_0)=(Y,\Sigma_0)\times_{(Z,T)}(Z',T')$, so that $(G,\tilde{\Sigma}')$
acts on $X'$ by transport of structure and $p':X'\to Y'$ is invariant.
If $Z'$ is flat over $Z$, then $p'$ satisfies the hypotheses of \ref{propadmissibl}.
Thus, $(G,\tilde{\Sigma}')$ acts on $(X',\Sigma')$ in an admissible fashion and we have that
$(X/G)\times_ZZ'\simeq (X\times_ZZ')/G$.
\end{proposition} 
\begin{proof}
We can clearly reduce to the case where $Z=Y$ and $Y$, $Y'$ affine.
It remains to prove that if $(B,T)$ is the subring of invariants of $(G,\tilde{\Sigma})$ acting on $(A,\Sigma)$, and if $(B',T')$ is a flat $(B,T)$-algebra, %
then 
$(B',T')$ is the subalgebra of invariants of 
$(A',\Sigma')=(A,\Sigma)\otimes_{(B,T)}(B',T')$, but this follows from the
fact that the exact sequence
$$
0\to B\stackrel{i}{\to} A\stackrel{j}{\to} A^{(G)},
$$
where the last term is a power of $A$ and $j(x)$ is the tuple with entries 
$s\cdot x-x$ for $s\in G$,
remains exact upon tensoring by $B'$, while the compatibility of the difference
structure is an easy exercise.
\end{proof}

\begin{proposition}
In addition to the assumptions of \ref{quotaff}, suppose $A$ is integral.
Then $\bar{\Sigma}=\Sigma/G$.
\end{proposition}
\begin{proof}
Suppose $\sigma_1,\sigma_2\in\Sigma$ with $\sigma_1\restriction B=\sigma_2\restriction B$. Consider the maps $j\circ\sigma_i:A\to K$, $i=1,2$, where
$j:A\hookrightarrow K=\mathop{\rm Fract}(A)$ is the natural inclusion.
By \cite[V \S2.2, Corollaire to Th.~2]{bourbaki-alg-comm}, there exists a $g\in G$
such that $j\circ\sigma_2=j\circ\sigma_1\circ g$, so $\sigma_2=\sigma_1\circ g$.
\end{proof}

\subsection{Decomposition and inertia groups}\label{ss:decompinert}

Suppose $(G,\tilde{\Sigma})$ acts on $(X,\Sigma)$. For $x\in X$, the
\emph{decomposition group} at $x$ is the stabiliser $G_d(x)=G_x$ of $x$.
With the notation of \ref{quotaff}, $(G_d(x),\tilde{\Sigma}_x)$ acts
on $(\kk(x),\tilde{\Sigma}^x)$ and the \emph{inertia group} $G_i(x)$ at $x$
is the set of elements of $G_d(x)$ which act trivially on $\kk(x)$.

Assume that $(G,\tilde{\Sigma})$ acts admissibly on $(X,\Sigma)$ and that
$(Y,\Sigma_0)=(X,\Sigma)/(G,\tilde{\Sigma})$ is a $(Z,T)$-difference scheme.
Fix a $z\in Z$ and an embedding of $(\kk(z),\tau^z)$ into a model 
$(\Omega,\omega)$ of ACFA, whose saturation is sufficient to accommodate all
extensions $\kk(x)/\kk(z)$ where $x\in X$ is above $z$. We can consider
$\spec^\omega(\Omega)$ as a $(Z,T)$-difference scheme and
$(\Omega,\omega)$-valued points of $X$ correspond to 
$(\kk(z),\tau^z)$-algebra homomorphisms $(\kk(x),\Sigma^x)\to(\Omega,\omega)$,
where $x$ is a point of $X$ above $z$ (since $\Omega$ is large enough, every
point $x$ above $z$ is a locus of some point in $X(\Omega,\omega)$).

Thus we deduce a natural map $X(\Omega,\omega)\to Y(\Omega,\omega)$
which is invariant under the action of $G$ on $X(\Omega,\omega)$.
By the conclusions (\ref{cvaj}), (\ref{draj}) of \ref{quotaff}, this map
is surjective and $Y(\Omega,\omega)\simeq X(\Omega,\omega)/G$.
Moreover, if $x$ is the locus of $a\in X(\Omega,\omega)$, the stabiliser of
$a$ in $G$ is exactly the inertia group $G_i(x)$.

Bearing this in mind, in case $G_i(x)=(e)$ it makes sense to define the 
\emph{local $\omega$-substitution $\omega_a$ at $a$} as 
the unique element $\omega_a\in\Sigma_x$ satisfying
$$
\omega_aa=a\omega.
$$
In other words, $\omega_a$ is the element of $\Sigma_x$ corresponding (via
the conclusion of \ref{quotaff}(\ref{draj}))
to the image $\omega^a$ of $\omega$ by the morphism of difference structure
$()^a:\{\omega\}\to \Sigma^x$.
Note, if $\pi(a)=\pi(a')=b\in Y(\Omega,\omega)$, there exists a $g\in G$
with $a'=ga$ and
$$
\omega_aa=a\omega=g^{-1}a'\omega=g^{-1}\omega_{a'}a'=g^{-1}\omega_{a'}ga,
$$
so we conclude that $\omega_a$ and $\omega_{a'}$ are $G$-conjugate.
Therefore, 
we can define the \emph{local $\omega$-substitution at $b$} as the
$G$-conjugacy class (equivalently, the $\Sigma$-conjugacy class) 
$\omega_b=\omega_b^{X/Y}$ of any $\omega_a$ with $\pi(a)=b$.  With hindsight, all of this holds even without the assumption of saturation or 
largeness of $\Omega$, and when $b\in Y(\Omega,\omega)$ is already given, we can even discuss $\omega_b$ when $\Omega$ is just algebraically closed since this
alone already guarantees the existence of some $a\in X(\Omega,\omega)$ above 
$b$.

Equivalently, if we fix a section $\Sigma_0\to \Sigma$, %
writing $\tilde{\sigma}\in\Sigma$ for the image of $\omega^b\in\Sigma_0$, 
we could consider
as the relevant part of the datum for the local $\omega$-substitution the
unique element $\dot{\omega}_a\in G_d(x)$ with the property 
$\omega_a=\dot{\omega}_a\tilde{\sigma}$. As explained above, for 
$\pi(a)=\pi(a')=b$, 
there is a $g\in G$ such that 
$$
\dot{\omega}_a\tilde{\sigma}=g^{-1}\dot{\omega}_{a'}\tilde{\sigma}g=
g^{-1}\dot{\omega}_{a'}g^{\tilde{\sigma}}\tilde{\sigma},
$$ 
so 
we can define
$\dot{\omega}_b$ as the $()^{\tilde{\sigma}}$-conjugacy class in $G$ of
any $\dot{\omega}_a$ with $\pi(a)=b$.

In the special case when $\Omega=\bar{k}$ is the algebraic closure 
of a finite field $k$ and $\omega=\varphi_k$ is the Frobenius automorphism
generating $\Gal(\bar{k}/k)$, and $a\in X(\bar{k},\varphi_k)$ is a 
$(\bar{k},\varphi_k)$-rational point mapping onto $b$ in $Y$, 
we obtain the notion of the 
\emph{local Frobenius substitution} at $b$, denoted by $\varphi_{k,b}$ when
considered as a conjugacy class in $\Sigma$, or $\dot{\varphi}_{k,b}$
when considered as a twisted conjugacy class in $G$.

From now on, a morphism satisfying the equivalent conditions of the Corollary
will be called a \emph{Galois covering} of $(X,\Sigma)/G$ with group 
$(G,\tilde{\Sigma})$. For the purposes of this paper, a 
Galois covering will be called \emph{\'etale} if all the inertia groups are trivial. 
It is shown in \cite[\ref{trivin-et},\ref{et-trivin}]{ive-tgs} that the adjective `\'etale'
is justified.
Note that for \'etale Galois coverings, the notion of local substitutions
is well defined, as discussed above.

\section{A twisted theorem of Chebotarev}\label{ch:ttCh}

\begin{notation}\label{fig1}
A recurring situation in this chapter is a morphism of two (families of) Galois
coverings, and it is convenient to fix the notation for use throughout the chapter.
In the diagram
$$
 \begin{tikzpicture} 
 [cross line/.style={preaction={draw=white, -,
line width=3pt}}]
\matrix(m)[matrix of math nodes, minimum size=1.7em, nodes={circle},
inner sep=0pt,
row sep=1em, column sep=.75em, text height=1.5ex, text depth=0.25ex]
 { 
&&& [-1em] &  |(uz)|{Z_s} &                  &  |(uw)|{W_s} \\[2em]
&&&&                 |(ux)|{X_s} &			 & |(uy)| {Y_s} \\[.3em]
&&&&                               &  |(us)|{\tilde{s}}      &    		   \\[-7em]
          |(z)|{Z} &                  &  |(w)|{W}  & 		     & &&&\\[2em]
          |(x)|{X} &			 & |(y)| {Y}    &		     & &&&\\[.3em]
                      & |(s)|{S}      &    		    &		     & &&&\\};
\path[->,font=\scriptsize,>=to, thin]
(uz) edge (z) edge (ux) edge (uw)
(ux) edge node[above]{$f_s$} (uy) edge node[below=-1pt,pos=0.05]{$p_s$} (us) 
(uw) edge (uy)
(uy) edge node [right]{$q_s$} (us)
(ux) edge node[above]{$r_s$}(x)
(z) edge (x) edge (w)
(x) edge node[above=-1pt,pos=0.7]{$f$} (y) edge node[left]{$p$} (s) 
(w) edge[cross line] (y)
(y) edge node[left,pos=0.3]{$q$} (s)
(uw) edge [cross line](w)
(uy) edge [cross line] node [below, pos=0.75]{$t_s$}(y)
(us) edge node[below]{$s$} (s)
;
\end{tikzpicture} 
$$
we assume that $(Z,\Sigma)\to (X,\sigma)$ and $(W,T)\to (Y,\sigma)$ are
Galois coverings of $(S,\sigma)$-difference schemes. 
For a given point $s\in S$, we consider 
$\tilde{s}=\spec(\kk(s),\sigma^s)$ and, by slightly abusing the notation, we
denote by $s$ the natural morphism $\tilde{s}\to (S,\sigma)$, and we let 
$(X_s,\sigma_s)=(X,\sigma)\times_{(S,\sigma)}\tilde{s}$,
$(Y_s,\sigma_s)=(Y,\sigma)\times_{(S,\sigma)}\tilde{s}$,
$(Z_s,\Sigma_s)=(Z,\Sigma)\times_{(S,\sigma)}\tilde{s}$,
$(W_s,T_s)=(W,T)\times_{(S,\sigma)}\tilde{s}$ be the relevant fibres above $s$.
\end{notation}

\subsection{Hrushovski's twisted Lang-Weil estimate}
\begin{theorem}[Hrushovski, \cite{udi}]\label{udiLW}
Let $(S,\sigma)$ be a normal connected difference scheme of finite $\sigma$-type
over $\Z$, and let $(X,\sigma)\to (S,\sigma)$ be a morphism of finite transformal type with geometrically
transformally integral fibres of finite %
total dimension $d$, %
limit degree $\delta$ and purely inseparable dual degree $\iota'$, and
let $\mu=\delta/\iota'$. There is a constant $C>0$ and a localisation $(S',\sigma)$
of $(S,\sigma)$ such that for every closed $s\in S'$, and every finite field $k$
with $(\bar{k},\varphi_k)$ extending $(\kk(s),\sigma^s)$,
$$
\left|
|(X_s(\bar{k},\varphi_k)|-\mu|k|^d
\right|
<C|k|^{d-1/2}.
$$
\end{theorem} 

\begin{corollary}\label{geomgaleq}
Let $(S,\sigma)$ be a normal connected difference scheme of finite $\sigma$-type over $\Z$, and 
let $(Z,\Sigma)\to(X,\sigma)$ be an \'etale Galois covering with Galois group 
$(G,\tilde{\Sigma})$ of  normal 
$(S,\sigma)$-difference schemes of finite transformal type such that the fibres of
$(Z,\Sigma)\to (S,\sigma)$ are geometrically transformally 
integral of total dimension $d$.
There is a constant $C>0$ and a localisation $(S',\sigma)$
of $(S,\sigma)$ such that for every closed $s\in S'$, and every finite field $k$
with $(\bar{k},\varphi_k)$ extending $(\kk(s),\sigma^s)$, and every $\tau\in\Sigma$,
$$
\left|
|(Z^\tau_s(\bar{k},\varphi_k)|-|X_s(\bar{k},\varphi_k)|
\right|
<C|k|^{d-1/2}.
$$
\end{corollary}
\begin{proof}
Let us assume for simplicity that the purely inseparable dual degree $\iota'=1$.
Since we only need an estimate, we can immediately reduce to the affine case.
Let $Z=\spec^\Sigma(A)$, $B=A^G$, $\Sigma=G\tilde{\sigma}$ for some
fixed $\sigma\in\Sigma$, and $X=\spec^\sigma(B)$. Since 
$$
\sum_{\tau\in\Sigma}|Z_s^\tau(\bar{k},\varphi_k)|=|Z_s(\bar{k},\varphi_k)|=
|G||X_s(\bar{k},\varphi_k)|,
$$
the result will follow if we can show that for any $\tau,\tau'\in\Sigma$,
$$
|Z_s^\tau(\bar{k},\varphi_k)|\approx|Z_s^{\tau'}(\bar{k},\varphi_k)|,
$$
up to $O(|k|^{d-1/2})$.
In view of \ref{udiLW}, it suffices to show that 
the relative limit degree of $(A,g\tilde{\sigma})$ over $(S,\sigma)$ 
does not depend on $g\in G$. Since we are allowed to localise the base $S$,
we can reduce to the consideration of the limit degree of the associated
function fields. The situation is reminiscent of \ref{fulllimdeg}, but we find it
informative to proceed with the proof in the present context.
Write $(F,\sigma)$, $(K,\sigma)$, $(L,\Sigma)$ for the
function fields of $(S,\sigma)$, $(X,\sigma)$ and $(Z,\Sigma)$, respectively.
We are given that $\Gal(L/K)=G$ and $\Sigma=G\tilde{\sigma}$ and
for any $g\in G$, $g\tilde{\sigma}=\tilde{\sigma}g^{\tilde{\sigma}}$.   
Suppose $L=K(\alpha)$, and let $\bar{\alpha}=\{h\alpha:h\in G\}$.
Let $\bar{\beta}$ be a tuple of $\sigma$-generators of $K$ over $F$, $K=F(\bar{\beta})_\sigma$.
Then $L=F(\bar{\alpha},\bar{\beta})_{\tilde{\sigma}}$, and since the
definition of limit degree does not depend on the choice of generators by 
\ref{def-limdeg},
$$
\dl((L,g\tilde{\sigma})/(F,\sigma))=\lim_i[L_{i+1}^{g\tilde{\sigma}}:L_i^{g\tilde{\sigma}}],
$$
where 
$$
L_i^{g\tilde{\sigma}}=F(\bar{\alpha},\bar{\beta},(g\tilde{\sigma})\bar{\alpha},
(g\tilde{\sigma})\bar{\beta},\ldots,
(g\tilde{\sigma})^i\bar{\alpha},(g\tilde{\sigma})^i\bar{\beta})
=
F(\bar{\alpha},\bar{\beta},(g\tilde{\sigma})\bar{\alpha},\sigma\bar{\beta},\ldots,
(g\tilde{\sigma})^i\bar{\alpha},\sigma^i\bar{\beta}).
$$
On the other hand, 
$$
g\tilde{\sigma}\bar{\alpha}=\{g\tilde{\sigma}h\alpha:h\in G\}=
\{\tilde{\sigma}(g^{\tilde{\sigma}}h)\alpha:h\in G\}=\tilde{\sigma}\bar{\alpha},
$$
So $L_i^{\tilde{\sigma}}=L_i^{g\tilde{\sigma}}$ for any $g\in G$ and
$\dl((L,\tilde{\sigma})/(F,\sigma))=\dl((L,g\tilde{\sigma})/(F,\sigma))$ for any $g\in G$.
\end{proof}

\subsection{Central functions on difference structures}

\begin{definition}
Let $\Sigma$ be an object of $\diff$.
\begin{enumerate}
\item
A function $\alpha:\Sigma\to\C$ is called \emph{central} if 
$\alpha(\sigma^\tau)=\alpha(\sigma)$ for every $\sigma,\tau\in\Sigma$.
\item
A subset of  $C$ of $\Sigma$ is called a \emph{conjugacy domain}
when $\sigma\in C$ if and only if $\sigma^\tau\in C$ for all $\sigma,\tau\in\Sigma$.
\item
We denote by $\cC(\Sigma)$ the algebra of all central functions on $\Sigma$.
\end{enumerate}
\end{definition}
\begin{example}
A \emph{representation} of a $\diff$-object is an inversive object $(V,\Sigma)$
in the category of vector spaces over $\C$. Then the morphism
$$
\sigma\mapsto\mathop{\rm tr}(\sigma|V)
$$
is a central function on $\Sigma$.
\end{example}

It is a trivial but useful observation that $\cC(\Sigma)$ is spanned by the
characteristic functions of conjugacy domains in $\Sigma$.
When the underlying set of $\Sigma$ is finite, we can equip 
$\cC(\Sigma)$ with an inner product as follows.
\begin{definition}
Let $\alpha,\gamma:\Sigma\to\C$ be central functions with $\Sigma$ finite. 
Their \emph{inner product} is defined as
$$
(\alpha, \gamma)_\Sigma=\frac{1}{|\Sigma|}\sum_{\sigma\in\Sigma}\alpha(\sigma)\overline{\gamma}(\sigma),
$$
where $\overline{z}$ denotes the complex conjugate of $z$.
\end{definition}

We can define the pullbacks and pushforwards of central functions along
$\diff$-morphisms which are analogous to the classical operations of restriction and (generalised) induction of characters.
\begin{definition}\label{defpushpull}
Let $\psi:\Sigma\to T$ be a $\diff$-morphism of finite difference structures, and let $\alpha:\Sigma\to\C$ and
$\beta:T\to\C$ be central functions.
\begin{enumerate}
\item The \emph{pullback} of $\beta$ along $\psi$ is the central function 
$$
\psi^*\beta=\beta\circ\psi.
$$
\item If $\psi:\Sigma\hookrightarrow T$ is injective, we let the \emph{pushforward} of
$\alpha$ along $\psi$ be the central function
$$
\psi_*\alpha(\tau)=\frac{1}{|\Sigma|}\sum_{\substack{\rho\in T\\ \tau^\rho\in\Sigma}}
\alpha(\tau^\rho),
$$
for $\tau\in T$.
\item If $\psi:\Sigma\rightarrow T$ is surjective, we let the \emph{pushforward} of
$\alpha$ along $\psi$ be the central function
$$
\psi_*\alpha(\tau)=\frac{1}{|\psi^{-1}(\tau)|}\sum_{\sigma\in\psi^{-1}(\tau)}\alpha(\sigma),
$$
for $\tau\in T$.
\item An arbitrary $\psi:\Sigma\to T$ decomposes as $\psi=\psi'\circ\psi''$, with
$\psi''$ a surjection and $\psi'$ an injection, so we can define
$$
\psi_*\alpha=\psi'_*\psi''_*\alpha.
$$
\end{enumerate}
\end{definition}

\begin{lemma}\label{pushpullfunct}
The operation ${}^*$ defines a contravariant functor from the category of finite
difference structures to the category of (inner product) algebras,
$$
\Sigma\stackrel{\psi}{\to}T\ \ \mapsto\ \ \cC(T)\stackrel{\psi^*}{\to}\cC(\Sigma),
$$
while ${}_*$ defines a covariant functor from the category of finite
difference structures to the category of (inner product) vector spaces,
$$
\Sigma\stackrel{\psi}{\to}T\ \ \mapsto\ \ \cC(\Sigma)\stackrel{\psi_*}{\to}\cC(T).
$$
To explicate the functoriality:
\newcounter{tempcounter}
\begin{enumerate}
\item $(\phi\circ\psi)^*=\psi^*\phi^*$ and $\mathop{\rm id}^*=\mathop{\rm id}$;
\item $(\phi\circ\psi)_*=\phi_*\psi_*$ and $\mathop{\rm id}_*=\mathop{\rm id}$.
 \setcounter{tempcounter}{\value{enumi}}
\end{enumerate}
Moreover, there is a \emph{projection formula}
\begin{enumerate}
  \setcounter{enumi}{\value{tempcounter}}
\item 
$
\psi_*(\alpha\cdot\psi^*\beta)=\psi_*\alpha\cdot\beta,
$
\end{enumerate}
for any  $\alpha\in\cC(\Sigma)$ and
$\beta\in\cC(T)$.
\end{lemma}

\begin{proof}
The only statement which requires checking due to our unconventional framework
is the projection formula. In view of the above functoriality, it suffices to check
it separately for cases where $\psi$ is injective or surjective. 

When $\psi$
is injective, identifying $\Sigma\subseteq T$ and using the fact that $\beta$ is central,
$$
\psi_*(\alpha\cdot\psi^*\beta)(\tau)=
\frac{1}{|\Sigma|}\sum_{\substack{\rho\in T\\ \tau^\rho\in\Sigma}}\alpha(\tau^\rho)
\beta(\tau^\rho)=
\frac{1}{|\Sigma|}\sum_{\substack{\rho\in T\\ \tau^\rho\in\Sigma}}\alpha(\tau^\rho)
\beta(\tau)=
\psi_*\alpha(\tau)\beta(\tau).
$$
When $\psi$ is surjective, the verification is entirely trivial. Indeed,
\begin{multline*}
\psi_*(\alpha\cdot\psi^*\beta)(\tau)=
\frac{1}{|\psi^{-1}(\tau)|}\sum_{\sigma\in\psi^{-1}(\tau)}
\alpha(\sigma)\beta(\psi(\sigma))\\
=\frac{1}{|\psi^{-1}(\tau)|}\sum_{\sigma\in\psi^{-1}(\tau)}
\alpha(\sigma)\beta(\tau)=\psi_*\alpha(\tau)\beta(\tau).
\end{multline*}
\end{proof}

\begin{lemma}[Base change]\label{repbasech}
Suppose we have a Cartesian diagram
$$
\begin{tikzpicture} 
\matrix(m)[matrix of math nodes, row sep=1.5em, column sep=-.3em] %
 {                       & |(P)|{\Sigma_1\times_T\Sigma_2} &           \\
 |(1)|{\Sigma_1} &                         & |(2)| {\Sigma_2}\\
                & |(h)|{T}            &\\}; 
\path[->,font=\scriptsize,>=to, thin]
(P) edge node[left,pos=0.35]{$\pi_1$} (1) edge node[right,pos=0.35]{$\pi_2$} (2)
(1) edge node[left,pos=0.65]{$\psi_1$}  (h)
(2) edge node[right,pos=0.65]{$\psi_2$} (h);
\end{tikzpicture}
$$
of surjective $\diff$-morphisms, and let $\alpha$ be a central function on $\Sigma$.
Then
$$
\pi_{2*}\pi_1^*\alpha=\psi_2^*\psi_{1*}\alpha.
$$
\end{lemma}
A proof is obtained by a direct calculation using nothing but \ref{defpushpull}.
\begin{proposition}[Frobenius reciprocity]\label{frobrecip}
Let $\psi:\Sigma\to T$ be a $\diff$-morphism of finite difference structures.
Assume that either
\begin{enumerate}
\item\label{oneinj} $\psi$ is injective and every map $()^\tau:T\to T$ is injective (equivalently, bijective), or
\item\label{twosurj} $\psi$ is surjective with all fibres of size $c$. 
\end{enumerate}
Then we have an `adjunction'
$$
(\alpha,\psi^*\beta)_\Sigma=(\psi_*\alpha,\beta)_T,
$$
for any two central functions $\alpha:\Sigma\to\C$ and
$\beta:T\to\C$.
\end{proposition}
\begin{proof}
In case (\ref{oneinj}), 
\begin{equation*}
\begin{split}
(\psi_*\alpha,\beta)_T & =\frac{1}{|T|}\sum_{\tau\in T}\psi_*\alpha(\tau)\overline{\beta}(\tau)=
\frac{1}{|T|}\sum_{\tau\in T}\frac{1}{|\Sigma|}
\sum_{\substack{\rho\in T\\ \tau^\rho\in\Sigma}}\alpha(\tau^\rho)\overline{\beta}(\tau)\\
& = \frac{1}{|\Sigma|}
\sum_{\substack{\tau, \rho\in T\\ \tau^\rho\in\Sigma}}\frac{1}{|T|}\alpha(\tau^\rho)\overline{\beta}(\tau^\rho)\stackrel{(\dag)}{=}
\frac{1}{|\Sigma|}\sum_{\sigma\in\Sigma}\alpha(\sigma)\overline{\beta}(\sigma)=
(\alpha,\psi^*\beta)_\Sigma,
\end{split}
\end{equation*}
where the equality $(\dag)$ follows from the fact that, for a fixed $\sigma\in\Sigma$, there are exactly $|T|$
pairs $(\tau,\rho)\in T\times T$ with $\tau^\rho=\sigma$.

Regarding (\ref{twosurj}), 
\begin{equation*}
\begin{split}
(\psi_*\alpha,\beta)_T & =\frac{1}{|T|}\sum_{\tau\in T}\psi_*\alpha(\tau)\overline{\beta}(\tau)=
\frac{1}{|T|}\sum_{\tau\in T}
\frac{1}{|\psi^{-1}(\tau)|}
\sum_{\sigma\in\psi^{-1}(\tau)}\alpha(\sigma)\overline{\beta}(\tau)\\
& \stackrel{(\ddag)}{=} \frac{1}{|T|}\frac{1}{c}
\sum_{\sigma\in\Sigma}\alpha(\sigma)\overline{\beta}(\psi(\sigma))=
\frac{1}{|\Sigma|}\sum_{\sigma\in\Sigma}\alpha(\sigma)\overline{\psi^*\beta}(\sigma)=
(\alpha,\psi^*\beta)_\Sigma,
\end{split}
\end{equation*}
where the equality $(\ddag)$ follows from the assumption on constant fibre size.
\end{proof}
Since the statement is compatible with composites of structure maps via
\ref{pushpullfunct}, we have the following.
\begin{corollary}
The Frobenius reciprocity 
$$
(\alpha,\psi^*\beta)_\Sigma=(\psi_*\alpha,\beta)_T,
$$
holds for any $\psi$ which is a composite of maps satisfying (\ref{oneinj})
or (\ref{twosurj}) from \ref{frobrecip}.
\end{corollary}

\begin{remark}
Suppose that
there exists a finite group $G$ acting faithfully on $\Sigma$ such that
$\psi$ can be identified with the canonical projection $\Sigma\to\Sigma/G=T$. 
Then the assumption (\ref{twosurj}) from \ref{frobrecip} is satisfied.
\end{remark}

\subsection{Constructible functions on difference schemes}

\begin{definition}\label{def:basic-constr}
With notation of \ref{fig1}, let $(Z,\Sigma)\to (X,T)$ be an \'etale Galois covering of $(S,\sigma)$-difference schemes
with group $(G,\tilde{\Sigma})$ (such that $\Sigma/G=T$) and let 
$\alpha:\Sigma\to\C$ be a central function. 
We shall say that a pair $(Z/X,\alpha)$ defines a
\emph{basic $(S,\sigma)$-constructible function}
on $X$, in the sense that for any $s\in S$ and every  $(F,\varphi)$ extending 
$(\kk(s),\sigma^s)$ with $F$ algebraically closed,
we obtain an actual function 
$$
\alpha_{s,(F,\varphi)}:X_s(F,\varphi)\to \C, \ \ \ x\mapsto \alpha(\varphi_{r_s(x)}).
$$
We reserve the possibility of writing the last term in an oversimplified manner
as $\alpha(\varphi_x)$.
\end{definition}

\begin{definition}
Let $(X,\sigma)$ be an $(S,\sigma)$-difference scheme. 
A \emph{constructible function} 
$$
\alpha=\langle X,Z_i/X_i,\alpha_i\ |\ i\in I\rangle
$$
on $(X,\sigma)$ is defined by a partition of $(X,\sigma)$ into a finite set
of integral normal locally closed difference $(S,\sigma)$-subschemes 
$(X_i,\sigma_i)$ of $(X,\sigma)$, each equipped with an \'etale Galois
covering $(Z_i,\Sigma_i)\to (X_i,\sigma)$ with group $(G_i,\tilde{\Sigma}_i)$,
and a central function $\alpha_i:\Sigma_i\to\C$.
In other words, $\alpha$ is determined by a normal stratification of $X$
and a basic constructible function on each stratum. 

Accordingly, for each $s\in S$ and algebraically closed $(F,\varphi)$
extending $(\kk(s),\sigma^s)$, we obtain a function, which we dub the
\emph{$(F,\varphi)$-realisation} of $\alpha$, 
$$
\alpha_{s,(F,\varphi)}:X_s(F,\varphi)\to\C, \ \ \ \alpha_{s,(F,\varphi)}\restriction_{X_i(F,\varphi)}=\alpha_{i,s,(F,\varphi)}.
$$
\end{definition}

\begin{lemma}\label{geomequid}
Let $(S,\sigma)$ be a normal connected difference scheme of finite $\sigma$-type over $\Z$, and 
let $f:(Z,\Sigma)\to(X,\sigma)$ be an \'etale Galois covering with Galois group 
$(G,\tilde{\Sigma})$ of  normal connected 
$(S,\sigma)$-difference schemes of finite transformal type
such that the fibres of $(Z,\Sigma)\to (S,\sigma)$
are geometrically transformally integral of total dimension $d$. 
Let $\alpha$ be a $(S,\sigma)$-constructible function on $X$ associated
with the covering $Z/X$. There is a constant $C>0$ and a localisation
$S'$ of $S$ such that for any closed $s\in S'$, and any finite field $k$ with
$(\bar{k},\varphi_k)$ extending $(\kk(s),\sigma^s)$,
$$
\left|
\sum_{x\in X(\bar{k},\varphi_k)}\alpha(\varphi_{k,x})-
|X(\bar{k},\varphi_k)|\sum_{\sigma\in\Sigma}\alpha(\sigma)
\right|
\leq C|k|^{d-1/2}.
$$
\end{lemma}
\begin{proof}
This is straightforward using \ref{geomgaleq}.
\end{proof}

\begin{definition}[Inflation]\label{def:inflat}
Let  $\alpha$ be an $(S,\sigma)$-constructible function on $(X,\sigma)$
given by $\langle X,(Z_i,\Sigma_i)/(X_i,\sigma),\alpha_i\rangle$ and
suppose that for each $i$ we have an \'etale Galois $(S,\sigma)$-covering
$(Z'_i,\Sigma')\to (X_i,\sigma)$ which
dominates  $Z_i/X_i$. Thus, we have a diagram (over $(S,\sigma)$)
$$
 \begin{tikzpicture} 
\matrix(m)[matrix of math nodes, row sep=1.8em, column sep=1em, text height=1.2ex, text depth=0.25ex]
 {
 |(11)|{Z'_i} & 		& |(12)|{Z_i}\\
			 & |(22)|{X_i} &	\\};
\path[->,font=\scriptsize,>=to, thin]
(11) edge (12) edge(22)
(12) edge (22);
\end{tikzpicture}
$$
and let us denote by $\pi_i:\Sigma'_i\to\Sigma_i$ the $\diff$-morphism part of
$(Z'_i,\Sigma'_i)\to(Z_i,\Sigma_i)$.
 
 The \emph{inflation} of $\alpha$ with respect to this data
  is defined as
$$
\alpha'=\langle X,Z'_i/X_i,\pi_i^*\alpha_i\rangle.
$$ 
\end{definition}
\begin{remark}
With the above notation, it is clear that  $\alpha'$ and
$\alpha$ can be thought of as the same $(S,\sigma)$-constructible function  on $X$,
since for every $s\in (S,\sigma)$ and every algebraically closed $(F,\varphi)$
extending $(\kk(s),\sigma^s)$,
$$
\alpha'_{s,(F,\varphi)}=\alpha_{s,(F,\varphi)},
$$
as functions $X_s(F,\varphi)\to\C$.
\end{remark}
 
\begin{definition}[Refinement]\label{def:refin}
Let $\alpha=\langle X,(Z_i,\Sigma_i)/(X_i,\sigma),\alpha_i\ |\ i\in I\rangle$ be an
$(S,\sigma)$-constructible function
on $(X,\sigma)$ and assume we have a further stratification of each $X_i$
into finitely many integral, normal, locally closed $(S,\sigma)$-subschemes
 $X_{ij}$.
For each $i,j$, let $Z_{ij}$ be a connected component of 
$(Z_i,\Sigma_i)\times_{(X_{i},\sigma)}(X_{ij},\sigma)$, and let
$D_{ij}$ be its decomposition subgroup in $Z_i/X_i$. Moreover, let
$\Sigma_{ij}=\{\sigma\in\Sigma_i:\sigma Z_{ij}=Z_{ij}\}$. Then 
$(Z_{ij},T_{ij})\to(X_{ij},\sigma)$ is a Galois covering over $(S,\sigma)$ with group 
$(D_{ij},\tilde{\Sigma}_{ij})$ and denote $\iota_{ij}:\Sigma_{ij}\hookrightarrow\Sigma_i$.
We define the \emph{refinement of $\alpha$ to the stratification $\{X_{ij}\}$ of $X$} 
to be the $(S,\sigma)$-constructible function
$$
\alpha'=\langle X,Z_{ij}/X_{ij},\iota^*_{ij}\alpha_i\rangle.
$$
\end{definition}

\begin{remark}\label{rem:refin}
Again, the refinement procedure does not fundamentally change the
constructible \emph{function}, since  for each $s\in S$, and algebraically closed
$(F,\varphi)$ extending $(\kk(s),\sigma^s)$, the realisations of $\alpha$ and
$\alpha'$ are the same:
$$
\alpha'_{s,(F,\varphi)}=\alpha_{s,(F,\varphi)},
$$
as functions $X_s(F,\varphi)\to \C$.
\end{remark}

\begin{definition}[Pullback]\label{def:pullbk}
Let $f:(X,\sigma)\to (Y,\sigma)$ be a morphism, and let 
$$
\beta=\langle Y,(W_j,T_j)/(Y_j,\sigma),\beta_j\ |\ j\in J\rangle
$$ 
be a constructible function on $Y$, where each $(W_j,T_j)\to (Y_j,\sigma)$  is
an \'etale Galois covering with
group $(G_j,\tilde{T}_j)$ such that $\beta_j$ is %
a central function on $T_j$. 
Let $Z_j$ be a component of $(X,\sigma)\times_{(Y,\sigma)}{(W_j,T)}$, and let $D_{Z_j}$ be its decomposition subgroup in $W_j/Y_j$.
Moreover, let $T_{Z_j}=\{\tau\in T_j:\tau Z_j=Z_j\}$.
Then $(Z_j,T_{Z_j})\to(X_j,\sigma)=f^{-1}(Y_j)$ is a Galois covering with group 
$(D_{Z_j},\tilde{T}_{Z_j})$ and denote $\iota_j:T_{Z_j}\hookrightarrow T_j$.
Morally speaking, we would like to 
define the \emph{pullback of $\beta$ along $f$} to be
$$
\langle X,Z_j/X_j,\iota_j^*\beta_j\ |\ j\in J\rangle,$$
but the strata $X_j$ need not be normal.
Thus, to be precise, we must choose (non-canonically, using \cite[\ref{locnorm}]{ive-tgs}) a normal stratification $X_{ij}$ which refines the stratification of $X$ into $X_j$  
and we define $f^*\beta$ to be the refinement of the above data. In spite of
the non-canonical choice of a normal refinement, the notation $f^*\beta$ can
be justified via \ref{rem:refin}.
\end{definition}

\begin{remark}
A prominent feature of the pullback construction 
is that for any $x\in X_s(F,\varphi)$, 
$$
f^*\beta(\varphi_x)=\beta(\varphi_{f(x)}).
$$
\end{remark}

\begin{definition}[Algebra of constructible functions]\label{constr-algebra}
Let us denote by $\cC(X,\sigma)$ the set of $(S,\sigma)$-constructible
functions on $(X,\sigma)$. We wish to endow $\cC(X,\sigma)$ with a 
\emph{$\C$-algebra}
structure. Suppose $\alpha, \beta\in\cC(X,\sigma)$, and let
$\alpha=\langle X,Z_i/X_i,\alpha_i\rangle$, $\beta=\langle X,Z'_j/X'_j,\beta_j\rangle$.
Upon a refinement of the underlying stratifications, we may assume
that $X_i=X'_i$, and upon an inflation, we may even assume that 
$(Z_i,\Sigma)/(X_i,\sigma)$ and $(Z_i',\Sigma_i)/(X_i,\sigma)$ 
are the same Galois coverings, and $\alpha_i,\beta_i:\Sigma_i\to\C$.
Then we can define:
\begin{enumerate}
\item $\alpha+\beta=\langle X,Z_i/X_i,\alpha_i+\beta_i\rangle$, and
\item $\alpha\cdot\beta=\langle X,Z_i/X_i,\alpha_i\cdot\beta_i\rangle$.
\end{enumerate}
\end{definition}

For the sake of simplicity, we will only define the pushforward along morphisms
of a very special kind. Note that we will never use it for more general morphisms.
For simplicity of notation, the definition is given in an absolute case, but it
is clear that all operations can be performed in a relative setting, over a given
base $(S,\sigma)$.

\begin{definition}[Pushforward]\label{defpushfwd}
Let $f:(X,\sigma)\to (Y,\sigma)$ be a morphism of finite transformal type of normal connected difference
schemes with geometrically integral fibres, and let $\alpha$ be
a basic constructible function on $X$ (we leave the sorites of the definition for a general constructible function to the reader). 
Let $(Z,\Sigma)\to (X,\sigma)$ be an \'etale
Galois covering with group $(G,\tilde{\Sigma})$ such that $\alpha$
is associated with a central function on $\Sigma$.
\begin{enumerate}%
\item%
\label{geomconn}
Assume that $Z$ is connected (i.e.,\ integral) and let $(W,T)$ be the
normalisation of $(Y,\sigma)$ in the algebraic closure of $(\kk(Y),\sigma)$ in
$(\kk(Z),\Sigma)$ as depicted in the diagram
$$
 
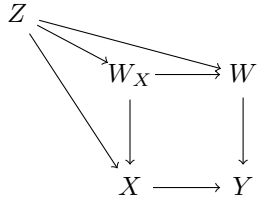
\begin{tikzpicture} 
 [cross line/.style={preaction={draw=white, -,
line width=3pt}}]
\matrix(m)[matrix of math nodes, minimum size=1.7em, nodes={circle},
inner sep=0pt,
row sep=.5em, column sep=2.5em, text height=1.5ex, text depth=0.25ex]
 { 
 	 |(3)|{Z}  &                    &\\
                       & |(P)|{W_X} & |(2)| {W}          \\[2em]
                        &|(1)|{X}                         & |(h)|{Y} \\};
\path[->,font=\scriptsize,>=to, thin]
(P) edge  (1) edge (2)
(1) edge  (h)
(2) edge  (h)
(3) edge (1) edge (2) edge (P)
;
\end{tikzpicture}
$$
where $(W_X,T_X)=(X,\sigma)\times_{(Y,\sigma)}(W,T)$.
Then $(W,T)$ is a Galois cover of $(Y,\sigma)$ with some group $(H,\tilde{T})$.
However, $W/Y$ is not necessarily \'etale, but in view of \ref{cor:norm-lattice}
we can find a finite $\sigma$-localisation $Y'$ to achieve that the corresponding
covering $(W',T)\to (Y',\sigma)$ is finite \'etale Galois. Thus, we may continue
to define $f_*\alpha$ on the stratum $Y'$ and we postpone the definition of the pushforward
by the morphism $f\restriction_{f^{-1}(Y\setminus Y')}$ for the next stage
of devissage, remarking that this morphism shares the required properties of $f$.
We henceforth relabel $W'$ and $Y'$ back to $W$ and $Y$.
 We obtain an exact sequence 
$$
1\to\Gal(Z/W_X)\to\Gal(Z/X)\to\Gal(W/Y)\to 1.
$$
Moreover, $K=\Gal(Z/W_X)$ acts faithfully on $\Sigma$ so that $T=\Sigma/K$.
If $\pi:\Sigma\to T$ denotes the difference structure part of the
morphism $(Z,\Sigma)\to (W,T)$, we define
$$
f_*\alpha:=\pi_*\alpha,
$$
as a central function on $T$ associated with the covering $W/Y$.
If $V$ is a representation of $\Sigma$, we define $f_*V=V^K$, as a representation
of $T$.

\item\label{notgeomconn} If $(Z,\Sigma)$ is not connected, 
$Z=\coprod_{i\in I}Z_i$ as a topological
space, but $\Sigma$ is transitive on the set of components $I$, we let
$D^i\leq G$ be the decomposition subgroup of $Z_i$, and let 
$\Sigma_i=\{\sigma\in\Sigma:\sigma Z_i=Z_i\}$.
Then each $(Z_i,\Sigma_i)$ is a Galois covering of $(X,\sigma)$ with group $D^i$
and maps onto $(W,T)$. Let us denote by $\alpha_i$ the restriction of $\alpha$
on $D^i$, associated with the covering $Z_i/X_i$.
We define, resorting to the previous case,
$$
f_*\alpha=\frac{1}{|I|}\sum_{i\in I}f_*\alpha_i, %
$$
as a central function on $T$ associated with the covering $W/T$.

\item\label{disconn} If $(Z,\Sigma)=\coprod_{i\in I}(Z_i,\Sigma_i)$, let $\alpha_i$ be a
restriction of $\alpha$ to $Z_i$, and we define
$$
f_*\alpha=\sum_{i\in I}f_*\alpha_i,
$$
where the sum of constructible functions is defined in terms of \ref{constr-algebra}.
\end{enumerate}
\end{definition}

\begin{remark}
The case \ref{notgeomconn} is a peculiarity of the difference framework,
since in the algebraic case one need consider only geometrically connected/disconnected dichotomy. One might argue that our definition in  
\ref{notgeomconn} ignores the full structure of $(Z,\Sigma)$
by breaking it up into $(Z_i,\Sigma_i)$, but the key underlying principle is that we are dealing
with fixed-point sets, and the structure morphisms which shuffle components too wildly cannot have fixed points.
Let us illustrate by showing a case where the definitions
\ref{geomconn} and  \ref{notgeomconn} agree.

With notation of \ref{notgeomconn}, suppose we have a representation 
$\rho:\Sigma\to GL(V)$
of $\Sigma$ such that $V=\oplus_{i\in I}V_i$, where each $V_i$ is associated
with a representation $\rho_i$
 of $\Sigma_i$. In other words, if 
$\iota:\Sigma_{i_0}\hookrightarrow\Sigma$, $V=\iota_*V_{i_0}$.

Now, for every $\sigma\in\Sigma$, 
$$
\mathop{\rm tr}(\rho(\sigma))=\sum_{\substack{i\in I \\ \rho(\sigma)V_i\subseteq V_i}}
\mathop{\rm tr}(\rho_i(\sigma))=
\sum_{\substack{i\in I \\ \sigma\in \Sigma_i}}\mathop{\rm tr}(\rho_i(\sigma)).
$$
Thus, writing $\alpha$ for the character of $\rho$ and $\alpha_i$ for the
character of $\rho_i$, we have
\begin{equation*}
\begin{split}
\pi_*\alpha(\tau) &=\frac{1}{|\pi^{-1}(\tau)|}\sum_{\sigma\in\pi^{-1}(\tau)}\alpha(\sigma)=
\frac{1}{|\pi^{-1}(\tau)|}\sum_{\sigma\in\pi^{-1}(\tau)}\sum_{\substack{i\in I\\ \sigma\in \Sigma_i}}\alpha_i(\sigma) \\
& =
\frac{1}{|\pi^{-1}(\tau)|}\sum_{i\in I}\sum_{\sigma\in\pi^{-1}(\tau)\cap \Sigma_i}\alpha_i(\sigma)=
\frac{1}{|\pi^{-1}(\tau)|}\sum_{i\in I}\sum_{\sigma\in\pi_i^{-1}(\tau)}\alpha_i(\sigma)\\
& =
\frac{1}{|\pi^{-1}(\tau)|}\sum_{i\in I}|\pi_i^{-1}(\tau)|\pi_{i*}\alpha_i(\tau)=
\frac{1}{|I|}\sum_{i\in I}\pi_{i*}\alpha_i(\tau)=f_*\alpha(\tau),
\end{split}
\end{equation*}
bearing in mind that $|\pi_i^{-1}(\tau)|/|\pi^{-1}(\tau)|=|D|/|G|=1/|I|$.

A justification for the above choice of a particular form for $\rho$
is beyond the scope
of this paper, and it can be clarified in the framework of difference \'etale sheaves
 from the forthcoming work of the present author
\cite{ive-etale}. 
Intuitively, suppose we have a locally constant \'etale sheaf $\cF$ on $(X,\sigma)$.
The scheme $(Z,\Sigma)$ can be thought as induced from $(Z_{i_0},\Sigma_{i_0})$ via $\iota:\Sigma_{i_0}\to\Sigma$. A consequence of the appropriate sheaf
condition is that 
$\cF(Z)=\cF(\iota_*Z_{i_0})=\iota_*\cF(Z_{i_0})$, where $V_{i_0}$ above should be thought
as the value $\cF(Z_{i_0})$ and $V$ as the value $\cF(Z)$.
\end{remark}

\begin{proposition}[Base change]\label{BC}
Ket $f:(X,\sigma)\to(Y,\sigma)$ be a morphism of finite transformal type 
of normal connected  difference
schemes with geometrically integral fibres. Let $y$ be a point of $Y$, and
let $X_y$ be a fibre over $y$. Writing $y:\spec(\kk(y),\sigma^y)\to (Y,\sigma)$
and $r_y:(X_y,\sigma_y)\to (X,\sigma)$,
for any constructible
function $\alpha$ on $(X,\sigma)$,
$$
f_{y*}r_y^*\alpha=y^*f_*\alpha.
$$
\end{proposition}
\begin{proof}
We may assume that $\alpha$ is basic, associated with an \'etale Galois covering
$(Z,\Sigma)\to (X,\sigma)$ and let $(W,T)$ be as in \ref{defpushfwd}, and we
may also assume that $y$ does not fall in the ramification locus of $W/Y$,
so without any loss, we shall continue with the assumption that $W/Y$ is finite
\'etale Galois. 
By the compatibility of the statement with the case \ref{notgeomconn} of \ref{defpushfwd},
we may further reduce to the case
where $Z$ is %
integral over $S$. Let $\tilde{y}=\spec^{\sigma^y}(\kk(y))$.
Choose an $y'\in (W_y,T_y)=(W,T)\times_{(Y,\sigma)}\tilde{y}$, let
$\tilde{y}'=\spec^{\Sigma^{y'}}(\kk(y'))$ and let
 $(X_{y},\sigma_y)=(X,\sigma)\times_{(Y,\sigma)}\tilde{y}$ be the fibre of $X$ over $y$. Moreover, let $(X_W,T_{X_W})=(X,\sigma)\times_{(Y,\sigma)}(W,T)$ and 
let $(X_{y'},T_{y'})=(X_W,T_{X_W})\times_{(W,T)}\tilde{y}'$ and 
$(Z_{y'},\Sigma_{y'})=(Z,\Sigma)\times_{(W,T)}\tilde{y}'$ denote the fibres of $X_W$ and
$Z$ above $y'$.
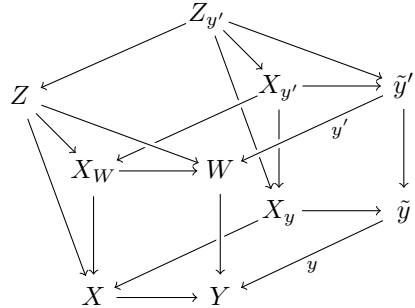
$$
 \begin{tikzpicture} 
 [cross line/.style={preaction={draw=white, -,
line width=3pt}}]
\matrix(m)[matrix of math nodes, minimum size=1.7em, nodes={circle},
inner sep=0pt,
row sep=1em, column sep=1em, text height=1.5ex, text depth=0.25ex]
 { 
&&[2em]&[-3.2em] 	 |(u3)|{Z_{y'}} &  & [2em] \\
&&&                       & |(uP)|{X_{y'}} &    |(u2)| {\tilde{y}'}      \\[2em]
&&&                       &  |(u1)|{X_y} &    |(uh)|{\tilde{y}}           \\  [-7em]
 	 |(3)|{Z}  & &&&&\\
                       & |(P)|{X_W} & |(2)| {W}          &&&&\\[2em]
                        &|(1)|{X}                         & |(h)|{Y} &&&&\\};
\path[->,font=\scriptsize,>=to, thin]
(uP) edge (u1) edge (u2)
(u3) edge (u1) edge (u2) edge (uP)
(u2) edge (uh)
(P) edge  (1) edge (2)
(1) edge  (h)
(u3) edge (3)
(uP) edge[cross line] (P)
(u1) edge (1)
(u2) edge[cross line] node[below=-2pt,pos=0.3]{$y'$}(2)
(uh) edge node[below=-1pt]{$y$}(h)
(u1) edge[cross line] (uh)
(2) edge[cross line]  (h)
(3) edge (1) edge[cross line] (2) edge (P)
;
\end{tikzpicture}
$$

By our assumptions, $X_y$ and $Z_{y'}$ are geometrically integral
and we obtain a morphism of exact sequences
\begin{center}\begin{tikzpicture}
\matrix(m)[matrix of math nodes, row sep=2.4em, column sep=1.5em, text height=1.5ex, text depth=0.25ex]
{
|(l1)|{1} &|(l2)|{\Gal(Z_{y'}/X_{y'})}& |(l3)|{{\Gal}(Z_{y'}/X_y)} &|(l4)|{{\Gal}(\kk(y')/\kk(y))}&|(l5)|{1}\\
|(1)|{1} & |(2)| {\Gal(Z/X_W)}& |(3)| {{\Gal}(Z/X)} & |(4)|{{\Gal}(W/S)} &|(5)|{1}\\
};
\path[->,font=\scriptsize,>=to, thin]
(1) edge (2)
(2) edge (3)
(3) edge (4)
(4) edge (5)
(l1)edge(l2)
(l2)edge(l3)
(l3)edge(l4)
(l4)edge(l5)
(l2)edge(2)
(l3)edge(3)
(l4)edge(4);
\end{tikzpicture}\end{center}
where the left vertical arrow is an isomorphism. A bit of diagram chasing
yields that $\Gal(Z_{y'}/X_y)\simeq\Gal(Z/X)\times_{\Gal(W/S)}\Gal(\kk(y')/\kk(y))$.
Consequently we obtain a Cartesian diagram of difference structures
$$
\begin{tikzpicture} 
\matrix(m)[matrix of math nodes, row sep=1.5em, column sep=1.4em,text height=1.3ex, text depth=0.25ex]
 {                       & |(P)|{\Sigma_{\smash{\mathrlap{y'}{}}}} &           \\
 |(1)|{\Sigma} &                         & |(2)| {\Sigma^{\smash{\mathrlap{y'}{}}}}\\
                & |(h)|{T}            &\\}; 
\path[->,font=\scriptsize,>=to, thin]
(P) edge  (1) edge (2)
(1) edge  (h)
(2) edge  (h);
\end{tikzpicture}
$$
so the result follows from \ref{repbasech}.
\end{proof}
\begin{corollary}[Uniform base change]\label{unfiBC}
With notation from \ref{fig1}, let $(S,\sigma)$ be a normal connected difference scheme and 
let $f:(X,\sigma)\to(Y,\sigma)$ be a morphism of finite transformal type of normal connected  difference
$(S,\sigma)$-schemes with geometrically integral fibres.
For any  $(S,\sigma)$-constructible
function $\alpha$ on $(X,\sigma)$,
$$
f_{s*}r_s^*\alpha=t_s^*f_*\alpha.
$$
\end{corollary}
\begin{proof}
The statement will be proven if we can show 
that for every geometric $y\in Y_s(F,\varphi)$,
$y^*f_{s*}r_s^*\alpha=y^*t_s^*f_*\alpha$, but this is readily obtained by applying \ref{BC}
to squares $(X_s,Y_s, X_y,y)$ and $(X,Y,X_y,y)$.
\end{proof}

\begin{definition}
Let $(X,T)$ be a difference scheme of finite total dimension over a finite difference field $(k_0,\varphi_0)$
and let $\alpha$, $\beta$ be two
constructible functions on $X$ (not necessarily associated with the same
Galois covering).
Let $k$ be a finite field such that $(\bar{k},\varphi_k)$ extends $(k_0,\varphi_0)$,
When $X(\bar{k},\varphi_k)\neq\emptyset$, we define a pairing
$$
(\alpha,\beta)_{X(\bar{k},\varphi_k)}=\frac{1}{|X(\bar{k},\varphi_k)|}\sum_{x\in X(\bar{k},\varphi_k)}
\alpha_{(\bar{k},\varphi_k)}(x)
\overline{\beta}_{(\bar{k},\varphi_k)}(x).
$$
When $X(\bar{k},\varphi_k)=\emptyset$, we stipulate the expression
to be $0$.

\end{definition}

\begin{proposition}[Adjointness of pullbacks and pushforwards]\label{adjpushpull}
Let $(S,\sigma)$ be a normal connected difference scheme of finite $\sigma$-type over $\Z$, and 
let $f:(X,\sigma)\to(Y,\sigma)$ be a morphism of finite transformal type of normal connected finite-dimensional difference
$(S,\sigma)$-schemes with geometrically transformally integral fibres.
Let $\alpha$ be a $(S,\sigma)$-constructible
function on $(X,\sigma)$ and let $\beta$ be a $(S,\sigma)$-constructible function on $(Y,\sigma)$. There is a constant $C>0$ and a localisation $(S',\sigma)$ of $(S,\sigma)$ such that for every closed $s\in S'$ and
every finite field $k$ with $(\bar{k},\varphi_k)$ extending $(\kk(s),\sigma^s)$,
$$
\left| (\alpha,f^*\beta)_{X_s(\bar{k},\varphi)}- (f_*\alpha,\beta)_{Y_s(\bar{k},\varphi)}
\right| < C|k|^{-1/2}.
$$
\end{proposition}

\begin{proof}
By the usual tricks with refinement, inflation and pullback, we may assume that 
we have a situation as in \ref{defpushfwd}, and that $\alpha$ is associated
with the covering $Z/X$ and $\beta$ is associated with $W/Y$ and we 
can continue using the notation of \ref{fig1}.

Fix $C>0$ and a localisation $S'$ of $S$ which depend only on the above data, 
which make all subsequent Hrushovski's estimates work. 
Choose an $s\in S'$.
We proceed in several steps and reductions.

\noindent\emph{Step 1.} Assuming momentarily that $Z_s$ is geometrically integral, 
 for any central $\alpha'$ associated with $Z_s/X_s$ and any
central $\beta'$ associated with $W_s/Y_s$,  
$$
\left| (\alpha',f_s^*\beta')_{X_s(\bar{k},\varphi)}- (f_{s*}\alpha',\beta')_{Y_s(\bar{k},\varphi)}
\right| < C|k|^{-1/2}.
$$
Indeed, using \ref{geomequid}, we have
\begin{multline*}
\sum_{x\in X_s(\bar{k},\varphi_k)}\alpha'(\varphi_{k,x})f_s^*\overline{\beta}'(\varphi_{k,x})
=\sum_{x\in X_s(\bar{k},\varphi_k)}\alpha'(\varphi_{k,x})\pi_s^*\overline{\beta}'(\varphi_{k,x})\\
=|X_s(\bar{k},\varphi_k)|\sum_{\sigma\in\Sigma}\alpha'(\sigma)\pi_s^*\overline{\beta}'(\sigma)
+O(|k|^{\dt(X_s)-1/2}).
\end{multline*}
Similarly,
\begin{multline*}
\sum_{y\in Y_s(\bar{k},\varphi_k)}f_{s*}\alpha'(\varphi_{k,y})\overline{\beta}'(\varphi_{k,y})
=\sum_{y\in Y_s(\bar{k},\varphi_k)}\pi_{s*}\alpha'(\varphi_{k,y})\overline{\beta}'(\varphi_{k,y})\\
=|Y_s(\bar{k},\varphi_k)|\sum_{\tau\in T}\pi_{s*}\alpha'(\tau)\overline{\beta}'(\tau)
+O(|k|^{\dt(Y_s)-1/2}).
\end{multline*}
The result follows from Frobenius reciprocity \ref{frobrecip} and Hrushovski's 
estimate \ref{udiLW}.

\noindent\emph{Step 2.} The conclusion of Step~1 holds for arbitrary $Z_s$.

Denoting topological components of $Z_s$ by $Z_i$, $i\in I$ and assuming
$k$ is large enough to yield some rational points, we obtain
\begin{multline*}
\frac{1}{|X_s(\bar{k},\varphi_k)|}\sum_{x\in X_s(\bar{k},\varphi_k)}\alpha'(\varphi_{k,x})f_s^*\overline{\beta}'(\varphi_{k,x})\\
=\frac{1}{|X_s(\bar{k},\varphi_k)||G|}\sum_{z\in Z_s(\bar{k},\varphi_k)}\alpha'(\varphi_{k,z})\overline{\beta}'(\varphi_{k,h_s(z)})\\
=\frac{1}{|X_s(\bar{k},\varphi_k)||G|}\sum_{i\in I}\sum_{z\in Z_i(\bar{k},\varphi_k)}\alpha'_i(\varphi_{k,z})\overline{\beta}'(\varphi_{k,h_i(z)})\\
=\frac{1}{|X_s(\bar{k},\varphi_k)||G|}\sum_{i\in I}|D^i|\sum_{x\in X(\bar{k},\varphi_k)}\alpha'_i(\varphi_{k,x})\overline{\beta}'(\varphi_{k,h_s(z)})\\
\stackrel{\text{Step 1}}{\approx}
\frac{1}{|Y_s(\bar{k},\varphi_k)||I|}\sum_{i\in I}\sum_{y\in Y_s(\bar{k},\varphi_k)}f_{s*}\alpha'_i(\varphi_{k,y})\overline{\beta}'(\varphi_{k,y})\\
=\frac{1}{|Y_s(\bar{k},\varphi_k)|}\sum_{y\in Y_s(\bar{k},\varphi_k)}f_{s*}\alpha'(\varphi_{k,y})\overline{\beta}'(\varphi_{k,y}),
\end{multline*}
where $\approx$ denotes  equality up to $O(|k|^{-1/2})$.

\noindent\emph{Step 3.} We have
\begin{equation*}
\begin{split}
(\alpha,f^*\beta)_{X_s(\bar{k},\varphi_k)} & \stackrel{\text{notation}}{=} 
(r_s^*\alpha,r_s^*f^*\beta)_{X_s(\bar{k},\varphi_k)}=
(r_s^*\alpha,f_s^*t_s^*\beta)_{X_s(\bar{k},\varphi_k)}\\
& \stackrel{\text{\hbox to 10ex{\hfil Step 2\hfil }}}{\approx} 
(f_{s*}r_s^*\alpha,t_s^*\beta)_{Y_s(\bar{k},\varphi_k)}
\stackrel{\text{BC}}{=}
(t_s^*f_*\alpha,t_s^*\beta)_{Y_s(\bar{k},\varphi_k)}\\
&\stackrel{\text{notation}}{=}
(f_*\alpha,\beta)_{Y_s(\bar{k},\varphi_k)},
\end{split}
\end{equation*}
where `notation' refers to the convention at the end of \ref{def:basic-constr},
and `BC' to \ref{unfiBC}.
\end{proof}

\subsection{The trace formula}

\begin{theorem}[The trace formula]\label{trace-formula}
With notation of \ref{fig1}, let 
$p:(X,\sigma)\to(S,\sigma)$ be a morphism of finite transformal type of normal connected difference schemes with geometrically transformally integral fibres of finite relative total dimension $d$, and suppose
$(S,\sigma)$ is of finite $\sigma$-type over $\Z$.
Let $\alpha$ be an $S$-constructible function on $Y$.
 There is a constant $C>0$ and a localisation $(S',\sigma)$
of $(S,\sigma)$ such that for every  
closed $s\in S'$ and every finite field $k$ with $(\bar{k},\varphi_k)$ extending
$(\kk(s),\sigma^s)$,
\[
\biggl|\sum_{x\in X_s(\bar{k},\varphi_k)}
\alpha(\varphi_{k,r_s(x)})-
p_*\alpha(\varphi_{k,s})\left|X_s(\bar{k},\varphi_k)\right|\biggr|
< C|k|^{d-1/2}.\]
\end{theorem}
\begin{proof}
The given sum is in fact $|X_s(\bar{k},\varphi_k)|$ multiplied
by the expression 
$$
(r_s^*\alpha,p_s^*1)_{X_s(\bar{k},\varphi_k)}\stackrel{\text{adj}}{\approx}
(p_{s*}r_s^*\alpha,1)_{\tilde{s}(\bar{k},\varphi_k)}\stackrel{\text{BC}}{=}
(s^*p_*\alpha,1)_{\tilde{s}(\bar{k},\varphi_k)}=p_*\alpha(\varphi_{k,s}).
$$
\end{proof}

In the next section we will need the following corollary, which conceptually states
that every conjugacy class in $\Sigma$ which is not explicitly banned, is achieved
as a local Frobenius substitution of some rational point over some field with
a high enough power of Frobenius. 

From a slightly different viewpoint, it states that the family of
fields with Frobenii (resp.\ a model of ACFA) possesses  a difference analogue
of the \emph{Frobenius property} of \cite{fried-jarden}.

\begin{corollary}\label{twisted-cebotarev}
Let $(Z,\Sigma)\to(X,\sigma)$ be a Galois covering over $(S,\sigma)$ with
group $(G,\tilde{\Sigma})$ such that $X\to S$ has geometrically transformally 
integral fibres,
and let $C$ be a conjugacy class in $\Sigma$. 
There exists a $\sigma$-localisation $S'$ of $S$ such that for every $s\in S'$, for every large enough 
finite field $k$ with $(\bar{k},\varphi_k)$ extending $(\kk(s),\sigma^s)$, 
if $C$ restricts to $\varphi_k$, there exists a point
 $x\in X_s(\bar{k},\varphi_k)$ with $\varphi_{k,x}=C$.
\end{corollary}
If $X$ is of finite relative total dimension over $S$, the  statement is immediate from 
\ref{trace-formula}. Otherwise, one can consider a sitation over a suitable
 sub-$(S,\sigma)$-scheme
of finite relative total dimension.

\subsection{Dirichlet density and Chebotarev}

\begin{definition}[Zeta and $L$-functions]
Let $(X,\sigma)$ be a geometrically integral difference scheme of finite transformal type over 
$(k,\varphi_0)$, where $k=\F_q$ is a finite field and $\varphi_0$ is a power of
Frobenius on $k$, $\varphi_0:u\mapsto u^{q_0}$, and $q_0$ divides $q$. 
We shall write $\varphi_n=\varphi_{k}^n\varphi_0$ %
for the appropriate power
of Frobenius on $\bar{k}$. 
Let $\alpha$ be a constructible function on $(X,\sigma)$.
\begin{enumerate}
\item The \emph{zeta function} of $(X,\sigma)$ over $(k,\varphi_0)$ is
defined as the formal power series
$$
Z((X,\sigma)/(k,\varphi_0),t)=\exp\left(\sum_{n\geq1}\frac{|X(\bar{k},\varphi_n)|}{n}t^n\right).
$$
\item The \emph{$L$-function} associated with $\alpha$ is defined as the formal
power series
$$
L((X,\sigma)/(k,\varphi_0),\alpha,t)=
\exp\left(\sum_{n\geq1}\frac{t^n}{n}\sum_{x\in X(\bar{k},\varphi_n)}\alpha_{(\bar{k},\varphi_n)}(x)\right).
$$
\end{enumerate}
\end{definition}

\begin{lemma}
Suppose $(X,\sigma)$ is of total dimension $d$ over $(k,\varphi_0)$. 
The zeta function can be written as
$$
Z(X,t)=\frac{E(t)}{(1-q^dt)^{\mu q_0^d}},
$$
where the function $E(t)$ is holomorphic for $|t|<\frac{1}{q^{d-1/2}}$.
\end{lemma}
\begin{proof}
Using Hrushovski's estimate \ref{udiLW}, there exists a constant $C>0$ and
numbers $\epsilon_n$ with $|\epsilon_n|<C$ so that
\begin{multline*}
\log Z(X,t)={\sum_n\frac{t^n}{n}\left[\mu(q^nq_0)^d+\epsilon_n(q^nq_0)^{d-1/2}\right]}
\\=-\mu q_0^d\log(1-q^dt)+q_0^{d-1/2}\sum_n\epsilon_n\frac{(tq^{d-1/2})^n}{n},
\end{multline*}
and it is a simple exercise to calculate the radius of convergence for the latter series.
\end{proof}

\begin{conjecture}[Weil rationality for difference zeta]
The zeta function of a finite-dimensional difference scheme $(X,\sigma)$
of finite $\sigma$-type over a finite difference field $(k,\varphi_0)$
is \emph{near-rational}, i.e., its logarithmic derivative is rational.
\end{conjecture}

\begin{definition}[Dirichlet density]
Let $\alpha$ be a constructible function on a difference scheme $(X,\sigma)$ 
of total dimension $d$ over
a finite $(k,\varphi_0)$.
We define the integral of $\alpha$ with respect to the \emph{Dirichlet density} $\delta$
to be the expression
$$
\int_X\alpha\, d\delta=\lim_{t\nearrow {q^{-d}}}\frac{\log L(X,\alpha,t)}{\log Z(X,t)}=
\lim_{t\nearrow {q^{-d}}}\frac{\sum_{n\geq1}\frac{t^n}{n}\sum_{x\in X(\bar{k},\varphi_n)}\alpha_{(\bar{k},\varphi_n)}(x)}
{\sum_{n\geq1}\frac{t^n}{n}|X(\bar{k},\varphi_n)|}.
$$
\end{definition}

\begin{lemma}\label{lem-dirichl}
Let $\alpha_n$ be a periodic sequence with period $m$. Let $|c_n|, |c'_n|<C$ for
all $n$, and let $0\leq \lambda<1$.
Then
$$
\lim_{t\nearrow 1}\frac{\sum_n\alpha_n\frac{t^n}{n}+\sum_n c_n\frac{(\lambda t)^n}{n}}
{\sum_n\frac{t^n}{n}+\sum_n c'_n\frac{(\lambda t)^n}{n}}
=
\lim_{t\nearrow 1}\frac{\sum_n\alpha_n\frac{t^n}{n}}
{\sum_n\frac{t^n}{n}}
=
\frac{1}{m}\sum_{i=1}^{m}\alpha_i.
$$
\end{lemma}
The proof is an exercise in calculus.

\begin{theorem}[Twisted Chebotarev]\label{chebotarev-density}
Let $(X,\sigma)$ be a geometrically transformally integral normal difference scheme of finite
transformal type
over a finite $(k,\varphi)$ and let $\alpha$ be a constructible function on $(X,\sigma)$
associated with an \'etale Galois covering $(Z,\Sigma)\to (X,\sigma)$ and a central function $\alpha:\Sigma\to\C$.
Then
$$
\int_X\alpha\, d\delta=\int_{\Sigma}\alpha\, d\mu_{\mathop{\rm Haar}},
$$
where the latter denotes the integral with respect to the normalised counting measure
on $\Sigma$. In other words, the local Frobenius elements $\varphi_{n,x}$
are \emph{equidistributed} on $X$ with respect to the Dirichlet density.
\end{theorem}
\begin{proof}
By additivity of the statement, we can reduce to the case of connected $Z$. Writing
$(S,\varphi_0)=\spec^{\varphi_0}(k)$ and $p:(X,\sigma)\to(S,\varphi_0)$ for the
structure morphism,
let us compute $p_*\alpha$. Denoting by
$(k_1,T)$ the algebraic closure of $(k,\varphi)$ inside $(\kk(Z),\Sigma)$,
$p_*\alpha$ is just $\pi_*\alpha$, where $\pi:\Sigma\to T$ is the difference
structure part of the morphism $(Z,\Sigma)\to\spec^T(k_1)$. 
Note that $T$ is just the coset $\Gal(k_1/k)\varphi_0$,
and thus the sequence $\pi_*\alpha(\varphi_n)=\pi_*\alpha(\varphi_k^n\varphi_0)$
is periodic with period of length $|T|$. 
Thus, 
\begin{equation*}
\begin{split}
\int_X\alpha\, d\delta & =
\lim_{t\nearrow q^{-d}}\frac{\sum_n\frac{t^n}{n}\sum_{x\in X(\bar{k},\varphi_n)}\alpha(\varphi_{n,x})}{\sum_n\frac{t^n}{n}|X(\bar{k},\varphi_n)|}\\
&
\stackrel{\ref{trace-formula}}{=}
\lim_{t\nearrow q^{-d}}\frac{\sum_n\frac{t^n}{n}\left[p_*\alpha(\varphi_{n,s})\mu(q^nq_0)^d+\epsilon_n(q^nq_0)^{d-1/2}\right]}
{\sum_n\frac{t^n}{n}\left[\mu(q^nq_0)^d+\epsilon'_n(q^nq_0)^{d-1/2}\right]}\\
& \stackrel{\ref{lem-dirichl}}{=}
\lim_{t\nearrow q^{-d}}\frac{\sum_n\frac{(tq^d)^n}{n}p_*\alpha(\varphi_{n,s})}
{\sum_n\frac{(tq^d)^n}{n}}
\stackrel{\ref{lem-dirichl}}{=}
\frac{1}{|T|}\sum_{\tau\in T}\pi_*\alpha(\tau)\\
&
\stackrel{\ref{defpushpull}}{=}
\frac{1}{|T|}\sum_{\tau\in T}\frac{1}{|\pi^{-1}(\tau)|}\sum_{\sigma\in\pi^{-1}(\tau)}\alpha(\sigma)=\frac{1}{|\Sigma|}\sum_{s\in\Sigma}\alpha(\sigma), \\
\end{split}
\end{equation*}
as required.
\end{proof}
Theorem~\ref{main-ch-th} is obtained by applying the above to the characteristic
function of the given conjugacy class.

\bibliographystyle{plain}

\end{document}